\documentclass[10pt,a4paper]{article}
\usepackage{lmodern} 
\usepackage[T1]{fontenc}
\usepackage{textcomp}
\usepackage[scr=boondoxo,frak=boondox,bb=boondox]{mathalfa}
\usepackage[english]{babel}
\usepackage{amssymb}
\usepackage{amsmath}
\usepackage{indentfirst}
\usepackage[dvips]{graphicx}
\usepackage{epsfig}

\usepackage[colorlinks,citecolor=blue,linkcolor=red]{hyperref}

\usepackage{color}

\usepackage{amsthm}
\usepackage{comment}
\usepackage{multirow}  
\usepackage{amsfonts}
\usepackage{mathtools} 
\usepackage{setspace} 

\newtheorem{cor}{Corollary}[section]

\newtheorem{lemma}[cor]{Lemma}
\newtheorem{teo}[cor]{Theorem}

\newtheorem{rem}[cor]{Remark}

\newtheorem{ass}[cor]{Assumption}

\newcommand{\R}{\mathbb{R}}
\newcommand{\N}{\mathbb{N}}

\numberwithin{equation}{section}

\begin{document}

\title{Flow of elastic networks: long-time existence result}

\author{{\sc Anna Dall'Acqua} \thanks{Institut f\"ur Analysis, Universit\"at Ulm, Germany,
\url{anna.dallacqua@uni-ulm.de}}
{\sc Chun-Chi Lin} \thanks{Department of Mathematics, National Taiwan Normal University, Taipei, 116 Taiwan, \url{chunlin@math.ntnu.edu.tw}} and
 {\sc Paola Pozzi}  \thanks{Fakult\"at f\"ur Mathematik, Universit\"at Duisburg-Essen, Germany, \url{paola.pozzi@uni-due.de}} 
}

\date{\today}
\maketitle

\begin{abstract}
We provide a long-time existence and sub-convergence  result for the elastic flow of a three network in $\R^{n}$ under some mild topological assumptions. The evolution is such that the sum of the elastic energies of the three curves plus their weighted lengths decrease in time. Natural boundary conditions are considered at the boundary of the curves and at the triple junction.
\end{abstract}

\\



\tableofcontents

\section{Introduction}
We consider the long-time evolution of an \emph{elastic} three networks in $\R^{n}$ ($n \geq 2$) as depicted in Figure~1, that is with three fixed boundary points $P_{1}$, $P_{2}$, $P_{3}$ and one moving triple junction. That is, a three-pointed curved star with a star-center that may move in time. 
\begin{figure}[h]
\setlength{\unitlength}{.5mm}
\begin{center}
\begin{picture}(120, 75)

\put(10,0){\circle*{2}}
\put(100,0){\circle*{2}}
\put(45,70){\circle*{2}}

\put(45,75){$P_3$}
\put(0,0){$P_1$}
\put(105,0){$P_2$}

\qbezier(10,0)(30,25)(45,30)
\qbezier(100,0)(75,20)(45,30)
\qbezier(45,70)(45,50)(45,30)
\end{picture}
\caption{The model}
\end{center}
\label{Fig1}
\end{figure}
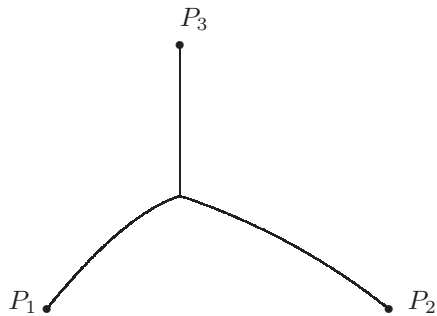

We study and give a long time existence result for elastic motion   with a penalization term on the length and some extra topological conditions that prevents the appearance of some pathological cases: for instance,  the triple junction should not be allowed to collapse on one of the boundary points $P_{i}$.

As far as we know, research on elastic networks is still at its beginning stage in mathematics. A lot of literature is focused on the case where the motion occurs by mean curvature flow (also called  curve shortening flow, a second order flow that decreases length of the curves), see for instance the survey paper \cite{MNPSsurvey}. A short time existence result for an elastic network of planar curves has been given recently in \cite{GMP} while the stationary case (in the special case of so-called \lq Theta\rq ~networks) has been considered in \cite{DNP}. Elastic flow with junctions is considered numerically in \cite{BGN12b}: in particular an appropriate variational formulation and two types of junction conditions, the so called $C^{0}$ respectively $C^{1}$ boundary conditions, are discussed. In that work the authors concentrated in the derivation of the Euler-Lagrange equations and did many numerical simulations.

Networks and flow of networks arise naturally in the study of multiphase systems and of the dynamics of their interfaces, see for instance \cite{LRV,G}. Elastic networks appear in some investigation in mechanical engineering or material sciences related to polymer gels, fiber or protein networks, e.g., \cite{BC07}, \cite{GD14}. 
In these physical systems, junctions between elastic beams play an important role in determining mechanical properties, e.g., rigidity or deformability.

\bigskip 

Before stating our main result, we introduce briefly the set up of  our work and recall some well known facts.

The elastic energy of a smooth regular curve (an immersion) $f: I \to \mathbb{R}^{n}$, $n \geq 2 $, $I=(0,1)$ is given by 
\begin{equation}\label{Elasten}
\mathcal{E}(f)=\frac{1}{2} \int_I |\vec{\kappa}|^2 ds,
\end{equation}
where $ds=|\partial_x f| dx$ is the arc-length element and $\vec{\kappa}$ is the curvature vector of the curve. Defining $\partial_s = |\partial_x f|^{-1} \partial_x f$, then $\vec{\kappa}=\partial_s^2 f$. The length is given by
\begin{equation*}
\mathcal{L}(f)=\int_I  ds.
\end{equation*}
For $\lambda \geq 0$ let
\begin{equation}\label{Elambda}
\mathcal{E}_{\lambda}(f)=\mathcal{E}(f) + \lambda \mathcal{L}(f) \, .
\end{equation} 
This is the energy that we consider:  the length of the curve is allowed to change in time but its growth is penalized according to the weight $\lambda$.

Now consider three  smooth regular curves $f_i:I \to \R^n$, $I=[0,1]$, $i=1,2,3$, such that
\begin{enumerate}
\item The end-points are fixed: 
\begin{equation}\label{IPs}
f_1(1)=P_1, \; f_2(1)=P_2,\;  f_3(1)=P_3,
\end{equation}
with given distinct points $P_i$, $i=1,2,3$, in $\R^n$ (recall Figure~\ref{Fig1}). Of course, there is a plane that contains these three points.
\item The curves start at the same point
$$ f_1(0)=f_2(0)=f_3(0) \, . $$
\end{enumerate}
In the following we call $\Gamma=\{f_1,f_2,f_3\}$ a \textit{three-pointed star network} or simply \textit{network}.

For $\lambda= (\lambda_1,\lambda_2, \lambda_3) \in \mathbb{R}_{+}^3$, the energy of the network $\Gamma=\{f_1,f_2,f_3\}$ is given by
\begin{equation}\label{energyNetwork}
\mathcal{E}_{\lambda}(\Gamma) = \sum_{i=1}^3 \mathcal{E}_{\lambda_i}(f_i) \, .
\end{equation}
Here and in the following we agree that $\mathcal{E}(\Gamma):=\mathcal{E}_{0}(\Gamma)$.

We let the network $\Gamma$ evolve in time according to an  $L^{2}$-gradient flow for the energy~$\mathcal{E}_{\lambda}$. Natural  boundary conditions are imposed on the three curves (more details are given in Section~\ref{sec:2}).
Our main result is the following:
\begin{teo}\label{mainthm}
Let $\Gamma_0=\{f_{1,0},f_{2,0}, f_{3,0}\}$ be a network of regular smooth curves in $\R^n$ such that:  
\begin{align} \nonumber
f_{1,0}(1) & =P_1, \; f_{2,0}(1)=P_2,\;  f_{3,0}(1)=P_3,\\  \label{initdatum1}
f_{1,0}(0) & =f_{2,0}(0)=f_{3,0}(0), \\ \nonumber
\vec{\kappa}_{i,0}(x) & = 0 \mbox{ for }x=0,1 \mbox{ and }i=1,2,3
\end{align}
as well as 
\begin{equation}\label{initdatum2}
\sum_{i=1}^3(\nabla_s \vec{\kappa}_{i,0} - \lambda_i \partial_s f_{i,0}) \Big|_{x=0}= 0 \, .
\end{equation}
Moreover let $\Gamma_{0}$ satisfy appropriate compatibility conditions,and be such that at the triple junction at least two curves form a strictly positive angle.
Then the following holds:

\noindent  \emph{\textbf{ (i) Long-time existence result:}} the equations
\begin{equation}\label{(P)}
 \partial_t f_i - \langle  \partial_t f_i, \partial_{s} f_{i} \rangle \partial_{s} f_{i} = - \nabla_{s}^2 \vec{\kappa}_i -\frac12 |\vec{\kappa}_i|^2 \vec{\kappa}_i  + \lambda_i \vec{\kappa}_i  \mbox{ on } (0,T) \times I \mbox{ for } i=1,2,3,
\end{equation} 
with boundary conditions
\begin{equation}\label{bc1}
\left\{\begin{aligned}
f_i(t,1)& =P_i, &\mbox{ for all }t\in (0,T), i=1,2,3, \\
\vec{\kappa}_i(t,1) & =0=\vec{\kappa}_i(t,0)  &\mbox{ for all }t\in (0,T), \ i=1,2,3,\\
f_1(t,0) & =f_2(t,0)=f_3(t,0) &\mbox{ for all }t\in (0,T), \\
\mbox{and }& \sum_{i=1}^3 (\nabla_s \vec{\kappa}_{i}(t,0) - \lambda_i \partial_s f_{i}(t,0)) = 0  &\mbox{ for all }t\in (0,T),
\end{aligned} \right.
\end{equation}
and initial value
$$ \Gamma(t=0)  = \{f_1(0,\cdot), f_2(0,\cdot), f_3(0,\cdot)\}= \Gamma_0, \mbox{ in } I$$
admit a smooth global solution in time, provided that, along the flow, the  lengths $\mathcal{L}(f_{i})$ of the three curves are uniformly bounded from below and that the dimension of the space spanned by the unit tangents $\partial_{s}f_{i}$, $i=1,2,3$,  at the triple junction is bigger or equal to two.

\noindent \emph{\textbf{(ii) Sub-convergence result:}} under the mentioned hypothesis, it is possible to find a sequence of time $t_{i} \to \infty$ such that the networks $\Gamma(t_{i})$ sub-converge, after an appropriate reparametrization, to a critical point for the energy $\mathcal{E}_{\lambda}(\Gamma)$  and subject to the boundary conditions given  in \eqref{initdatum1} and  \eqref{initdatum2}.
\end{teo}
The compatibility conditions are discussed in Appendix~\ref{sec:cc}. Smooth solution means that the three parametrization of the three curves are smooth functions in the time and space variable. The extension of the long-time existence result to the case $\lambda_i \geq 0$ is discussed in Remark \ref{rem:lambda0} below. 

Note that  
the above theorem must be understood in a \emph{geometrical sense}: that is the existence of a global parametrization of  the flow is meant up to reparametrization.
So our result  states that we are able to find a global in time \emph{smooth motion} of the network, provided two topological constraints are fulfilled during the flow: namely that  the lengths of the curves are uniformly bounded from 
below 
and that the the curves never entirely ``collapse'' to a configuration where all tangents vectors are parallel at the triple junctions.
The necessity of the topological constraints occurs naturally as follows: the bound from below on the lengths of the curves is needed to apply interpolation inequalities (cf. for instance Lemma~\ref{lemineqsh} and Lemma~\ref{lemineqshsum} below); that the dimension of the space spanned by the unit tangents at the triple junction should always be bigger or equal to two arise when we express the tangential  components at the boundary in terms of geometric quantities (cf. Remark~\ref{remjunctphi} and Remark~\ref{remjunctphi2} below.)
Not surprisingly it arises also in the proof of short-time existence of the flow given in \cite{GMP} (cf. \cite[Definition~3.2]{GMP}).
At the moment we have no means to control these topological constraints: whether and how this could be done is subject to future studies. 

To achieve our goal, we will consider in place of \eqref{(P)} equations of type
$$\partial_t f_i=  - \nabla_{s}^2 \vec{\kappa}_i -\frac12 |\vec{\kappa}_i|^2 \vec{\kappa}_i  + \lambda_i \vec{\kappa}_i   + \varphi_{i} \partial_{s} f_{i}, \mbox{ on } (0,T) \times I \mbox{ for } i=1,2,3,
$$
where $\varphi_{i}$ are smooth functions. Note that the presence of the tangential components  
 is necessary in order for the flow to fulfill the topological constraint that the curves stay ``glued'' at the triple junction (concurrency condition), with the latter being able to move freely in time.
 A proper choice of tangential component is necessary and is discussed in details in Section~\ref{ste}.

Our strategy can be summarized as follows: starting from a short-time existence result (see Section~\ref{ste}) we reparametrize the flow in such a way that for each curve the maps $\varphi_{i}$  linearly interpolate their values between the boundary points. For this choice of parametrizations we consider  the long-time behavior of the network, and show that if the flow does not exist globally then we obtain a contradiction. This is achieved by obtaining uniform bounds for the curvature and its derivatives, together with a control on the arc-length, up to the maximal time of existence $0<T <\infty$. With these estimates we  are able to extend  the flow  smoothly up to $T$ and then restart the flow,  contradicting the maximality of $T$.

In its essence our proof strategy is not different from our previous works on long-time existence for open elastic  curves in $\mathbb{R}^n$ (\cite{DLP}, \cite{DLP2}, \cite{DP}, \cite{DP-rims}): we use inequalities of Gagliardo-Nirenberg type, exploit the boundary conditions to reduce the order of some boundary terms, and rely heavily on interpolation estimates presented in \cite{DLP2}.
However, the treatment of the tangential components is completely new and far from trivial. In particular the  ``algebra'' for the maps $\varphi_{i}$  (that is how their derivatives in time and space behave with respect to the order of the studied PDEs, see Remark \ref{rem:varphi} for more details) must be thoroughly understood.
Furthermore, an accurate choice of the ``right'' vector field (specifically $\vec{\phi}$ in  Lemma~\ref{lempartint}) for which uniform bounds are derived is absolutely crucial for any of the presented arguments to work.
Finally, because of the interaction of the three curves proofs become increasingly technical and lengthy,  and several new lemmas are derived in order to  make our arguments more concise and more transparent.

The paper  is organized as follows: after introducing the notation and  motivating  the definition of the flow in Section~\ref{sec:2}, we collect several preliminary estimates and interpolation estimates in Section~\ref{sec:3}. The treatment of the boundary term is given in Section~\ref{sec:4}, whereas the influence of the tangential components at the boundary is discussed in Section~\ref{sec:5}. Finally, in Section~\ref{sec:6} we prove the main Theorem~\ref{mainthm}. The proof of the latter is divided in several steps: we have a initial step, where first bounds on the curvature vectors are derived. In the second step, we show the starting procedure of an induction argument: it is at this point that a proper choice of $\varphi_{i}$ starts playing an important  role. 
After the somewhat cumbersome induction step, where uniform estimates for the derivative of the curvature vectors are derived, we are finally able to conclude long-time existence by the contradiction procedure mentioned above.  To ease the presentation many technical proofs are collected in the Appendix. 

\bigskip

\noindent \textbf{Acknowledgements:} This project has been funded by the Deutsche Forschungsgemeinschaft (DFG, German Research Foundation)- Projektnummer: 404870139, 
and Ministry of Science and Technology, Taiwan (MoST 107-2923-M-003 -001 -MY3).

\section{Set up and notation}\label{sec:2}
\subsection{First variation and natural boundary conditions}

First of all recall, that for sufficiently smooth $\phi:I \to \R^n$  the first variation of the length is given by
\begin{equation}
\label{dlength}
\frac{d}{d\varepsilon} \mathcal{L}(f+\varepsilon \phi) \Big|_{\varepsilon=0}= \frac{d}{d\varepsilon}  \int_I   |\partial_x (f+\varepsilon \phi)| dx  \Big|_{\varepsilon=0}=   \langle \partial_s f , \phi \rangle \Big|_{\partial I} - \int_{I} \langle \vec{\kappa} ,  \phi \rangle \, ds \, ,
\end{equation} 
while the first variation of elastic energy \eqref{Elasten} (see \cite[Proof of Lemma A1]{DP}) is
\begin{align}\nonumber
\frac{d}{d\varepsilon} \mathcal{E}(f+\varepsilon \phi) \Big|_{\varepsilon=0} & = \frac{d}{d\varepsilon}  \int_I  |\vec{\kappa}_{f+\varepsilon \phi}|^2  |\partial_x (f+\varepsilon \phi)| dx  \Big|_{\varepsilon=0} \\
& =   \langle \partial_s \phi ,\vec{\kappa} \rangle \Big|_{\partial I}  -  \langle  \phi ,\nabla_s \vec{\kappa} +\frac12  |\vec{\kappa}|^2 \partial_s f \rangle \Big|_{\partial I} + \int_{I} \langle \nabla_s^2 \vec{\kappa} +\frac12 |\vec{\kappa}|^2 \vec{\kappa} , \phi \rangle \, ds  \, .
\label{delastic}
\end{align} 
Here $\nabla_s$ is an operator that on a smooth vector field $\phi$ acts as follows $\nabla_s \phi = \partial_s \phi - \langle \partial_s \phi, \partial_s f \rangle \partial_s f$, i.e. it is the normal projection of $\partial_s \phi$.

Let consider a variation $\Gamma_{t\eta}$ of the network $\Gamma$ given by $t \in (-\delta,\delta)$ and sufficiently smooth vector fields $\{\eta_1,\eta_2,\eta_3\}$, $\eta_i:I \to \R^n$, such that
$$ \eta_i(1)=0 \mbox{ for } i=1,2,3 \mbox{ and }\eta_i(0)=\eta_j(0) \mbox{ for } i,j=1,2,3 \, ,$$
so that $\Gamma_{t\eta}=\{f_1+t\eta_1,f_2+t\eta_2,f_3+t\eta_3\}$ is still a three-pointed star network.

Then by \eqref{dlength} and \eqref{delastic} we find
\begin{align*}
\left. \frac{d}{dt}\right|_{t=0} \mathcal{E}_{\lambda}(\Gamma_{t\eta}) & = \sum_{i=1}^3  \left(\langle \partial_s \eta_i ,\vec{\kappa}_i \rangle \Big|_{\partial I}  -  \langle  \eta_i ,\nabla_s \vec{\kappa}_i +\frac12  |\vec{\kappa}_i|^2 \partial_s f_i \rangle \Big|_{\partial I} + \lambda_i \langle \partial_s f_i , \eta_i \rangle \Big|_{\partial I} \right.\\
& \qquad \left.  + \int_{I} \langle \nabla_s^2 \vec{\kappa}_i +\frac12 |\vec{\kappa}_i|^2 \vec{\kappa}_i ,  \eta_i \rangle \, ds  - \lambda_i \int_{I} \langle \vec{\kappa}_i ,  \eta_i \rangle \, ds\right) \, .
\end{align*}
Notice that here and in the rest of the work for simplicity of notation we simply write $ds$ instead of $ds^i$ and also in the derivatives we simply write $ \nabla_s$ instead of the correct $\nabla_{s^i}$.

Choosing first test functions with compact support we see that each critical point has to satisfy
$$ \nabla_s^2 \vec{\kappa}_i +\frac12 |\vec{\kappa}_i|^2 \vec{\kappa}_i - \lambda_i \vec{\kappa}_i =0 \, , \mbox{ on }(0,1),$$
$i=1,2,3$. Moreover, at $x=1$
$$\sum_{i=1}^3 \langle \partial_s \eta_i ,\vec{\kappa}_i \rangle \Big|_{x=1} =0 \, ,$$
for any test function $\eta_i$ and hence that 
$$\vec{\kappa}_i (1)=0 \mbox{ for }  i=1,2,3.$$ 
Instead, at $x=0$ we find
$$ \sum_{i=1}^3  \left(\langle \partial_s \eta_i  ,\vec{\kappa}_i \rangle -  \langle  \eta_i ,\nabla_s \vec{\kappa}_i + (\frac12  |\vec{\kappa}_i|^2 - \lambda_i)\partial_s f_i \rangle \right) \Big|_{x=0}= 0 \, . $$
This implies
$$\vec{\kappa}_i (0)=0 \mbox{ for }   i=1,2,3 
\, ,$$ 
and together with $\eta_1(0)=\eta_2(0)=\eta_3(0)$ that
$$\sum_{i=1}^3 (\nabla_s \vec{\kappa}_i - \lambda_i \partial_s f_i ) \Big|_{x=0}= 0 \, .$$

\subsection{The flow }
\label{sec:VL}
Let $\Gamma_0=\{f_{1,0},f_{2,0}, f_{3,0}\}$ be a network of regular smooth curves as in the case under consideration, that is satisfying \eqref{initdatum1}, \eqref{initdatum2} and being such that at the triple junction at least two curves form a strictly positive angle.  
Moreover the initial network needs to satisfy a set of compatibility conditions. These are required to ensure that the solution of the parabolic problem is smooth up to the initial time $t=0$. Details are given in Appendix~\ref{sec:cc}. 

We take an $L^2$-flow for the energy $\mathcal{E}_{\lambda}$ with the condition that the network keeps its topological properties along the flow. For this one needs a tangential component. The problem we consider is then
\begin{equation}\label{flownetwork1}
\partial_t f_i  = - \nabla_{s}^2 \vec{\kappa}_i -\frac12 |\vec{\kappa}_i|^2 \vec{\kappa}_i  + \lambda_i \vec{\kappa}_i + \varphi_i \partial_s f_i \mbox{ on } (0,T) \times I \mbox{ for } i=1,2,3,
\end{equation}
with $\varphi_i$, $i=1,2,3$, smooth functions (whose role and definition is discussed below), with boundary conditions given in \eqref{bc1}
and initial value 
$$ \Gamma(t=0)  = \{f_1(0,\cdot), f_2(0,\cdot), f_3(0,\cdot)\}= \Gamma_0, \mbox{ in } I .$$


\subsection{Short time existence and reparametrization}\label{ste}
Our starting point  is a short time existence results, stating that given  an initial network $\Gamma_{0}$ of smooths curves satisfying \eqref{initdatum1}, \eqref{initdatum2} and suitable compatibility conditions (cf. Appendix~\ref{sec:cc}), then there exists an interval of time where our problem  \eqref{flownetwork1} admits a smooth regular solution, meaning that $f_{i} \in C^{\infty}([0,T) \times [0,1])$, $i=1,2,3$, are regular parametrizations. 
A proof of this result  in this form has not been given  yet and will be provided by the authors in future work. 
A short-time existence for \emph{planar} curves, in appropriate H\"older spaces can be found in \cite{GMP}.

Before we proceed some comments are in order, in particular more information must be given on the choice of the tangential components $\varphi_{i}$.
First of all, notice that to construct a short-time existence solution for \eqref{(P)} (together with the chosen initial and boundary conditions), one typically proceeds by 1) introducing a a suitable choice of tangential components $\varphi_{i}$  (in order to factor out the degeneracies due to the geometric invariances), 2) applying  a linearization procedure and Solonnikov-theory (see \cite{Sol}), 3) employing a fixed point argument to show short-time existence for the systems of non-linear equations under consideration (see \cite{Bronsard}). 
In particular we see that at a first sight the tangential components ``destroy the geometric nature'' of the equations \eqref{flownetwork1}. This is, however, not entirely true. It is well known, that tangential components can be modified by a reparametrization and that indeed all geometrical quantities (tangents, curvature vectors,  length, etc.) are invariant under reparametrization. That, for the geometric motion, $\varphi_{i}$ plays no role   in the interior of each curves  becomes evident also during the computations performed in this paper. The role of the tangential components $\varphi_{i}$ becomes tangible only at the boundary of each curve, when enforcing the concurrency condition at the triple junction  and influencing variation of the length of each curve (cf. also Remark~\ref{rem3.3} below). On the other hand even at the boundary the tangential components are determined by geometric quantities (see  Remark~\ref{remjunctphi}, in particular \eqref{phi1in0} below). 
So the ``freedom of choice'' in the tangential components is in principle only allowed in the interior of the curve, where, as we have already stated, the geometric quantities do not ``register'' it. Needles to say, we want to avoid tangential components, hence parametrizations, that destroy the regularity properties of the flow.

For the long-time existence proof it is important to have a good control of the tangential components also in the interior of the curves. To that end we reparametrize our short-time solution as follows: given $f_{i}$ and  $\varphi_{i}$ satisfying  \eqref{flownetwork1} on some time interval $[0,T)$, define
\begin{align}\label{evviva}
\tilde{\varphi}_{i}(t,x) =\varphi_{i}(t,0) \left(1 -\frac{1}{\mathcal{L}(f_{i}(t))}\int_{0}^{x} |\partial_{x}f_{i}(t,\xi)| d\xi \right),
\end{align}
that is $\tilde{\varphi}_{i}$ interpolates linearly the map $\varphi_{i}$ between $x=0$ and $x=1$ (where $\varphi_{i}(t,1)=0$, since at $x=1$ the velocities vanish due to the boundary conditions). 
Next choose a  family of smooth diffeomorphisms $\phi_{i}(t, \cdot): I \to I$, $i=1,2,3$, such that $\phi_{i}(t,x)=x$ for $(t,x) \in [0,T_{1}) \times  \partial I$ and
\begin{align*}
\partial_{t} \phi_{i}(t,x)& = \Big(\tilde{\varphi}_{i}(t, \phi_{i}(t,x)) -\varphi_{i} (t, \phi_{i}(t,x)) \Big) / |\partial_{x} f_{i} (t, \phi_{i}(t,x))|, \quad \text{ for } x \in I,\\
\phi_{i}(0,x)&=x,
\end{align*}
see \cite[Sec.1.3]{Mantegazza} and \cite[App.D]{Lee}.
Here $T_{1}$ is some positive time such that $0< T_{1} \leq T$.
Next, define 
$$ \tilde{f}_{i}(t,x)= f_{i}(t, \phi_{i}(t,x)),  \text{ for } (t,x) \in [0,T_{1}) \times  \bar{I}.$$
A straightforward computation gives then that
\begin{align*}
\partial_{t}\tilde{f}_{i} (t,x) & = \partial_{t} f_{i} (t, \phi_{i}(t,x)) + \partial_{x} f_{i} (t, \phi_{i}(t,x)) \partial_{t}\phi_{i}  (t, \phi_{i}(t,x))\\
&= [\partial_{t} f_{i}  - \langle \partial_{t} f_{i} , \partial_{s} f_{i} \rangle \partial_{s} f_{i} +
\tilde{\varphi}_{i}  \partial_{s} f_{i}](t, \phi_{i}(t,x)).
\end{align*}
In other words we can reparametrize the flow in such a way that the tangential component interpolates linearly its boundary values. At the same time problem \eqref{(P)} is satisfied on $[0, T_{1})$. 
As we will see in the proof of the long time existence, this will be of great help in many estimates. 
Also note that
\begin{align*}
\varphi_{i}(t,x)=\tilde{\varphi}_{i}(t,x),  \text{ for } (t,x) \in [0,T) \times \partial I,
\end{align*}
as well as 
\begin{align*}
\partial_{t }^{m}\varphi_{i}(t,x)=\partial_{t}^{m}\tilde{\varphi}_{i}(t,x),  \text{ for } (t,x) \in [0,T) \times \partial I, \text{ and } m \in \N
\end{align*}
so that $\tilde{f}_{i}$ fulfills all the same boundary conditions as $f_{i}$, $i=1,2,3$.
Summarizing, we can always assume without loss of generality that \eqref{flownetwork1} is fulfilled for some tangential components $\varphi_{i}$ for which \eqref{evviva} holds. This fact will be assumed henceforth.

\section{Preliminaries} \label{sec:3}
First of all we state a simple fact that will be used repeatedly  in the computations that follow: for $\vec{\phi}$ any smooth normal field along $f$ and $h$ a scalar map we have that for any $m \in \N$
\begin{align}\label{cs1}
&\nabla_{s}(h \vec{\phi}) = (\partial_{s } h ) \vec{\phi} + h\nabla_{s} \vec{\phi}, \qquad \qquad \nabla_{s}^{m}(h \vec{\phi}) = \sum_{r=0}^{m}\binom{m}{r} \partial_{s}^{m-r}h \nabla_{s}^{r} \vec{\phi}\\
\label{cs3}
&\nabla_{t}(h \vec{\phi}) = (\partial_{t } h ) \vec{\phi} + h\nabla_{t} \vec{\phi}, \qquad \qquad \nabla_{t}^{m}(h \vec{\phi}) = \sum_{r=0}^{m}\binom{m}{r} \partial_{t}^{m-r}h \nabla_{t}^{r} \vec{\phi}\\
\label{cs2}
&\nabla_{t}(h \, \partial_{s} f )= h \nabla_{t} (\partial_{s} f),
\end{align}
where $\nabla_t \phi = \partial_t \phi - \langle \partial_t \phi, \partial_s f \rangle \partial_s f$.

\begin{lemma}[\textbf{Evolution of geometric quantities}]\label{lemform} 
Let $f:[0, T)\times I \rightarrow \mathbb{R}^{n}$ be a smooth solution of $\partial_{t} f = \vec{V} + \varphi \tau$ for $t \in (0, T)$ with $\vec{V}$ the normal velocity. Given $\vec{\phi}$ any smooth normal field along $f$, the following formulas hold.
\allowdisplaybreaks{\begin{align}
\label{a}
\partial_{t}(ds)&=(\partial_{s} \varphi - \langle \vec{\kappa}, \vec{V} \rangle) ds \\
\label{b}
\partial_{t} \partial_{s}- \partial_{s}\partial_{t} &= (\langle \vec{\kappa}, \vec{V} \rangle -\partial_{s} \varphi) \partial_{s}\\
\label{c}
\partial_{t} \tau &= \nabla_{s} \vec{V} + \varphi \vec{\kappa}\\
\label{d}
  \partial_{t} \vec{\phi}&= \nabla_{t} \vec{\phi}- \langle \nabla_s \vec{V} + \varphi \vec{\kappa}, \vec{\phi} \rangle \tau \\ 
\label{e0}
\partial_{t} \vec{\kappa} & = \partial_{s} \nabla_{s} \vec{V} + \langle \vec{\kappa}, \vec{V}\rangle  \vec{\kappa} + \varphi \partial_{s} \vec{\kappa} \\
\label{e}
\nabla_{t} \vec{\kappa}& = \nabla_{s}^{2} \vec{V} + \langle \vec{\kappa}, \vec{V}\rangle  \vec{\kappa} + \varphi \nabla_{s} \vec{\kappa}\\
\label{f}
(\nabla_{t}\nabla_{s}- \nabla_{s}\nabla_{t}) \vec{\phi} &=(\langle \vec{\kappa},\vec{V} \rangle -\partial_{s }\varphi) \nabla_{s}\vec{\phi} +[\langle \vec{\kappa}, \vec{\phi}\rangle \nabla_{s}\vec{V} -\langle \nabla_{s} \vec{V}, \vec{\phi} \rangle \vec{\kappa}]. 
\end{align}} 
\end{lemma}
\begin{proof}
The proof follows by straightforward computation: see \cite[Lemma 2.1]{DP}.
\end{proof}

\paragraph{Decrease of the energy along the flow.}

As a first application of the above lemma we show that the energy decreases along the flow. 
Let $\Gamma(t)=\{f_1(t,\cdot),f_1(t,\cdot),f_1(t,\cdot)\}$ be a three-pointed star network moving according the elastic flow as considered in Section~\ref{sec:VL}. Then by  \eqref{Elambda}, \eqref{a} and \eqref{e} we find
\begin{align*}
\frac{d}{d t} \sum_{i=1}^3 \mathcal{E}_{\lambda}(f_i) & = \sum_{i=1}^3 \int_I \langle \vec{\kappa}_i, \nabla_{s}^{2} \vec{V}_i + \langle \vec{\kappa}_i, \vec{V}_i\rangle  \vec{\kappa}_i + \varphi_i \nabla_{s} \vec{\kappa}_i \rangle ds \\
&  \qquad + \sum_{i=1}^3 \int_I \big( \frac12  |\vec{\kappa}_i|^2 + \lambda_i  \big) ( \partial_s \varphi_i -  \langle \vec{\kappa}_i, \vec{V}_i\rangle ) \rangle ds \, .
\end{align*}
Integrating by parts and using that the curvature is zero at both boundary points
\allowdisplaybreaks{\begin{align*}
\frac{d}{d t} \sum_{i=1}^3 \mathcal{E}_{\lambda}(f_i) & = - \sum_{i=1}^3 \int_I \langle \nabla_{s} \vec{\kappa}_i, \nabla_{s} \vec{V}_i  \rangle ds + \sum_{i=1}^3 \int_I \langle \vec{\kappa}_i,  \langle \vec{\kappa}_i, \vec{V}_i\rangle  \vec{\kappa}_i + \varphi_i \nabla_{s} \vec{\kappa}_i \rangle ds\\
&  \qquad + \hspace{-.1cm}\sum_{i=1}^3 \lambda_i \varphi_i \big|_{\partial I} - \hspace{-.1cm} \sum_{i=1}^3 \int_I \varphi_i \langle \vec{\kappa}_i, \nabla_s\vec{\kappa}_i\rangle  ds  - \hspace{-.1cm}\sum_{i=1}^3 \int_I \big( \frac12  |\vec{\kappa}_i|^2 + \lambda_i  \big) \langle \vec{\kappa}_i, \vec{V}_i\rangle  ds \\
& = - \sum_{i=1}^3 \langle \nabla_{s} \vec{\kappa}_i,\vec{V}_i  \rangle \big|_{\partial I} + \sum_{i=1}^3 \int_I \langle \nabla_s^2\vec{\kappa}_i,  \vec{V}_i\rangle   ds \\
&  \qquad + \sum_{i=1}^3 \lambda_i \varphi_i \big|_{\partial I}  + \sum_{i=1}^3 \int_I \big( \frac12  |\vec{\kappa}_i|^2- \lambda_i  \big)  \langle \vec{\kappa}_i, \vec{V}_i\rangle  ds\\
& = - \sum_{i=1}^3 \langle \nabla_{s} \vec{\kappa}_i-\lambda_i \partial_s f_i,\vec{V}_i  + \varphi_i \partial_s f_i \rangle \big|_{\partial I} + \sum_{i=1}^3 \int_I \langle - \partial_t f_i + \varphi_i \partial_s f_i,  \vec{V}_i\rangle   ds  \\
& = - \sum_{i=1}^3 \langle \nabla_{s} \vec{\kappa}_i-\lambda_i \partial_s f_i,\partial_t f_i \rangle \big|_{\partial I} - \sum_{i=1}^3 \int_I |\partial_t f_i - \varphi_i \partial_s f_i|^2  ds\, ,
\end{align*}}
since $\vec{V}_i = \partial_t f_i- \varphi_i \partial_s f_i$. 

As $f(t,x=1)$ is fixed in time there is no contribution by the boundary terms at $x=1$. At $x=0$ one uses that $\partial_t f_i = \partial_t f_j$. Under natural boundary conditions at zero the boundary term also vanishes and we find that the energy is indeed decreasing.


\begin{lemma}[\textbf{Crucial lemma}]\label{lempartint}
Suppose $\partial_{t}f =\vec{V}+\varphi \tau$ on $(0,T) \times I$. Let $\vec{\phi}$ be a normal vector field along $f$ and $Y=\nabla_{t} \vec{\phi} + \nabla_{s}^{4} \vec{\phi}$.
Then
\begin{align}\label{eqgen0}
\frac{d}{dt} \frac{1}{2}\int_{I} |\vec{\phi}|^{2} ds + \int_{I}
|\nabla_{s}^{2} \vec{\phi}|^{2 } ds & = - [ \langle \vec{\phi}, \nabla_s^3 \vec{\phi} \rangle ]_0^1+ [ \langle \nabla_s \vec{\phi}, \nabla_s^2 \vec{\phi} \rangle ]_0^1
\\ \nonumber
& \qquad +\int_{I} \langle Y + \frac{1}{2} \vec{\phi} \,  \varphi_{s}, \vec{\phi} \rangle ds - \frac{1}{2} \int_{I} |\vec{\phi}|^{2} \langle \vec{\kappa}, \vec{V} \rangle ds ,
\end{align}
\end{lemma}
\begin{proof}
See \cite[Lemma 2.3]{DP} for a similar statement. The claim follows  using \eqref{a} and integration by parts.
\end{proof}

\begin{rem}\label{rem3.3}
It turns out that the influence of the tangential part $\varphi$ is dectectable only at the end-points of the considered curve (this is in accordance with the fact that geometric quantities are independent of parametrization).
Suppose  that the curve $f$ moving according to the evolution law $\partial_{t} f = \vec{V} + \varphi \tau$ is fixed at the end-point $x=1$.  We have that $\varphi(t,1)=0$ (here the velocity is zero!) for all $t$: then using \eqref{a} we observe that
\begin{align*}
-\varphi(t,0)= \varphi(t,1)- \varphi(t,0)= \int_{I} (|f_{x}|)_{t} dx + \int_{I} \langle \vec{\kappa}, \vec{V} \rangle ds.
\end{align*}
In other words at the moving point we infer
\begin{align}\label{reltangpart0}
\varphi(t,0)=- \frac{d}{dt} \mathcal{L}(f) -  \int_{I} \langle \vec{\kappa}, \vec{V} \rangle ds.
\end{align}
In particular  using the expression $V = - \nabla_{s}^2 \vec{\kappa} -\frac12 |\vec{\kappa}|^2 \vec{\kappa}  + \lambda\vec{\kappa}$, integration by parts and the boundary conditions $\vec{\kappa}(0)=\vec{\kappa}(1)=0$  we obtain
\begin{align}\label{reltangpart}
\frac{d}{dt} \mathcal{L}(f) +\varphi(t,0) + \int_{I}|\nabla_{s} \vec{\kappa}|^{2} ds+ \lambda \int_{I} |\vec{\kappa}|^{2}ds =
\frac{1}{2} \int_{I} |\vec{\kappa}|^{4} ds. 
\end{align}
It is not surprising that the length plays a role, since  $\varphi(t,0)$ determines how the curve grows or shrinks. 
\end{rem}

As in \cite[Lem.2.3]{DKS} and \cite[Sec.3]{DP} we denote by the product $ \vec{\phi}_1 *  \vec{\phi}_2 * \cdots * \vec{\phi}_k$ the product of $k$ normal vector fields $\vec{\phi}_i$ ($i=1,..,k$) defined as 
$\langle \vec{\phi}_1,\vec{\phi}_2 \rangle \cdot .. \cdot\langle \vec{\phi}_{k-2},\vec{\phi}_{k-1} \rangle \vec{\phi}_{k}$
if $k$ is odd and as 
$\langle \vec{\phi}_1,\vec{\phi}_2 \rangle \cdot .. \cdot \langle \vec{\phi}_{k-1},\vec{\phi}_{k} \rangle$, if $k$ is even. 
The expression $P_b^{a,c}(\vec{\kappa})$ stands for any linear combination of terms of the type
\[
 (\nabla_{s})^{i_1} \vec{\kappa}  * \cdots * (\nabla_{s})^{i_b} \vec{\kappa}  \text{ with }i_1 + \ldots + i_b = a \text{ and }\max i_j \leq c
\]
with universal, constant coefficients. Notice that $a$ gives the total number of derivatives, $b$ denotes the number of factors and $c$ gives a bound on the highest number of derivatives falling on one factor. 
With a slight abuse of notation, 
$|P^{a,c}_b(\vec{\phi})|$ 
denotes any linear combination with non-negative coefficients of terms of  type 
$$|\nabla_{s}^{i_1} \vec{\phi}|\cdot |\nabla_{s}^{i_2} \vec{\phi}| \cdot ... \cdot | \nabla_{s}^{i_{b}} \vec{\phi}| \mbox{ with }i_1 + \dots +i_{b}= a \mbox{ and } \max i_{j} \leq c \, .$$
Observe that for odd $b \in \mathbb{N}$ we have
$ \nabla_s P^{a,c}_b (\vec{\phi})= P^{a+1,c+1}_b (\vec{\phi}).$
For sums we write 
\begin{equation}
\sum_{\substack{[[a,b]] \leq [[A,B]]\\c\leq C}} P^{a,c}_{b} (\vec{\kappa})
:=\sum_{a=0}^{A}\sum_{b=1}^{2A+B-2a} \sum_{c=0}^C\text{ } P_{b}^{a,c}(\vec{\kappa})
\text{.}
\label{eq:sum_P^a_b}
\end{equation}
Similarly we set
$
\sum_{\substack{[[a,b]] \leq [[A,B]]\\c\leq C}}| P^{a,c}_{b} (\vec{\phi})| : = \sum_{a=0}^{A} \sum_{b=1}^{2A+B-2a} \sum_{c=0}^{C} | P^{a,c}_{b}(\vec{\phi}) |\, .
$
For our convenience and motivated by the interpolation inequalities below, we say that the \emph{order} of $\sum_{\substack{[[a,b]] \leq [[A,B]]\\c\leq C}} P^{a,c}_{b} (\vec{\kappa})$ 
is equal to $A+B/2$.

With this notation we can state the following results.
\begin{lemma}
\label{Qlemma}
We have the identities
\begin{align*}
& \partial_{s} \vec{\kappa}= \nabla_{s}\vec{\kappa} -|\vec{\kappa}|^{2}\tau,\\
& \partial_{s}^{m} \vec{\kappa} = \nabla_{s}^{m} \vec{\kappa} +\tau \sum_{\substack{[[a,b]] \leq [[m-1,2]]\\c \leq m-1, \ b \   even}} P^{a,c}_{b} 
(\vec{\kappa}) + \sum_{\substack{[[a,b]] \leq [[m-2,3]]\\c \leq m-2\ b \  odd}} P^{a,c}_{b} (\vec{\kappa})  \quad \text{ for } m\geq 2 \, .
\end{align*}
\end{lemma}
\begin{proof} The proof can be found for instance in \cite[Lemma~ 4.5]{DP}.
The first claim is obtained directly using that 
$$\partial_s \vec{\kappa}= \nabla_s \vec{\kappa} + \langle \partial_s \vec{\kappa}, \tau \rangle \tau  = \nabla_s \vec{\kappa} - |\vec{\kappa}|^2 \tau \, .$$ 
The second claim follows by induction.
\end{proof}

\begin{lemma}\label{evolcurvature}
Suppose $\partial_t f= -\nabla_s^2 \vec{\kappa} + \lambda \vec{\kappa}- \frac12 |\vec{\kappa}|^2 \vec{\kappa}+ \varphi \tau$, where $\lambda =\lambda(t)$. Then for $m  \in \mathbb{N}_0$ we have
\begin{align*} \nabla_t \nabla_s^m \vec{\kappa}+ \nabla_s^4 \nabla_s^m \vec{\kappa}&= P^{m+2,m+2}_3 (\vec{\kappa})+ \lambda (\nabla_s^{m+2} \vec{\kappa} + P^{m,m}_3 (\vec{\kappa})) +P^{m,m}_5 (\vec{\kappa}) + \varphi \nabla_s^{m+1} \vec{\kappa}.
\end{align*}
If $\lambda$ is a given fixed constant then we  simply write
\begin{align*}
\nabla_t \nabla_s^m \vec{\kappa}+ \nabla_s^4 \nabla_s^m \vec{\kappa}
& =  \varphi \nabla_s^{m+1} \vec{\kappa}  +\sum_{\substack{[[a,b]] \leq [m+2,3]]\\c\leq m+2, \\ b \,\, odd}}  P^{a,c}_{b} (\vec{\kappa})  \, .
\end{align*}
\end{lemma}
\begin{proof}
See \cite[Lemma~2.3]{DKS} and \cite[Lemma~2.3]{DLP} in the case that there is no tangential component.
For $m=0$ the claim follows directly from \eqref{e}. For $m=1$ we find using \eqref{f} and \eqref{e}
\allowdisplaybreaks{\begin{align*}
\nabla_t \nabla_s \vec{\kappa}+ \nabla_s^5 \vec{\kappa} & = \nabla_s \nabla_t  \vec{\kappa} + (\langle \vec{\kappa},\vec{V} \rangle -\partial_{s }\varphi) \nabla_{s}\vec{\kappa} +[\langle \vec{\kappa}, \vec{\kappa}\rangle \nabla_{s}\vec{V} -\langle \nabla_{s} \vec{V}, \vec{\kappa} \rangle \vec{\kappa}] + \nabla_s^5 \vec{\kappa}\\
& = \nabla_{s}^{3} \vec{V} + \langle \nabla_s\vec{\kappa}, \vec{V}\rangle  \vec{\kappa}  + \varphi \nabla_{s}^2 \vec{\kappa} + 2\langle \vec{\kappa}, \vec{V}\rangle \nabla_s \vec{\kappa}  +\langle \vec{\kappa}, \vec{\kappa}\rangle \nabla_{s}\vec{V} + \nabla_s^5 \vec{\kappa} \\
& = P^{3,3}_3 (\kappa)+ \lambda (\nabla_s^{3} \vec{\kappa} + P^{1,1}_3 (\vec{\kappa})) +P^{1,1}_5 (\vec{\kappa}) + \varphi \nabla_s^{2} \vec{\kappa} \,.
\end{align*}}
The general statement follows with an induction argument.
\end{proof}

Notice that no derivatives of the tangential component $\varphi$ appear.
In the following lemma we collect further important formulae. 

\begin{lemma}\label{lemLin}
Suppose $f: [0,T) \times \bar{I} \rightarrow \mathbb{R}^n$ is a smooth regular solution to
\begin{equation*}
\partial_{t} f= -\nabla_s^2 \vec{\kappa} -\frac12 |\vec{\kappa}|^2 \vec{\kappa} + \lambda \vec{\kappa}  + \varphi \partial_{s} f= \vec{V} + \varphi \partial_{s} f\,  
\end{equation*}
in $(0,T) \times I$. Here $\lambda \in \R$ is a given fixed constant. Then, the following formulae hold on $(0,T) \times I$. 
\begin{enumerate}
\item  For $\mu,d\in \mathbb{N}_0$, $\nu \in  \mathbb{N}$, $\nu $ odd we have
\begin{align} \label{E2sumsimple}
\nabla_{t} P^{\mu,d}_{\nu} (\vec{\kappa}) &  = \sum_{\substack{[[a,b]] \leq [[\mu+4,\nu]]\\c\leq 4+d\\b \in [\nu,\nu+4], odd }} P^{a,c}_{b} (\vec{\kappa}) 
 +\varphi\, P_{\nu}^{\mu+1, d+1} (\vec{\kappa})   \, .
\end{align}
\item For any $A, C \in \mathbb{N}_0$, $B, N, M \in \mathbb{N}$, $B$ odd,
\begin{align}\label{E2sum}
\nabla_{t} \sum_{\substack{[[a,b]]\leq [[A,B]]\\c\leq C\\b\in [N,M], odd}}P^{a,c}_{b} (\vec{\kappa}) &= \sum_{\substack{[[a,b]] \leq [[A +4,B]]\\c\leq C+4\\b \in [N,M+4], odd}} P^{a,c}_{b} (\vec{\kappa}) 
+  \varphi \sum_{\substack{[[a,b]] \leq [[A+1,B]]\\c\leq C+1\\b\in[N,M],odd}}  P_{b}^{a, c} (\vec{\kappa})  \,  .
\end{align}
\item  For $\mu,d\in \mathbb{N}_0$, $\nu \in  \mathbb{N}$, $\nu $ odd we have
\begin{align} \label{E3odd}
\partial_{t} P^{\mu,d}_{\nu} (\vec{\kappa}) &  = \sum_{\substack{[[a,b]] \leq [[\mu+4,\nu]]\\c\leq 4+d\\b \in [\nu,\nu+4], odd }} P^{a,c}_{b} (\vec{\kappa}) 
+\varphi\, P_{\nu}^{\mu+1, d+1} (\vec{\kappa})  \\
& \qquad 
+  (\partial_{s}f) \Big (\varphi \langle P^{\mu,d}_{\nu}(\vec{\kappa}), \vec{\kappa} \rangle +  \sum_{\substack{[[a,b]] \leq [[\mu+3,\nu+1]]\\c\leq \max \{ d, 3 \}\\b \in [\nu+1, \nu+3]\, \, even }} P^{a,c}_{b} (\vec{\kappa})  \,\Big)\, ,\nonumber
\end{align}
whereas for $\nu$ even we have
\begin{align} \label{E3even}
\partial_{t} P^{\mu,d}_{\nu} (\vec{\kappa}) &  = \sum_{\substack{[[a,b]] \leq [[\mu+4,\nu]]\\c\leq 4+d\\b \in [\nu,\nu+4], even }} P^{a,c}_{b} (\vec{\kappa}) 
+\varphi\, P_{\nu}^{\mu+1,  d+1 } (\vec{\kappa})  \,.
\end{align}


\item  For any $A, C \in \mathbb{N}_0$, $B, N, M \in \mathbb{N}$, $B$ odd, 
\begin{align} \label{E4sum-odd}
\partial_{t} &\sum_{\substack{[[a,b]]\leq [[A,B]]\\c\leq C\\b\in [N,M], odd}}P^{a,c}_{b} (\vec{\kappa}) = 
\sum_{\substack{[[a,b]] \leq [[A+4, B]]\\c\leq C+4\\b \in [N,M+4], odd }} P^{a,c}_{b} (\vec{\kappa}) 
+\varphi\, \sum_{\substack{[[a,b]] \leq [[A+1, B]]\\c\leq C+1\\b \in [N,M], odd }} P^{a,c}_{b} (\vec{\kappa})   \\
& \qquad 
+  (\partial_{s}f) \Big (\varphi \sum_{\substack{[[a,b]] \leq [[A, B+1]]\\c\leq C\\b \in [N+1,M+1], even}} P^{a,c}_{b} (\vec{\kappa}) +  \sum_{\substack{[[a,b]] \leq [[A+3,B+1]]\\c\leq \max \{ C, 3 \}\\b \in [N+1, M+3]\, \, even }} P^{a,c}_{b} (\vec{\kappa})  \,\Big)\, , \nonumber
\end{align} 
(Note that the term multiplying $\varphi$ in the tangential component vanishes if we know that $\vec{\kappa}=0$. This fact will be used repeatedly in some computations.)
whereas for $B$ even we have
\begin{align} \label{E4sum-even}
\partial_{t} \sum_{\substack{[[a,b]]\leq [[A,B]]\\c\leq C\\b\in [N,M], even}}P^{a,c}_{b} (\vec{\kappa}) &= 
\sum_{\substack{[[a,b]] \leq [[A+4, B]]\\c\leq C+4\\b \in [N,M+4], even }} P^{a,c}_{b} (\vec{\kappa}) 
+\varphi\, \sum_{\substack{[[a,b]] \leq [[A+1, B]]\\c\leq C+1\\b \in [N,M], even}} P^{a,c}_{b} (\vec{\kappa})  \, .
\end{align}

\end{enumerate}
\end{lemma}
\begin{proof}The proof  is obtained by a straight forward generalisation of \cite[Lemma~3.1]{DP}. We report all details for the sake of the reader in Appendix \ref{AppBLem3.6}.
\end{proof}
Finally we give some estimates that will be used repeatedly for boundary terms.
\begin{lemma}\label{lem:trickbdry}
We have that for any $x \in [0,1]$ there holds
\begin{align}\label{trickboundary}
| P^{a,c}_{b}(\vec{\kappa})(x)|^{2} \leq C ( \int_{I} |P^{2a+1,c+1}_{2b}(\vec{\kappa})| + |P^{2a,c}_{2b} (\vec{\kappa})| ds), \quad \text{ if $b$ is odd},\\ \label{trickboundary2}
| P^{a,c}_{b}(\vec{\kappa})(x)| \leq C ( \int_{I} |P^{a+1,c+1}_{b}(\vec{\kappa})| + |P^{a,c}_{b} (\vec{\kappa})| ds), \quad \text{ if $b$ is even},
\end{align}
where $C=C(\frac{1}{\mathcal{L}(f)})$.
\end{lemma}
The above lemma will be used in conjunction with interpolation estimates shown below: in particular \eqref{trickboundary} will be used when $b=1$.
\begin{proof}
Let $x \in [0,1]$.
We first start by showing out statement when $b$ is odd. To motivate our proof's strategy, observe that
 by embedding theory we know that for any normal vector field $\phi $ we have
\begin{align}
\| \phi \|_{L^{\infty}(I)} &\leq c(n) \| \partial_{s} \phi \|_{L^{1}(I)} + \frac{c(n)}{\mathcal{L}(f)}\|  \phi \|_{L^{1}(I)} \notag \\ \label{embed}
&\leq c(n) \| \nabla_{s} \phi \|_{L^{1}(I)} +  c(n)\| \langle \phi, \vec{\kappa} \rangle \|_{L^{1}(I)} + \frac{c(n)}{\mathcal{L}(f)}\|  \phi \|_{L^{1}(I)}. 
\end{align}
On the other hand, to apply interpolation inequalities later on, where the number of factors $b$ must be $b \geq 2$, it is better  (instead of applying Cauchy-Schwarz to the $L^{1}$-norms) to consider
\begin{align*}
\| \phi \|_{L^{\infty}(I)}^{2}= \| |\phi|^{2} \|_{L^{\infty}(I)} &\leq  C\| \partial_{s} |\phi|^{2} \|_{L^{1}(I)} + \frac{C}{\mathcal{L}(f)}\|  |\phi|^{2} \|_{L^{1}(I)} \\
&\leq C\int_{I} |\phi|(|\nabla_{s} \phi|  + |\phi| ) ds,
\end{align*}
using that $\partial_s |\phi|^2= 2 \langle \phi, \nabla_s \phi\rangle$. Then this gives \eqref{trickboundary}.
 When $b$ is even we use the inequality $\| \phi \|_{L^{\infty}(I)} \leq C \| \partial_{s} \phi \|_{L^{1}(I)} + \frac{C}{\mathcal{L}(f)}\|  \phi \|_{L^{1}(I)} $, which holds for any scalar map $\phi$. Choosing $\phi=P^{a,c}_{b} (\vec{\kappa}) (x)$ we obtain \eqref{trickboundary2}.
\end{proof}

\subsection{Interpolation inequalities}

Interpolation inequalities are crucial in the proof of long-time existence. Consider the scale invariant norms for $k \in \mathbb{N}_{0}$ and $p \in [1, \infty)$
\begin{equation*}
\| \vec{\kappa}\|_{k,p} := \sum_{i=0}^{k} \| \nabla_s^{i} \vec{\kappa}\|_{p} \quad \mbox{ with } \quad \| \nabla_s^{i} \vec{\kappa} \|_{p} := \mathcal{L}[f]^{i+1-1/p} \Big( \int_I |\nabla_s^{i} \vec{\kappa}|^p \, ds \Big)^{1/p} \, ,
\end{equation*}
(as in \cite{DKS}) and the usual $L^p$- norm $\| \nabla_s^{i} \vec{\kappa} \|_{L^p}^p :=  \int_I |\nabla_s^{i} \vec{\kappa}|^p \, ds $.  

Most of the following results (that we briefly state without proof) can be found in several papers (e.g. \cite{DKS,Lin,DP}). We give the precise reference to the paper where a complete proof can be found. These inequalities are satisfied by closed and open curves and allow also for the boundary points of the curve to move in time. One needs only a control from below on the length of the curve.

\begin{lemma}[Lemma~4.1 \cite{DP}]\label{leminter}
Let $f: I \rightarrow \mathbb{R}^n$ be a smooth regular curve. Then for all $k \in \mathbb{N}$, $p \geq 2$ and $0 \leq i<k$ we have
\begin{equation*}
\| \nabla_s^{i} \vec{\kappa}\|_{p} \leq C \|\vec{\kappa}\|_{2}^{1-\alpha} \|\vec{\kappa}\|_{k,2}^{\alpha} \, ,  
\end{equation*}
with $\alpha= (i+\frac12 -\frac{1}{p})/k$ and $C=C(n,k,p)$.
\end{lemma}

\begin{cor}[Corollary~4.2 \cite{DP}]\label{corinter}
Let $f: I \rightarrow \mathbb{R}^n$ be a smooth regular curve. Then for all $k \in \mathbb{N}$ we have
\begin{equation*}
\|\vec{\kappa}\|_{k,2} \leq C ( \|\nabla^{k}_{s}\vec{\kappa}\|_{2} + \|\vec{\kappa}\|_{2}) \, ,
\end{equation*}
with $C=C(n,k)$.
\end{cor}

\begin{lemma}[Lemma~3.4 \cite{DLP2}]\label{lemineqsh} 
Let $f: I \rightarrow \mathbb{R}^n$ be a smooth regular curve. For any $a,c,\ell\in \mathbb{N}_0$, $b \in \mathbb{N}$, $b\geq 2$, $c \leq \ell+2$ and $a < 2(\ell+2)$ we find
\begin{equation}\label{interineq1sh}
\int_I |P^{a,c}_{b} (\vec{\kappa})| \, ds \leq C \mathcal{L}[f]^{1-a-b} \|\vec{\kappa}\|_{2}^{b-\gamma} \| \vec{\kappa}\|_{\ell+2,2}^{\gamma} \, ,
\end{equation}
with $\gamma=(a+\frac12 b-1)/(\ell+2)$ and $C=C(n,\ell,a, b)$. Further if  $a+\frac12 b<2\ell+5$, then for any $\varepsilon>0$
\begin{align}\label{interineq1sh2}
 \int_I |P^{a,c}_{b}(\vec{\kappa})| \, ds & \leq \varepsilon \int_I |\nabla_s^{\ell+2} \vec{\kappa}|^2 \, ds +C \varepsilon^{-\frac{\gamma}{2-\gamma}}  (\|\vec{\kappa}\|^2_{L^{2}})^{\frac{b-\gamma}{2-\gamma}}
  + C  \mathcal{L}[f]^{1-a-\frac{b}{2}}  \|\vec{\kappa}\|_{L^{2}}^{b}  \, ,
\end{align}
with $C=C(n,\ell,a, b)$.
\end{lemma}

\begin{lemma}[Lemma~3.5 \cite{DLP2}]\label{lemineqshsum}
Let $f: I \rightarrow \mathbb{R}^n$ be a smooth regular curve and $\ell \in \mathbb{N}_{0}$. If $A,B\in \mathbb{N}$ with $B \geq 2$,  and $A+\frac12 B<2\ell+5$ then we have
\begin{align}\label{interineq1sum}
&\sum_{\substack{[[a,b]]\leq[[A,B]]\\c\leq \ell+2, \,  2\leq b}}  \int_I |P^{a,c}_{b} (\vec{\kappa})| \, ds 
\\ 
& \leq C \min\{1, \mathcal{L}([f])\}^{1-2A-B} \max\{1, \| \vec{\kappa}\|_{2}\}^{2A +B} \max\{1, \| \vec{\kappa}\|_{\ell+2,2}\}^{\overline{\gamma}}  \, , \nonumber
\end{align}
and for any $\varepsilon \in (0,1)$
\begin{align}\label{interineq2sum}
\sum_{\substack{[[a,b]]\leq[[A,B]]\\c\leq \ell+2, \,  2\leq b}}  \int_I |P^{a,c}_{b} (\vec{\kappa})| \, ds 
& \leq \varepsilon \int_I |\nabla_s^{\ell+2} \vec{\kappa}|^2 \, ds +C \varepsilon^{-\frac{\overline{\gamma}}{2-\overline{\gamma}}} \max\{1, \|\vec{\kappa}\|^2_{L^{2}}\}^{\frac{2A +B}{2-\overline{\gamma}}}\\
& \quad  + C \min\{1, \mathcal{L}[f]\}^{1-A-\frac{B}{2}} \max\{1, \|\vec{\kappa}\|_{L^{2}}\}^{2 A+ B}   \, , \nonumber
\end{align}
with $\overline{\gamma}= (A + \frac12 B-1)/(\ell+2)$ and $C=C(n,\ell,A,B)$ .
\end{lemma}

\section{Treatment of the boundary terms} \label{sec:4}

Similar to \cite[Lemma 2.4]{DLP} we see that at the fixed boundary points the derivatives of the curvature of any even order vanish.
\begin{lemma}\label{kgood}
Let $f$ be a smooth solution of $\partial_{t}f = -\nabla_s^2 \vec{\kappa} + \lambda \vec{\kappa}- \frac12 |\vec{\kappa}|^2 \vec{\kappa}+\varphi \tau$ on $(0,T) \times I$ with $\lambda=\lambda(t)$ subject to the boundary conditions $f(t,1)=P$ (for some $P \in \R^n$) and $\vec{\kappa}(t,1)=0$ for all $t \geq 0$. Then $\varphi(t,1) =0$ and 
$\nabla_s^{2l} \vec{\kappa}(t,1)=0$ for all $l \in \mathbb{N}_0$ and for all times $t \in (0,T)$.
\end{lemma}
\begin{proof} 
For $l=0,1$ the claim follows immediately from the boundary conditions and the fact that $f$ is a smooth solution. 
Notice that because of the boundary conditions $\varphi(t,1)=0$ for all $t$. 
The case $l=2$ is a consequence of \eqref{e}: indeed at $x=1$ we have
\begin{equation*}
0= \nabla_t \vec{\kappa} = \nabla_s^2 V = -\nabla_s^4 \vec{\kappa} -\frac12 \nabla_s^2 (|\vec{\kappa}|^2 \vec{\kappa}) + \lambda \nabla_s^2 \vec{\kappa} = - \nabla_s^4 \vec{\kappa} \,.
\end{equation*} 
The general statement follows from an induction argument. Indeed, assume that the claim is true up to $2n$, $n \in \mathbb{N}$. By the induction assumption it follows that
$$\nabla_t \nabla_s^{2n-2} \vec{\kappa}(t,1)=0 \mbox{ for all }t \in (0,T) \, . $$
Lemma \ref{evolcurvature} with $m=2n-2$ gives that at $x=1$
$$\nabla_s^{2n+2} \vec{\kappa}= P^{2n,2n}_3 (\vec{\kappa})+ \lambda (\nabla_s^{2n} \vec{\kappa} + P^{2n-2,2n-2}_3 (\vec{\kappa})) + P^{2n-2,2n-2}_5 (\vec{\kappa})+ \varphi \nabla_s^{2n-1} \vec{\kappa}   .$$
Since $\varphi(t,1)=0$ for all $t$ 
and in all the other terms on the right-hand side there is an odd number of factors and an even number of derivatives, we see that $\nabla_s^{2n+2} \vec{\kappa}(t,1)=0$  for all $t \in (0,T)$.
\end{proof}

We consider now the triple junction where the tangential component plays an important role.

\begin{lemma}\label{lemtriple}
Let $\Gamma=\{f_1,f_2,f_3\}$ be a smooth solution of \eqref{flownetwork1}  on $(0,T) \times I$ subject to the boundary conditions \eqref{bc1}, and assume $\lambda_i$ is constant for any $i \in\{1,2,3\}$. 
Then at $x=0$ (i.e. at the junction point) for any $t \in(0,T)$ we have
\begin{align}\label{bm4}
\nabla_s^4  \vec{\kappa}_i  & = \lambda_i \nabla_s^{2} \vec{\kappa}_i + \varphi_i \nabla_s \vec{\kappa}_i \, , \quad i=1,2,3,\\
\label{bm8}
\nabla_s^8  \vec{\kappa}_i  & =\sum_{\substack{[[a,b]]\leq [[6,3]]\\c\leq 6\\b \, \, odd}}P^{a,c}_{b} (\vec{\kappa}_{i})
+ \varphi_{i} \sum_{\substack{[[a,b]]\leq [[5,1]]\\c\leq 5\\b\,\, odd}}P^{a,c}_{b} (\vec{\kappa}_{i}) - \varphi_{i}^{2}\nabla_s^{2} \vec{\kappa}_i - (\partial_{t} \varphi_{i} ) \, \nabla_s \vec{\kappa}_i,
\end{align}
more generally we can write for $m \in \N$, $m \geq 2$
\begin{align}\label{b4m}
\nabla_{s}^{4m} \vec{\kappa}_i  &=\sum_{\substack{[[a,b]]\leq [[4m-2,3]]\\c\leq 4m-2\\b \, \, odd}}P^{a,c}_{b} (\vec{\kappa}_{i}) +
\sum_{\beta \in S_{4m}^{m-1}} \prod_{l=0}^{m-1} ( \varphi_{i}^{(l)})^{\beta_{l}}\sum_{\substack{[[a,b]]\leq [[4m-|\beta|,1]]\\c\leq 4m-|\beta|\\b \, \, odd}}P^{a,c}_{b} (\vec{\kappa}_{i}) \, , 
\end{align}
where 
$\varphi_{i}^{(l)} =\frac{\partial^{l}}{\partial t^l} \varphi_{i}$ 
and 
\begin{align}\label{DefBeta}
S_{i}^{l}:=\{ \beta=(\beta_{0}, \ldots, \beta_{l}) \in \N_{0}^{l+1} \, : \,0< |\beta|:=3\beta_{0}+ (3+4)\beta_{1} + \ldots +(3+4l) \beta_{l} < i  \}. 
\end{align}

Furthermore at $x=0$ (i.e. at the junction point) for any $t \in(0,T)$ we have
\begin{align}
\sum_{i=1}^3 \nabla_s^5  \vec{\kappa}_i  & = \sum_{i=1}^3 \left( P^{3,3}_3 (\vec{\kappa}_i) +2\lambda_i \nabla_s^{3} \vec{\kappa}_i - \lambda_i^2 \nabla_s \vec{\kappa}_i + \varphi_i \nabla_s^{2} \vec{\kappa}_i + (P_2^{4,3}(\vec{\kappa}_i)- \lambda_i  |\nabla_s \vec{\kappa}_i|^2)\partial_s f_i\right) \nonumber \\ \label{bm5}
& =\sum_{i=1}^3  \Big(  \sum_{\substack{[[a,b]]\leq [[3,3]]\\c\leq 3\\b \, \, odd}}P^{a,c}_{b} (\vec{\kappa}_{i}) + \varphi_i \nabla_s^{2} \vec{\kappa}_i + (\partial_s f_i) \sum_{\substack{[[a,b]]\leq [[4,2]]\\c\leq 3\\b \, \, even}}P^{a,c}_{b} (\vec{\kappa}_{i})
  \Big) \, , \\ 
 \sum_{i=1}^3 \nabla_s^9  \vec{\kappa}_i  &  = \sum_{i=1}^3 \Big(
 \sum_{\substack{[[a,b]]\leq [[7,3]]\\c\leq 7\\b \, \, odd}}P^{a,c}_{b} (\vec{\kappa}_{i}) 
 +\varphi_{i} \sum_{\substack{[[a,b]]\leq [[6,1]]\\c\leq 6\\b \, \, odd}}P^{a,c}_{b} (\vec{\kappa}_{i})  -\varphi_{i}^{2} \nabla_{s}^{3}\vec{\kappa}_{i} - \partial_{t}\varphi_{i}\nabla_{s}^{2}\vec{\kappa}_{i} \nonumber \\  \label{bm9}
 & \qquad 
 +  \partial_{s}f_{i} 
 \Big[    
 \sum_{\substack{[[a,b]]\leq [[8,2]]\\ c\leq 7 \\b \, \, even}} P^{a,c}_{b} (\vec{\kappa}_{i}) 
 +  \varphi_{i}\sum_{\substack{[[a,b]]\leq [[5,2]]\\c\leq 4\\b \, \, even}}     P^{a, c}_{b} (\vec{\kappa}_{i}) 
  \Big ]    \Big) ,
\end{align} 
more generally we can write for $m \in \N$, 
\begin{align}
  \label{b5+4m}
 \sum_{i=1}^3 &\nabla_s^{5+4m}  \vec{\kappa}_i   = \sum_{i=1}^3 \Big(
 \sum_{\substack{[[a,b]]\leq [[5+4m -2,3]]\\c\leq 5+4m-2\\b \, \, odd}}P^{a,c}_{b} (\vec{\kappa}_{i}) 
\\& \quad +\sum_{\beta \in S_{4+4m}^{m}} \prod_{l=0}^{m} ( \varphi_{i}^{(l)})^{\beta_{l}}\sum_{\substack{[[a,b]]\leq [[5+4m-|\beta|,1]]   \nonumber \\c\leq 5+4m-|\beta|\\b \, \, odd}}P^{a,c}_{b} (\vec{\kappa}_{i}) 
 \\
 & \quad +  \partial_{s}f_{i} \Big[    \sum_{\substack{[[a,b]]\leq [[4+4m,2]]\\c\leq  3+4m\\b \, \, even}} 
 P^{a,c}_{b} (\vec{\kappa}_{i}) 
 +  \sum_{\beta \in S_{ 4m}^{m-1}} \prod_{l=0}^{m-1} ( \varphi_{i}^{(l)})^{\beta_{l}}\sum_{\substack{[[a,b]]\leq [[4+4m-|\beta|,2]]\\c\leq 3+4m-|\beta|\\b \, \, even}}P^{a,c}_{b} (\vec{\kappa}_{i}) 
  \Big ]    \Big) .\nonumber
\end{align}

\end{lemma}
\begin{rem}\label{rem:varphi}
In this lemma we see the \lq algebra\rq ~of the tangential component. More precisely, looking at the sets $S_i^l$ as defined in \eqref{DefBeta} and, in particular, at the special definition of the length of the multiindex, one sees that a factor $\varphi_i$ takes the place of three derivatives of the curvature, while a factor $\partial_t ^{\ell} \varphi_i=\varphi_i^{(\ell)}$ takes the place as $3+4l$ derivatives of the curvature. The order of these terms is given in Lemma \ref{lem:ordervarphi} below.
\end{rem}
\begin{proof}
Since $\vec{\kappa}(t,0)=0$ for all $t\in (0,T)$ then \eqref{bm4} follows from \eqref{e} as done in Lemma~\ref{kgood} above.
Next, since \eqref{bm4} holds for any time, we can apply $\nabla_{t}$ to both sides of the equation. Application of Lemma~\ref{evolcurvature} and \eqref{cs3} give then \eqref{bm8}. An induction argument using Lemma~\ref{evolcurvature} and \eqref{E2sum} gives then \eqref{b4m}. Details are given in Appendix \ref{AppBLemtriple}.\\
From the other boundary condition at the triple junction we find that at $x=0$
$$ \partial_t \sum_{i=1}^3 (\nabla_s \vec{\kappa}_i - \lambda_i \partial_s f_i) = 0 \, . $$
Using \eqref{d}, Lemma \ref{evolcurvature}, \eqref{c} and the fact that $\vec{\kappa}_i(t,0)=0$ we get 
\begin{align*}
& \sum_{i=1}^3 (-\nabla_s^5 \vec{\kappa}_i + P_3^{3,3}(\vec{\kappa}_i)+ \lambda_i \nabla_s^3 \vec{\kappa}_i + \varphi_i \nabla_s^2 \vec{\kappa}_i + (P_2^{4,3}(\vec{\kappa}_i)- \lambda_i  |\nabla_s \vec{\kappa}_i|^2 - \partial_t \lambda_i)\partial_s f_i)  \\
& \qquad = \sum_{i=1}^3 \lambda_i (-\nabla_s^3 \vec{\kappa}_i + \lambda_i \nabla_s \vec{\kappa}_i)\, ,
\end{align*}  
from which the \eqref{bm5} 
follows. 
Differentiating in time \eqref{bm5}, using \eqref{d}, the fact that $\vec{\kappa}=0$, \eqref{E3odd}, \eqref{E3even}, Lemma~\ref{evolcurvature}, and \eqref{c} yield \eqref{bm9}. In a similar way we obtain by induction \eqref{b5+4m}. Details are given in the Appendix~\ref{AppBLemtriple}.
\end{proof}


\section{Treatment of the tangential component} \label{sec:5}
Here we study the tangential component at the junction point, the only point where the problem gives us information on the $\varphi_i$'s (see Section \ref{ste}). It is exactly here that we need the topological condition that the dimension of the space spanned by the unit tangents at the junction is at least two, see \eqref{dim2} below.

\begin{rem}\label{remjunctphi}
In order that the curves remain attached, it is necessary that 
$$ \partial_t f_i(t,0) = \partial_t f_j(t,0) \mbox{ for all }t \in (0,T) \mbox{ and }i,j=1,2,3. $$
This follows from differentiating with respect to $t$ the equality $f_i(t,0) =f_j(t,0)$.

The condition is also sufficient. Indeed, since the initial datum is attached we have for any $i,j =1,2,3$
\begin{align*}
f_i(t,0) & = f_i(0,0) + \int_0^t \partial_t f_i(u,0) \ du 
= f_j(0,0) + \int_0^t \partial_t f_j(u,0) \ du = f_j(t,0) \, .
\end{align*} 

This gives in particular a condition on the tangential part. Indeed, since the curvature is zero at $0$ it is necessary that at $0$ for $i,j=1,2,3$
$$ - \nabla_{s}^2 \vec{\kappa}_i  + \varphi_i \partial_s f_i = - \nabla_{s}^2 \vec{\kappa}_j  + \varphi_j  \partial_s f_j \, .$$
That is, at zero (i.e. at the triple junction)
$$ \varphi_i  =  - \langle \nabla_{s}^2 \vec{\kappa}_j , \partial_s f_i \rangle   + \varphi_j  \langle \partial_s f_j, \partial_s f_i \rangle = \langle \partial_t f_j, \partial_s f_i \rangle \, .$$
Let us elaborate on this a bit further. For the sake of notation 
denote $$A_{i}= A_{i}(t):=\nabla_{s}^2 \vec{\kappa}_i \Big|_{x=0} \qquad \text{ and }  \qquad T_{i}=T_{i}(t)= \partial_{s} f_{i}(t,0).$$ Also we write $\varphi_{i}$ meaning $\varphi_{i}(t,0)$.
Using the above identity yields that
\begin{align*}
\varphi_{i}& = -\langle A_{i+1}, T_{i} \rangle + \varphi_{i+1} \langle T_{i+1}, T_{i} \rangle,\\
\varphi_{i}& = -\langle A_{i+2}, T_{i} \rangle + \varphi_{i+2} \langle T_{i+2}, T_{i} \rangle,
\end{align*}
and after addition
\begin{align*}
2 \varphi_{i} - \varphi_{i+1} \langle T_{i+1}, T_{i} \rangle -\varphi_{i+2} \langle T_{i+2}, T_{i} \rangle = -\langle A_{i+1} + A_{i+2}, T_{i} \rangle
\end{align*}
for any $i=1,2,3$, where the subindex have to be understood modulo 3. This yields the system
\begin{align*}\left (
\begin{array}{ccc}
2 & - \langle T_{2}, T_{1} \rangle & - \langle T_{3}, T_{1} \rangle\\
- \langle T_{2}, T_{1} \rangle &2 &  - \langle T_{3}, T_{2} \rangle\\
- \langle T_{3}, T_{1} \rangle  &  - \langle T_{3}, T_{2}\rangle & 2 \\
\end{array}\right) 
\left ( 
\begin{array}{c}
\varphi_{1}\\ \varphi_{2}\\ \varphi_{3}
\end{array}
\right)
=\left(
\begin{array}{c}
-\langle A_{2} + A_{3}, T_{1} \rangle \\
-\langle A_{1} + A_{3}, T_{2} \rangle \\
-\langle A_{1} + A_{2}, T_{3} \rangle 
\end{array}
\right).
\end{align*} 
The above real and symmetric matrix  is positive definite (by Sylvester's criterion) if and only if  its determinant is strictly positive.
A straight forward  calculation gives that
\begin{align*}
\det &= 8 -2 (\langle T_{3}, T_{2} \rangle)^{2} - 2 (\langle T_{1}, T_{2} \rangle)^{2}-2 (\langle T_{3}, T_{1} \rangle)^{2} -
2 \langle T_{1}, T_{2} \rangle \langle T_{2}, T_{3} \rangle \langle T_{3}, T_{1} \rangle \\
&\geq 2 (1 - \langle T_{1}, T_{2} \rangle \langle T_{2}, T_{3} \rangle \langle T_{3}, T_{1} \rangle ) \geq 0.
\end{align*}
with equality if and only if $T_{1}=T_{2}=T_{3}$ or ($T_{i}=T_{i+1}$ and $T_{i+2}=-T_{i}$) for some $i=1,2,3$.
These degenerate situations are always excluded if we assume that
\begin{equation}\label{dim2}
dim (span \{ T_{1}, T_{2}, T_{3} \}) \geq 2.
\end{equation}
Since the inverse of the matrix is given by
\begin{align}\mathcal{J} & = \frac{1}{\det}  \left (
\begin{array}{ccc}
4- T_{23}^2&  2T_{12} + T_{13} T_{23}  & 2 T_{13}  + T_{12} T_{23} \\
 2 T_{12}+ T_{13} T_{23}  &4- T_{13}^2 &  T_{12}  T_{13} +2  T_{23} \\
 2 T_{13}  +  T_{12}  T_{23}   &  T_{12} T_{13} +2  T_{23}  & 4- T_{12}^2 \\
\end{array}\right)\label{eq:inverse} 
\end{align}
with $T_{ij}=\langle T_{i}, T_{j} \rangle$, $i,j=1,2,3$, we see for instance that
\begin{align}\nonumber
\varphi_1(0) & = -\frac{1}{\det} \left( (4- T_{23}^2)\langle A_{2} + A_{3}, T_{1} \rangle+ (2T_{12}  + T_{13}  T_{23} ) \langle A_{1} + A_{3}, T_{2} \rangle \right.\\
& \qquad \qquad \left.+(2 T_{13}  + T_{12}  T_{23} ) \langle A_{1} + A_{2}, T_{3} \rangle \right) \,, \label{phi1in0}
\end{align}
and similar formulas hold for $\varphi_2(0), \varphi_3(0)$.
\end{rem}

\begin{rem}\label{remjunctphi2}
What we have observed in Remark~\ref{remjunctphi} can be repeated for higher order conditions. This allow us to find formula for the derivatives with respect to time of the tangential components $\varphi_{i}$ at the triple junction.
More precisely if the flow is sufficiently smooth then we also have
\begin{equation}\label{Equality2}
\partial_t^{2} f_i(t,0) = \partial_t^{2} f_j(t,0) \mbox{ for all }t \in (0,T) \mbox{ and }i,j=1,2,3. 
\end{equation}
Now using \eqref{d}, \eqref{c} and the fact that $\vec{\kappa}_{i} =0$ at the junction we infer that
\begin{align}\label{Speed2}
\partial_t^{2} f_i(t,0) = -\tilde{A}_{i}(t) + \psi_{i}(t) \partial_{s} f^{i}
\end{align}
with normal component (cf. Lemma~\ref{evolcurvature} and use $\vec{\kappa}_{i} =0$)
\begin{align*}
\tilde{A}_{i}&=\tilde{A}_{i}(t) =(\nabla_{t} \nabla_{s}^{2} \vec{\kappa}_{i} - \varphi_{i} \nabla_{s} \vec{V}_{i})|_{{x=0}}\\
&= -\nabla_{s}^{6} \vec{\kappa}_{i} + P^{4,4}_{3} (\vec{\kappa}_{i}) + \lambda_{i} (\nabla_{s}^{4} \vec{\kappa}_{i} + P^{2,2}_{3} (\vec{\kappa}_{i})) + P^{2,2}_{5} (\vec{\kappa}_{i}) +  \varphi_{i} (2\nabla_{s}^{3} \vec{\kappa}_{i} -\lambda_{i} \nabla_{s} \vec{\kappa}_{i} )\\
& =     \sum_{\substack{[[a,b]] \leq [[6,1]]\\c\leq 6 \\ b \, odd}} P^{a,c}_{b} (\vec{\kappa}_{i}) + \varphi_{i}(t,0) \sum_{\substack{[[a,b]] \leq [[3,1]]\\c\leq 3 \\ b \, odd}} P^{a,c}_{b} (\vec{\kappa}_{i}) ,
\end{align*}
and
\begin{align}\label{psi2}
\psi_{i}= \partial_{t}\varphi_{i}(t,0) + \langle \nabla_{s} \vec{V}_{i}, \nabla_{s}^{2} \vec{\kappa}_{i} \rangle |_{{x=0}}.
\end{align}
For the sake of notation 
we write again  
 $$ T_{i}=T_{i}(t)= \partial_{s} f_{i}(t,0).$$ 
Using \eqref{Equality2} and \eqref{Speed2} yields that
\begin{align*}
\psi_{i}& = -\langle \tilde{A}_{i+1}, T_{i} \rangle + \psi_{i+1} \langle T_{i+1}, T_{i} \rangle,\\
\psi_{i}& = -\langle \tilde{A}_{i+2}, T_{i} \rangle + \psi_{i+2} \langle T_{i+2}, T_{i} \rangle,
\end{align*}
and after addition
\begin{align*}
2 \psi_{i} - \psi_{i+1} \langle T_{i+1}, T_{i} \rangle -\psi_{i+2} \langle T_{i+2}, T_{i} \rangle = -\langle \tilde{A}_{i+1} + \tilde{A}_{i+2}, T_{i} \rangle
\end{align*}
for any $i=1,2,3$, where the subindex have to be understood modulo 3. This yields the system
\begin{align*}\left (
\begin{array}{ccc}
2 & - \langle T_{2}, T_{1} \rangle & - \langle T_{3}, T_{1} \rangle\\
- \langle T_{2}, T_{1} \rangle &2 &  - \langle T_{3}, T_{2} \rangle\\
- \langle T_{3}, T_{1} \rangle  &  - \langle T_{3}, T_{2}\rangle & 2 \\
\end{array}\right) 
\left ( 
\begin{array}{c}
\psi_{1}\\ \psi_{2}\\ \psi_{3}
\end{array}
\right)
=\left(
\begin{array}{c}
-\langle \tilde{A}_{2} + \tilde{A}_{3}, T_{1} \rangle \\
-\langle \tilde{A}_{1} + \tilde{A}_{3}, T_{2} \rangle \\
-\langle \tilde{A}_{1} + \tilde{A}_{2}, T_{3} \rangle 
\end{array}
\right).
\end{align*} 
which we have already solved in Remark~\ref{remjunctphi}. Therefore we find again that the matrix is invertible if we assume \eqref{dim2}. 
By the expression for the inverse of the matrix given in \eqref{eq:inverse}, 
we see for instance from \eqref{psi2} that
\begin{align}\nonumber
\partial_{t}\varphi_1(0) & = - \langle \nabla_{s} \vec{V}_{i}, \nabla_{s}^{2} \vec{\kappa}_{i} \rangle |_{{x=0}} \\  \nonumber
& \quad -\frac{1}{\det} \left( (4- T_{23}^2)\langle \tilde{A}_{2} + \tilde{A}_{3}, T_{1} \rangle+ (2 T_{12}  + T_{13} T_{23} ) \langle \tilde{A}_{1} + \tilde{A}_{3}, T_{2} \rangle \right.\\
& \qquad \qquad \left.+(2 T_{13}  +  T_{12} T_{23} ) \langle \tilde{A}_{1} + \tilde{A}_{2}, T_{3} \rangle \right) \,, \label{phi1tin0}
\end{align}
and similar formulas hold for $\partial_{t}\varphi_2(0), \partial_{t}\varphi_3(0)$.

More generally an induction argument, that uses \eqref{c}, \eqref{d}, $\vec{\kappa}_{i}=0$ at the boundary, Lemma~\ref{evolcurvature}, \eqref{E4sum-odd}, and  \eqref{E4sum-even}, gives that for any $m \in \N$, $m \geq2$, at the boundary points we have
\begin{align}\label{Speedm}
\partial_{t}^{m}f_{i} =-\tilde{A}_{i}(t) + \psi_{i}(t) \partial_{s} f^{i}
\end{align}
where
\begin{align}\label{Atilde}
\tilde{A}_{i}(t)= \sum_{\substack{[[a,b]]\leq [[4m-2,1]]\\c\leq 4m-2\\b \, \, odd}}P^{a,c}_{b} (\vec{\kappa}_{i}) 
 +  \sum_{\beta \in S_{4m-2}^{m-2}} \prod_{l=0}^{m-2} ( \varphi_{i}^{(l)})^{\beta_{l}}\sum_{\substack{[[a,b]]\leq [[4m-2-|\beta|,	1]]\\c\leq 4m-2-|\beta|\\b \, \, odd}}P^{a,c}_{b} (\vec{\kappa}_{i}) 
\end{align}
and 
\begin{align*}
\psi_{i}(t)&=\partial_{t}^{m-1}\varphi_{i}(t,0)\\
& \quad +\sum_{\substack{[[a,b]]\leq [[4m-3,2]]\\c\leq 4m-3\\b \, \, even}}P^{a,c}_{b} (\vec{\kappa}_{i}) 
+ \sum_{\beta \in S_{4(m-1)-2}^{m-3}} \prod_{l=0}^{m-3} ( \varphi_{i}^{(l)})^{\beta_{l}}\sum_{\substack{[[a,b]]\leq [[4m-3-|\beta|,	2]]\\c\leq 4m-3-|\beta|\\b \, \, even}}P^{a,c}_{b} (\vec{\kappa}_{i}). 
\end{align*}
Using
\begin{equation}\label{Equalitym}
\partial_t^{m} f_i(t,0) = \partial_t^{m} f_j(t,0) \mbox{ for all }t \in (0,T) \mbox{ and }i,j=1,2,3, 
\end{equation}
and \eqref{Speedm} the same arguments as above give then for $m \in \N$, $m \geq2$
\begin{align}\nonumber
\partial_{t}^{m-1}\varphi_1(0) & = \Big(\sum_{\substack{[[a,b]]\leq [[4m-3,2]]\\c\leq 4m-3\\b \, \, even}} \hspace{-.6cm}P^{a,c}_{b} (\vec{\kappa}_{i}) 
+ \sum_{\beta \in S_{4m-6}^{m-3}} \prod_{l=0}^{m-3} ( \varphi_{i}^{(l)})^{\beta_{l}}\sum_{\substack{[[a,b]]\leq [[4m-3-|\beta|,	2]]\\c\leq 4m-3-|\beta|\\b \, \, even}}\hspace{-.7cm}P^{a,c}_{b} (\vec{\kappa}_{i}) \Big) \Big|_{{x=0}} \\  \nonumber
& \quad -\frac{1}{\det} \left( (4- T_{23}^2)\langle \tilde{A}_{2} + \tilde{A}_{3}, T_{1} \rangle+ (2T_{12}  + T_{13} T_{23} ) \langle \tilde{A}_{1} + \tilde{A}_{3}, T_{2} \rangle \right.\\
& \qquad \qquad \left.+(2 T_{13} + T_{12}  T_{23} ) \langle \tilde{A}_{1} + \tilde{A}_{2}, T_{3} \rangle \right) \,, \label{phi1tmin0}
\end{align}
and similar formulas hold for $\partial_{t}^{m-1}\varphi_2(0), \partial_{t}^{m-1}\varphi_3(0)$.
\end{rem}

Next we give estimates for the tangential components $\varphi_{i}$ ( and their time derivatives) at the triple junction.
To that end we will use repeatedly Lemma~\ref{lem:trickbdry}.

\begin{lemma}\label{lem:ordervarphi}
Let $\Gamma=\{f_1,f_2,f_3\}$ be a smooth solution of \eqref{flownetwork1}  on $(0,T) \times I$ subject to the boundary conditions \eqref{bc1}. Furthermore 
let Assumption~\ref{assnetwork} (see below) hold and assume that there exists a constant $C>0$ such that
\begin{align}\label{++} \mathcal{L}(f_{i}(t)) \geq C \mbox{ and }\| \vec{\kappa}_{i} (t)\|_{2} \leq C \mbox{ for }i=1,2,3 ,\end{align}
and any $t \in (0,T).$
 Then we have for any $\ell \in \N$ and $i=1,2,3$ that
\begin{align}\label{gio1zero}
|\varphi_i(0)| &\leq \frac{C}{\delta} \sum_{j=1}^3 \mathcal{L}[f_j]^{-3} \| \vec{\kappa}_j\|_{2}^{\frac{2\ell -1}{2(\ell+2)}}  \| \vec{\kappa}_j\|_{\ell+2,2}^{\frac{5}{2(\ell +2)}} \leq \frac{C}{\delta} \sum_{j=1}^3  \| \vec{\kappa}_j\|_{\ell+2,2}^{\frac{5}{2(\ell +2)}} \,,\\
\label{gio1Tzero}
|\partial_{t}\varphi_{i}(0)| & \leq \frac{C}{\delta}\sum_{j=1}^{3} \max
\{ 1, \| \vec{\kappa}_{j}\|_{\ell+2,2} \}^{\frac{13}{2(\ell+2)}} \qquad  \mbox{ for }\ell \geq 5.
\end{align}
More generally, we have for any $\ell \in \N$ and $i=1,2,3$ that
\begin{align}\label{giomzero}
|\partial_{t}^{m}\varphi_i(0)| & \leq C(m,\frac{1}{\delta}) \sum_{j=1}^3  \max
\{ 1, \| \vec{\kappa}_j\|_{\ell+2,2} \}^{\frac{5+8m}{2(\ell +2)}} \qquad  \mbox{ for }\ell \geq 4m+1.
\end{align}
\end{lemma}
\begin{proof}
An expression for $\varphi_{i}(t,0)$ is given in Remark~\ref{remjunctphi} 
(see \eqref{phi1in0} for $\varphi_1$).  Again we write here $\varphi_i(0)$ meaning $\varphi_{i}(t,0)$ for $t \in (0,T)$.
 Using Assumption \ref{assnetwork} we find with Lemma \ref{kgood} and Lemma~\ref{lemineqsh} (for any  $\ell \geq 1$)
\begin{align*}\nonumber
|\varphi_i(0)| & \leq \frac{C}{\delta} \sum_{j=1}^3|\nabla_s^2 \vec{\kappa}_j(0)|=  \frac{C}{\delta} \sum_{j=1}^3 \left( |\nabla_s^2 \vec{\kappa}_j(0)|^2-|\nabla_s^2 \vec{\kappa}_j(1)|^2\right)^{\frac12}\\ 
& \leq \frac{C}{\delta} \sum_{j=1}^3 \left( -2 \int_0^1 \langle \nabla_s^3 \vec{\kappa}_j,\nabla_s^2 \vec{\kappa}_j\rangle \; ds \right)^{\frac12} \leq \frac{C}{\delta} \sum_{j=1}^3 \mathcal{L}[f_j]^{-3} \| \vec{\kappa}_j\|_{2}^{\frac{2\ell -1}{2(\ell+2)}}  \| \vec{\kappa}_j\|_{\ell+2,2}^{\frac{5}{2(\ell +2)}} \,.
\end{align*} 
An expression for $\partial_{t}\varphi_{i}(t,0)$ is given in Remark~\ref{remjunctphi2} 
(see for instance \eqref{phi1tin0} for $\partial_{t}\varphi_1$). 
Again using Assumption \ref{assnetwork} we find
\begin{align*}
|\partial_{t} \varphi_{i}(0)| & \leq |\langle \nabla_{s} \vec{V}_{i} ,\nabla_{s}^{2} \vec{\kappa}_{i} \rangle|_{{x=0}}| \\
& \quad + \frac{C}{\delta}\sum_{j=1}^3
\Big|   \sum_{\substack{[[a,b]] \leq [[6,1]]\\c\leq 6 \\ b \, odd}} P^{a,c}_{b} (\vec{\kappa}_{j})  \Big|_{x=0} + \varphi_{j}(t,0) \sum_{\substack{[[a,b]] \leq [[3,1]]\\c\leq 3 \\ b \, odd}} P^{a,c}_{b} (\vec{\kappa}_{j})  \Big|_{x=0} \Big|
\end{align*}
By \eqref{trickboundary}, Lemma~\ref{lemineqshsum}  and the uniform bound on length and the $L^2$-norm of the curvature we infer for any $\ell \geq 5$
\begin{align*}
\sum_{\substack{[[a,b]] \leq [[6,1]]\\c\leq 6 \\ b \, odd}} \hspace{-.5cm} |P^{a,c}_{b} (\vec{\kappa}_{j})|(0)  \leq  C \Big( \sum_{\substack{[[a,b]] \leq [[13,2]]\\c\leq 7 \\ b \geq 2}} \int_{I} |P^{a,c}_{b} (\vec{\kappa}_{j}) | ds \Big)^{\frac12} \leq C
\max
\{ 1, \| \vec{\kappa}_{j}\|_{\ell+2,2} \}^{\frac{13}{2(\ell+2)}}, 
\end{align*}
and
\begin{align*}
\sum_{\substack{[[a,b]] \leq [[3,1]]\\c\leq 3 \\ b \, odd}} \hspace{-.5cm} | P^{a,c}_{b} (\vec{\kappa}_{j}) | (0) \leq 
C \big( \sum_{\substack{[[a,b]] \leq [[7,2]]\\c\leq 4 \\ b \geq 2}} \int_{I} |P^{a,c}_{b} (\vec{\kappa}_{j}) | ds \big)^{\frac12} \leq C
\max
\{ 1, \| \vec{\kappa}_{j}\|_{\ell+2,2} \}^{\frac{7}{2(\ell+2)}}. 
\end{align*}
Also, using once again that $\nabla_{s}^{2} \vec{\kappa}_{i} (1)=0$,  by Lemma~\ref{lemineqshsum}
\begin{align*}
|\langle \nabla_{s} \vec{V}_{i} ,\nabla_{s}^{2} \vec{\kappa}_{i} \rangle|_{{x=0}} & = \sum_{\substack{[[a,b]] \leq [[5,2]]\\c\leq 3 , b \geq2, \, b \,\,even}}| P^{a,c}_{b} (\vec{\kappa}_{i}) (0)-  P^{a,c}_{b} (\vec{\kappa}_{i}) (1)| \\
& \leq \sum_{\substack{[[a,b]] \leq [6,2]]\\c\leq 4,  b \geq 2}} \int_{I} |P^{a,c}_{b} (\vec{\kappa}_{i}) | ds  \leq C
\max
\{ 1, \| \vec{\kappa}_{i}\|_{\ell+2,2} \}^{\frac{6}{\ell+2}}.
\end{align*}
Therefore we get
\begin{align*}
|\partial_{t} \varphi_{i}(0)| & \leq C \max
\{ 1, \| \vec{\kappa}_{i}\|_{\ell+2,2} \}^{\frac{6}{\ell+2}} + \frac{C}{\delta}\sum_{j=1}^{3} \max
\{ 1, \| \vec{\kappa}_{j}\|_{\ell+2,2} \}^{\frac{13}{2(\ell+2)}} \\
& +\frac{C}{\delta} \sum_{j=1}^{3} \max
\{ 1, \| \vec{\kappa}_{j}\|_{\ell+2,2} \}^{\frac{7}{2(\ell+2)}}  |\varphi_{j}(0)|
\end{align*}
Using \eqref{gio1zero}  we finally infer
\begin{align*}
|\partial_{t} \varphi_{i}(0)| & \leq C \max
\{ 1, \| \vec{\kappa}_{i}\|_{\ell+2,2} \}^{\frac{12}{2(\ell+2)}} +\frac{C}{\delta} \sum_{j=1}^{3} \max
\{ 1, \| \vec{\kappa}_{j}\|_{\ell+2,2} \}^{\frac{13}{2(\ell+2)}}\\
& \qquad + \frac{C}{\delta} \sum_{j=1}^{3} \max
\{ 1, \| \vec{\kappa}_{j}\|_{\ell+2,2} \}^{\frac{7}{2(\ell+2)}}  \sum_{r=1}^{3}\| \vec{\kappa}_{r}\|_{\ell+2,2}^{\frac{5}{2(\ell+2)}}\\
& \leq \frac{C}{\delta}\sum_{j=1}^{3} \max
\{ 1, \| \vec{\kappa}_{j}\|_{\ell+2,2} \}^{\frac{13}{2(\ell+2)}},
\end{align*}
where we have used $\|\vec{\kappa}_{r}\|_{\ell+2,2}^{\frac{5}{2(\ell+2)}} \leq \max \{1, \|\vec{\kappa}_{r}\|_{\ell+2,2} \}^{\frac{5}{2(\ell+2)}} \leq
\max \{1, \|\vec{\kappa}_{r}\|_{\ell+2,2} \}^{\frac{6}{2(\ell+2)}} $ and Young inequality $|a||b| \leq C (|a|^{p} + |b|^{q})$ with  $p=13/7$ and $q=13/6$ for the product of terms of type 
$a=\{ 1, \| \vec{\kappa}_{j}\|_{\ell+2,2} \}^{\frac{7}{2(\ell+2)}}$ and $b=\max \{1, \|\vec{\kappa}_{r}\|_{\ell+2,2} \}^{\frac{6}{2(\ell+2)}}$.
The general statement \eqref{giomzero} follows by an induction argument, that uses \eqref{phi1tmin0}, \eqref{Atilde}, Lemma~\ref{lem:trickbdry}, Lemma~\ref{lemineqshsum}, and Lemma~\ref{PQgeneral}. Let us look at the single terms. First we estimate the terms appearing in \eqref{Atilde} (where $m$ is replaced by $m+1$).
Using \eqref{trickboundary}, \eqref{++}, and Lemma~\ref{lemineqshsum} we have for $\ell \geq 4m+1$
\begin{align*}
\sum_{\substack{[[a,b]]\leq [[4m+2,1]]\\c\leq 4m+2\\b \, \, odd}} |P^{a,c}_{b} (\vec{\kappa}_{j})|(0) &\leq  C (
\sum_{\substack{[[a,b]]\leq [[8m+5,2]]\\c\leq 4m+3\\b \, \, even}}  \int_{I}|P^{a,c}_{b} (\vec{\kappa}_{j})| ds )^{1/2} \\
& \leq C\max \{1, \|\vec{\kappa}_{j}\|_{\ell+2,2} \}^{\frac{5+8m}{2(\ell+2)}}
\end{align*}
and similarly together with the  induction assumptions
\begin{align*}
 \Big| &\sum_{\beta \in S_{4m+2}^{m-1}} \prod_{l=0}^{m-1} ( \varphi_{j}^{(l)})^{\beta_{l}}\sum_{\substack{[[a,b]]\leq [[4m+2-|\beta|,	1]]\\c\leq 4m+2-|\beta|\\b \, \, odd}}P^{a,c}_{b} (\vec{\kappa}_{j}) \, \Big|(0) \\
 & \leq  C \sum_{\beta \in S_{4m+2}^{m-1}} \prod_{l=0}^{m-1} | \varphi_{j}^{(l)}|^{\beta_{l}}  \big(
\sum_{\substack{[[a,b]]\leq [[8m+5-2|\beta|,2]]\\c\leq 4m+3-|\beta|\\b \, \, even}} \int_{I}|P^{a,c}_{b} (\vec{\kappa}_{j})| ds \big)^{1/2} \\
& \leq C\sum_{\beta \in S_{4m+2}^{m-1}} \prod_{l=0}^{m-1} | \varphi_{j}^{(l)}|^{\beta_{l}}  \max \{1, \|\vec{\kappa}_{j}\|_{\ell+2,2} \}^{\frac{5+8m-2|\beta|}{2(\ell+2)}}\\
& \leq C \sum_{\beta \in S_{4m+2}^{m-1}} \prod_{l=0}^{m-1} (C(l, \frac{1}{\delta}) \sum_{k=1}^3  \max
\{ 1, \| \vec{\kappa}_k\|_{\ell+2,2} \}^{\frac{5+8l}{2(\ell +2)}})^{\beta_{l}}\max \{1, \|\vec{\kappa}_{j}\|_{\ell+2,2} \}^{\frac{5+8m-2|\beta|}{2(\ell+2)}}\\
& \leq C(m, \frac{1}{\delta}) \sum_{k=1}^3  \max
\{ 1, \| \vec{\kappa}_k\|_{\ell+2,2} \}^{\frac{5+8m}{2(\ell +2)}}
\end{align*}
where we have used Lemma~\ref{PQgeneral} with $\gamma=\frac{5+8m}{2(\ell +2)}$ (on recalling \eqref{DefBeta} note that 
$\sum_{l=0}^{m-1}\beta_{l}(5+8\ell)< \sum_{l=0}^{m-1}2\beta_{l}(3+4\ell)= 2 |\beta|$
.)
For the remaining terms in \eqref{phi1tmin0} (with $m$ replaced by $m+1$) we observe that
\begin{align*}
\Big| &\sum_{\substack{[[a,b]]\leq [[4m+1,2]]\\c\leq 4m+1\\b \, \, even}}P^{a,c}_{b} (\vec{\kappa}_{i}) 
+ \sum_{\beta \in S_{4m-2}^{m-2}} \prod_{l=0}^{m-2} ( \varphi_{i}^{(l)})^{\beta_{l}}\sum_{\substack{[[a,b]]\leq [[4m+1-|\beta|,	2]]\\c\leq 4m+1-|\beta|\\b \, \, even}}P^{a,c}_{b} (\vec{\kappa}_{i})  \Big|(0)\\
& \leq  C\max \{1, \|\vec{\kappa}_{i}\|_{\ell+2,2} \}^{\frac{8m+4}{2(\ell+2)}}  \\
& \quad +C\sum_{\beta \in S_{4m-2}^{m-2}} \prod_{l=0}^{m-2}        (C(l, \frac{1}{\delta}) \sum_{k=1}^3  \max
\{ 1, \| \vec{\kappa}_k\|_{\ell+2,2} \}^{\frac{5+8l}{2(\ell +2)}})^{\beta_{l}}  \max \{1, \|\vec{\kappa}_{i}\|_{\ell+2,2} \}^{\frac{8m+4-2|\beta|}{2(\ell+2)}}\\
& \leq C(m, \frac{1}{\delta}) \sum_{k=1}^3  \max
\{ 1, \| \vec{\kappa}_k\|_{\ell+2,2} \}^{\frac{5+8m}{2(\ell +2)}}
\end{align*}
where we have used Lemma~\ref{lem:trickbdry}, \eqref{++}, Lemma~\ref{lemineqshsum}, and Lemma~\ref{PQgeneral}. The claim now follows from  \eqref{phi1tmin0} putting all estimates together.
\end{proof}

\section{Long-time existence result} \label{sec:6}

\begin{ass}\label{assnetwork}
 We assume that on maximal interval of existence time $[0,T)$
\begin{enumerate}
\item the length of each of the curves remains uniformly bounded from below away from zero;
\item there exists a $1\geq \delta>0$ such that for any $t \in [0,T) $ there exist $i=i(t),j=j(t) \in \{1,2,3\}$ with $|\langle T_i,T_j\rangle|\leq 1-\delta$.
\end{enumerate}
\end{ass}

\begin{rem}
The first assumption is necessary to be able to use interpolation inequalities.
Under the second assumption the determinant of the matrix in Remark~\ref{remjunctphi} is bounded from below uniformly by $2 \delta$. 
Moreover $dim ( span \{ T_{1}, T_{2}, T_{3} \}) \geq 2$ for all $t \in [0,T)$.
\end{rem}

\begin{proof}[Proof of Theorem~\ref{mainthm}]
First of all let us remark that in the following the constant $C$ may change from line to line.
A short time existence result   gives that a solution $\Gamma=\{f_1,f_2,f_3\}$ exists in a small time interval, see Section \ref{ste}. In particular (see again Section \ref{ste}) the tangential components $\varphi_{i}$ grow linearly in the interior of each curves. We will give notice when this fact plays a role in the proof.

We assume by contradiction that the solution does not exist globally in time, that is there exists $0<T <\infty$ where $T$ denotes the maximal existence time. 
In view of our Assumption~\ref{assnetwork} this implies that at least one curve ceases to be smooth or regular at $t=T$.

Since \eqref{flownetwork1} is a gradient flow, the energy is decreasing in time and in particular the $L^2$-norm of the curvature is uniformly bounded in $(0,T)$. Indeed,
\begin{align*}
\max_{i=1,2,3} \| \vec{\kappa}_i\|^2_{L^2} \leq 2 \sum_{i=1}^3 \mathcal{E}_{\lambda} (f_i) \leq 2 \sum_{i=1}^3 \mathcal{E}_{\lambda} (f_{i,0}) <\infty \, .
\end{align*}
Similarly, since $\lambda_i>0$, $i=1,2,3$, the length of the curves remains uniformly bounded from above. 

\bigskip
\noindent \textbf{First Step} \underline{ Uniform estimate of $\|\nabla_s^2 \vec{\kappa}_i\|_{L^2}$ on $(0,T)$, $i=1,2,3$}.
\smallskip

We wish now to estimate the derivatives of the curvature. We use here that the normal velocity in \eqref{flownetwork1} can be written as
$$ -\nabla_{s}^{2} \vec{\kappa}_{i} + P^{0,0}_{3}(\vec{\kappa}_{i}) + \lambda_i \vec{\kappa}_i \, \quad  i=1,2,3,$$
and we simplify as much as possible the notation using the $ P_b^{a,c}(\vec{\kappa}_i)$. Using Lemma \ref{lempartint} (taking $\vec{\phi}=\nabla_s^2 \vec{\kappa_i}$) and Lemma \ref{evolcurvature}, and summing over $i$ we find
\allowdisplaybreaks{\begin{align*}
& \sum_{i=1}^3 \left(\frac{d}{dt} \frac{1}{2}\int_{I} |\nabla_s^2 \vec{\kappa}_i|^{2} ds + \int_{I}
|\nabla_s^4 \vec{\kappa}_i|^{2 } ds \right) \\ 
& \quad  = \sum_{i=1}^3 \Big( - [ \langle \nabla_s^2 \vec{\kappa}_i, \nabla_s^5 \vec{\kappa}_i \rangle ]_0^1+ [ \langle \nabla_s^3 \vec{\kappa}_i, \nabla_s^4 \vec{\kappa}_i \rangle ]_0^1  +
\int_{I} (P^{6,4}_{4} (\vec{\kappa}_i) + P^{6,2}_{4} (\vec{\kappa}_i) + P^{4,2}_{6} (\vec{\kappa}_i) )\, ds 
\Big)
\\ \nonumber
& \quad \quad + \sum_{i=1}^3 \Big(\int_{I} \langle \varphi_i \nabla_s^3 \vec{\kappa}_i + \frac{1}{2} \nabla_s^2 \vec{\kappa}_i \,  \varphi_{i,s}, \nabla_s^2 \vec{\kappa}_i\rangle ds + \lambda_i \int_{I} ( \langle \nabla_s^2 \vec{\kappa}_i, \nabla_s^4 \vec{\kappa}_i \rangle + P_4^{4,2}(\vec{\kappa}_i) )  ds \Big) \\ 
& \quad \hspace{-.5cm} = \sum_{i=1}^3 \Big( - [ \langle \nabla_s^2 \vec{\kappa}_i, \nabla_s^5 \vec{\kappa}_i \rangle ]_0^1+ [ \langle \nabla_s^3 \vec{\kappa}_i, \nabla_s^4 \vec{\kappa}_i \rangle ]_0^1 
+ \frac12 [ \varphi_i |\nabla_s^2 \vec{\kappa}_i |^2 ]_0^1  
+\sum_{\substack{[[a,b]]\leq[[6,4]]\\c\leq 4, \,  2\leq b}}  \int_I P^{a,c}_{b} (\vec{\kappa}_{i}) \, ds   \Big)
\end{align*}}
where we have used the fact that $\lambda_{i}$ are given constants.
Summing on both sides $\frac12 \int_{I} |\nabla_s^2 \vec{\kappa}_i|^{2} ds$ we find
\begin{align} \label{firstfor}
& \sum_{i=1}^3 \left(\frac{d}{dt} \frac{1}{2}\int_{I} |\nabla_s^2 \vec{\kappa}_i|^{2} ds + \frac{1}{2}\int_{I} |\nabla_s^2 \vec{\kappa}_i|^{2} ds + \int_{I}
|\nabla_s^4 \vec{\kappa}_i|^{2 } ds \right) \\ \nonumber
& \hspace{-.25cm} = \sum_{i=1}^3 \Big( - [ \langle \nabla_s^2 \vec{\kappa}_i, \nabla_s^5 \vec{\kappa}_i \rangle ]_0^1+ [ \langle \nabla_s^3 \vec{\kappa}_i, \nabla_s^4 \vec{\kappa}_i \rangle ]_0^1 
+ \frac12 [ \varphi_i |\nabla_s^2 \vec{\kappa}_i |^2 ]_0^1   
+ \hspace{-.3cm} \sum_{\substack{[[a,b]]\leq[[6,4]]\\c\leq 4, \,  2\leq b}}  \int_I P^{a,c}_{b} (\vec{\kappa}_{i}) \, ds
\Big) .
\end{align}
Using Lemma \ref{lemineqshsum} (with $A=6$, $B=4$, $C \leq 4$, $\ell=2$) one sees that  the integrals on the right hand side can be controlled as follows
\begin{align}\label{secondfor}
&\sum_{i=1}^3 \Big(\hspace{-.3cm}\sum_{\substack{[[a,b]]\leq[[6,4]]\\c\leq 4, \,  2\leq b}}  \int_I P^{a,c}_{b} (\vec{\kappa}_{i}) \, ds
\Big) 
 \leq  \varepsilon \sum_{i=1}^3  \int_0^1 |\nabla_s^4 \vec{\kappa}_i|^2 \, ds + C_{\epsilon}(\lambda_{i}, f_{i,0}, \mathcal{L}[f_i], \mathcal{E}_{\lambda_{i}}(f_{i,0}))\,
\end{align}
with $\varepsilon \in (0,1)$. Here we have used that fact that the length of the curves remains bounded away from zero. 
It remains to consider the boundary terms.

At the fixed boundary points (i.e. at $x=1$) all even derivatives of the curvature are zero by Lemma \ref{kgood} and hence in reality we have a contribution from the boundary terms only from the junction point. As one can immediately see we have high derivatives on the boundary terms but using the boundary condition as done in Lemma \ref{lemtriple} we can lower the order of the boundary terms. Let consider the three terms separately. Since $\vec{\kappa}=0$ at the boundary and at the junction point $\partial_t f_i= \partial_t f_j$ we can write
\begin{align}\label{trickStep1}
\sum_{i=1}^3 \left.  \langle -\nabla_s^2 \vec{\kappa}_i, \nabla_s^5 \vec{\kappa}_i \rangle \right|_{x=0} =  \sum_{i=1}^3  \left. \langle \partial_t f_i, \nabla_s^5 \vec{\kappa}_i \rangle \right|_{x=0} = \langle \partial_t f_1,\sum_{i=1}^3 \nabla_s^5 \vec{\kappa}_i \rangle \big|_{x=0} \, .
\end{align}
Then by Lemma \ref{lemtriple} at $x=0$ (recall that $\vec{\kappa}=0$ at $x=0$, therefore $\partial f_{i} = -\nabla_s^2 \vec{\kappa}_i + \varphi_{i} \partial_s f_i$)
\begin{align*}
& \sum_{i=1}^3  \langle -\nabla_s^2 \vec{\kappa}_i, \nabla_s^5 \vec{\kappa}_i \rangle \\
& = \langle \partial_t f_1,\sum_{i=1}^3  (P^{3,3}_3 (\vec{\kappa}_i) +2\lambda_i \nabla_s^{3} \vec{\kappa}_i - \lambda_i^2 \nabla_s \vec{\kappa}_i + \varphi_i \nabla_s^{2} \vec{\kappa}_i +(P_2^{4,3}(\vec{\kappa}_i)- \lambda_i  |\nabla_s \vec{\kappa}_i|^2)\partial_s f_i) \rangle  \\
& =  \sum_{i=1}^3 \langle \partial_t f_i ,  P^{3,3}_3 (\vec{\kappa}_i) +2\lambda_i \nabla_s^{3} \vec{\kappa}_i - \lambda_i^2 \nabla_s \vec{\kappa}_i + \varphi_i \nabla_s^{2} \vec{\kappa}_i +(P_2^{4,3}(\vec{\kappa}_i)- \lambda_i  |\nabla_s \vec{\kappa}_i|^2)\partial_s f_i\rangle   \\
& =  \sum_{i=1}^3  \big(  \sum_{\substack{[[a,b]]\leq[[5,4]]\\c\leq 3, \,  2\leq b\\ b\,  even}}   P^{a,c}_{b} (\vec{\kappa}_{i}) + \varphi_i (-|\nabla_s^{2} \vec{\kappa}_i |^2+P_2^{4,3}(\vec{\kappa}_i)- \lambda_i  |\nabla_s \vec{\kappa}_i|^2)\big)\, ,
\end{align*}
where we have used the fact that $\lambda_{i}$ are constant. By \eqref{trickboundary2} we get 
\begin{align*}
\sum_{\substack{[[a,b]]\leq[[5,4]]\\c\leq 3, \,  2\leq b\\ b\,  even}}   P^{a,c}_{b} (\vec{\kappa}_{i}) |_{x=0} \leq 
C \int_{I}\sum_{\substack{[[a,b]]\leq[[6,4]]\\c\leq 4, \,  2\leq b\\ b\,  even}}   |P^{a,c}_{b} (\vec{\kappa}_{i}) | ds \, .
\end{align*}
Thanks to the interpolation inequality \eqref{interineq2sum} (together with the uniform control on the lengths from below, see Assumption~\ref{assnetwork}) we find that the above terms can be controlled via absorbtion into the terms $\| \nabla_s^4 \vec{\kappa}_{i}\|_{L^2}^{2}$.  
Next we study the term that depends on the tangential component of the flow equation and hence is new compared to our previous studies (\cite{DLP}, \cite{DLP2}, \cite{DP}, \cite{Lin}, \cite{DP-rims}). 
The term we need to control is 
\begin{equation}\label{gab1}
\varphi_i (-|\nabla_s^{2} \vec{\kappa}_i |^2+P_2^{4,3}(\vec{\kappa}_i)- \lambda_i  |\nabla_s \vec{\kappa}_i|^2)\big|_{x=0} = \varphi_i \sum_{\substack{[[a,b]]\leq[[4,2]]\\c\leq 3, \,  2\leq b\\ b\,  even}}   P^{a,c}_{b} (\vec{\kappa}_{i}) |_{x=0}\, .
\end{equation} 
The factor $\varphi_i(0)$ can be expressed using the formulas for the tangential component in $x=0$ (see \eqref{phi1in0} for $\varphi_1$). Then by Assumption \ref{assnetwork} and  \eqref{gio1zero} (with $\ell=2$) we find
\begin{align}\label{gio1}
|\varphi_i(0)| 
\leq \frac{C}{\delta} \sum_{j=1}^3 \mathcal{L}[f_j]^{-3} \| \vec{\kappa}_j\|_{2}^{\frac38}  \| \vec{\kappa}_j\|_{4,2}^{\frac58} \,.
\end{align} 
For the other factor in \eqref{gab1} we use \eqref{trickboundary2}, Lemma \ref{lemineqshsum} and the uniform bounds on the lengths to write
\begin{align*}
\Big| \sum_{\substack{[[a,b]]\leq[[4,2]]\\c\leq 3, \,  2\leq b\\ b\,  even}}   P^{a,c}_{b} (\vec{\kappa}_{i}) |_{x=0}\, \Big| \leq C 
 \sum_{\substack{[[a,b]]\leq[[5,2]]\\c\leq 4, \,  2\leq b}}  \int_I | P^{a,c}_{b} (\vec{\kappa}_{i}) |\, ds \leq C \max\{1,\|\vec{\kappa}_i\|_{4,2}^{\frac54}\}
\end{align*}
Combining the estimates above and using the uniform bounds on the lengths we infer 
\begin{align*} 
& \Big|  \sum_{i=1}^3  \varphi_i (-|\nabla_s^{2} \vec{\kappa}_i |^2+P_2^{4,3}(\vec{\kappa}_i)+ \lambda_i  |\nabla_s \vec{\kappa}_i|^2) \big|_{x=0} \Big| \\
& \leq \sum_{i=1}^3  \Big(\sum_{j=1}^3 C(\mathcal{L}([f_j]), \delta) \max\{1, \| \vec{\kappa}_j\|_{4,2}\}^{\frac{5}{8}} \Big) 
\max\{1, \| \vec{\kappa}_i\|_{4,2}\}^{\frac{5}{4}} 
\leq \sum_{i=1}^3 C \max\{1, \| \vec{\kappa}_i\|_{4,2}\}^{\frac{15}{8}}
\end{align*}
using Young's inequality $|ab | \leq C(|a|^{p}+ |b|^{q})$ with $p=3$ and $q=\frac32$ on product terms of type $b= \max\{1, \| \vec{\kappa}_i\|_{4,2}\}^{\frac{5}{4}}$, $a= \max\{1, \| \vec{\kappa}_j\|_{4,2}\}^{\frac{5}{8}} $. Since the exponent of the factor $\| \vec{\kappa}_i\|_{4,2}$ is smaller than $2$ we can control this term.
\medskip

Proceeding similarly for the second boundary term in \eqref{firstfor} we find with Lemma \ref{lemtriple} for $i=1,2,3$
\begin{align*}
  \langle \nabla_s^3 \vec{\kappa}_i, \nabla_s^4 \vec{\kappa}_i \rangle \big|_0^1 
=  \lambda_i \langle \nabla_s^3 \vec{\kappa}_i, \nabla_s^2 \vec{\kappa}_i \rangle \big|_0^1  + \varphi_i \langle \nabla_s^3 \vec{\kappa}_i, \nabla_s \vec{\kappa}_i \rangle \big|_0^1 \,.
\end{align*}
As before with \eqref{trickboundary2} and  Lemma \ref{lemineqshsum} the first term on the right hand side can be controlled by $\| \nabla_s^4 \vec{\kappa}\|_{2}$ since $\lambda_{i}$ are constant and by Assumption~\ref{assnetwork}. The second term coincides with one of the term coming from the tangential component treated above and hence it is also controlled. 



The last boundary term in \eqref{firstfor} coincides with terms that we have already treated and hence this also behaves as $\|\vec{\kappa}_i\|_{4,2}^{\frac{15}{8}}$. By Corollary \ref{corinter} and Young's inequality we finally achieve that for $\varepsilon\in(0,1)$
 \begin{align*}
&  \sum_{i=1}^3 \Big( - [ \langle \nabla_s^2 \vec{\kappa}_i, \nabla_s^5 \vec{\kappa}_i \rangle ]_0^1+ [ \langle \nabla_s^3 \vec{\kappa}_i, \nabla_s^4 \vec{\kappa}_i \rangle ]_0^1 
+ \frac12 [ \varphi_i |\nabla_s^2 \vec{\kappa}_i |^2 ]_0^1   \Big) \\
& \qquad \leq \varepsilon \sum_{i=1}^3  \int_0^1 |\nabla_s^4 \vec{\kappa}_i|^2 \, ds + C_{\varepsilon}(\lambda_{i}, f_{i,0}, \mathcal{L}[f_i], \mathcal{E}_{\lambda_{i}}(f_{i,0}))\,.
\end{align*}
From \eqref{firstfor}, \eqref{secondfor} and the estimate above choosing $\varepsilon$ appropriately we find that
$$ \sum_{i=1}^3 \left(\frac{d}{dt} \frac{1}{2}\int_{I} |\nabla_s^2 \vec{\kappa}_i|^{2} ds + \frac{1}{2}\int_{I} |\nabla_s^2 \vec{\kappa}_i|^{2} ds \right) \leq \sum_{i=1}^3 C(\lambda, f_{i,0}, \mathcal{L}[f_i], \mathcal{E}_{\lambda_{i}}(f_{i,0})) \,,$$
and hence,  using the smoothness of the flow and of the initial data, we infer that $\|\nabla_s^2 \vec{\kappa}_i\|_{L^2} $ are uniformly bounded on $[0,T)$ for $i=1,2,3$.
\bigskip

\noindent \textbf{Second Step} \underline{Uniform estimate of $\|\nabla_s^6 \vec{\kappa}_i\|_{L^2}$ on $(0,T)$, $i=1,2,3$ for  special choice of $\varphi_{i}$} 
\smallskip 

Since equation \eqref{(P)} is of fourth order it is natural, after having controlled the second order derivative of the curvature, to control now the sixth order derivative of the curvature and then (later ) in the general step the derivative of order $2+4m$ for $m \in \mathbb{N}$. With interpolation inequalities we then get also estimates for the intermediate terms.

In this step we proceed as in the previous one applying Lemma \ref{lempartint}. The crucial point is the choice of the normal vector field $\vec{\phi}$. The main difficulties in the case of networks are the boundary terms and the tangential components. Concerning the boundary terms it is necessary to  lower their order (using the boundary conditions). Recall that since the last boundary condition (see \eqref{bc1}) involves a sum it was fundamental in \eqref{trickStep1} that $-\nabla_s^2 \vec{\kappa}_i$ is equal to the normal component of $\partial_t f_i$ at the triple junction. In order to use the same idea at this step instead of working with $\nabla_s^6 \vec{\kappa}_i$ directly we are going to work with the vector field $\nabla_t S_{1,i}$ where
$$S_{1,i}= -\nabla_s^2 \vec{\kappa}_i+ \varphi_i \partial_s f_i$$
that we call \emph{speed} (motivated by the fact that indeed at the triple junction $S_{1,i}=\partial_{t} f_{i}$). This choice is due to the fact that at the triple junction $\partial_t^2 f_i=\partial_t^2 f_j$. Moreover, since in the boundary terms scalar products with normal vector fields appear one could work with $\nabla_t \partial_t f_i$. In order to simplify further the computations we use that at the junction point $\partial_t f_i= -\nabla_s^2 \vec{\kappa}_i+ \varphi_i \partial_s f_i$. As we will see, $\nabla_t S_{1,i}$ goes like $-\nabla_s^6 \vec{\kappa}_i$ plus lower order terms.

The other difficulty is due to the tangential component characterised by the functions $\varphi_i$. These functions are needed to keep the network connected and hence play a role only at the triple junction. Because of this we have informations only at the triple junction (see Remark \ref{remjunctphi} and \ref{remjunctphi2}), more precisely on $\varphi_i(0)$ and $\partial_t^m \varphi_i(0)$ for $m \in \N$. On the other hand, when computing the evolution of several quantities, derivatives with respect to $s$ of $\varphi_i$ appear. In order to simplify the computations we make a special choice of $\varphi_i$ taking the linear interpolation between the value at the junction point, i.e. $\varphi_i(0)$, and the other boundary point where $\varphi_i(1)=0$ (since the point is kept fixed in time). More precisely, we choose to work with
\begin{align}\label{defEigvarphi}
\varphi_i(t,x) = \varphi_i(t,0) \Big( 1-\frac{1}{\mathcal{L}(f_i)} \int_0^x |\partial_x f_i| \, dx \Big), \; i=1,2,3 \,. 
\end{align}  
That this is possible has been discussed in details in Section~\ref{ste}.
As a consequence, for $i=1,2,3$ and for any $t \in(0,T)$
\begin{equation}\label{Eigvarphi}
\partial_s \varphi_i(x,t) = - \frac{1}{\mathcal{L}(f_i)}  \varphi_i(t,0) \mbox{ and } \partial_s^k \varphi_i(x,t) =0 \mbox{ for } x\in [0,1] \mbox{ for }k\geq 2\,.
\end{equation}
Moreover using \eqref{reltangpart0} we can write
\begin{align}
\partial_{t}&\varphi_{i} (t,x)= \partial_{t}\varphi_{i}(t,0) \left (1 - \frac{1}{\mathcal{L}(f_{i})} \int_{0}^{x} |(f_{i})_{x}| dx \right)
+\frac{\varphi_{i}(t,0)}{\mathcal{L}(f_{i})^{2}} (\frac{d}{dt} \mathcal{L}(f_{i}))\int_{0}^{x} |\partial_{x}f_{i}| dx \notag \\
& \quad- \frac{\varphi_{i}(t,0)}{\mathcal{L}(f_{i})} \int_{0}^{x} \langle \partial_{s}f_{i}, \partial_{x}\partial_{t}f_{i} \rangle dx,\notag\\
&=\partial_{t}\varphi_{i}(t,0) \big(1 - \frac{1}{\mathcal{L}(f_{i})} \int_{0}^{x} |\partial_{x}f_{i}| dx \big)
-\frac{\varphi_{i}(t,0)}{\mathcal{L}(f_{i})^{2}} (\varphi_{i}(t,0)+\int_{I}\langle\vec{\kappa}_{i}, \vec{V}_{i} \rangle ds)\int_{0}^{x} |\partial_{x}f_{i}| dx \notag  \\ \label{dert1varphi}
& \quad- \frac{\varphi_{i}(t,0)}{\mathcal{L}(f_{i})} \int_{0}^{x} (\partial_{s} \varphi_{i} - \langle\vec{\kappa}_{i}, \vec{V}_{i} \rangle) ds .
\end{align}
From here on we will use these special choices of $\varphi_i$'s without further notice. Observe that in the first step $\varphi_i$ were arbitrary.

We can now start with the computations. By Lemma \ref{lempartint} with $\vec{\phi}= \nabla_t S_{1,i}$ and summing over $i$ we find
\allowdisplaybreaks{\begin{align}\label{39+2}
& \frac{d}{dt} \sum_{i=1}^3  \frac{1}{2}\int_{I} |\nabla_t S_{1,i}|^{2} ds +\sum_{i=1}^3  \int_{I}
|\nabla_{s}^{2} \nabla_t S_{1,i}|^{2 } ds \\ \nonumber
& = - \sum_{i=1}^3 [ \langle \nabla_t S_{1,i}, \nabla_s^3 \nabla_t S_{1,i} \rangle ]_0^1+\sum_{i=1}^3  [ \langle \nabla_s \nabla_t S_{1,i}, \nabla_s^2 \nabla_t S_{1,i} \rangle ]_0^1
\\  \nonumber
& \qquad + \sum_{i=1}^3 \int_{I} \langle (\nabla_t +\nabla_s^4)\nabla_t S_{1,i} + \frac{1}{2} \nabla_t S_{1,i} \,  \partial_s \varphi_i, \nabla_t S_{1,i} \rangle ds - \frac{1}{2} \sum_{i=1}^3  \int_{I} |\nabla_t S_{1,i}|^{2} \langle \vec{\kappa}_i, \vec{V}_i \rangle ds .
\end{align}}
We need to compute these terms. Since 
$$\vec{V}_i   = \sum_{\substack{[[a,b]] \leq [[2,1]]\\c\leq 2,\ b \, odd}} P^{a,c}_{b} (\vec{\kappa}_{i})\, ,\qquad \quad \quad 
\langle \vec{\kappa}_i, \vec{V}_i \rangle = \sum_{\substack{[[a,b]] \leq [[2,2]]\\c\leq 2,\ b \, even}} P^{a,c}_{b} (\vec{\kappa}_{i}),$$
by Lemma \ref{evolcurvature} (with $m=2$ and $m=6$), \eqref{cs2}, \eqref{cs3}, \eqref{cs1}, \eqref{c}, \eqref{Eigvarphi}, \eqref{E2sum}, \eqref{e} we find
$$h:=\varphi_{i}^{2}, \qquad \partial_{s}h = 2 \varphi_{i} \partial_{s} \varphi_{i}, \qquad  \partial_{s}^{2}h = 2  (\partial_{s} \varphi_{i})^{2}, \qquad \partial_{s}^{m} h =0 \text{ for } m \geq 3,$$
and
\allowdisplaybreaks{\begin{align}\label{haha}
\nabla_t S_{1,i} & = \nabla_s^6 \vec{\kappa}_i + \sum_{\substack{[[a,b]] \leq [[4,3]]\\c\leq 4,\ b \, odd}} P^{a,c}_{b} (\vec{\kappa}_{i}) + \varphi_i \sum_{\substack{[[a,b]] \leq [[3,1]]\\c\leq 3,\ b \, odd}} P^{a,c}_{b} (\vec{\kappa}_{i}) + \varphi_i^2 \vec{\kappa}_{i} \, ,\\ \nonumber
\nabla_s^2 \nabla_t S_{1,i} & = \nabla_s^8 \vec{\kappa}_i + \sum_{\substack{[[a,b]] \leq [[6,3]]\\c\leq 6,\ b \, odd}} P^{a,c}_{b} (\vec{\kappa}_{i}) + \varphi_i \sum_{\substack{[[a,b]] \leq [[5,1]]\\c\leq 5,\ b \, odd}} P^{a,c}_{b} (\vec{\kappa}_{i}) \\ \nonumber
& \quad +
\partial_s \varphi_i \sum_{\substack{[[a,b]] \leq [[4,1]]\\c\leq 4,\ b \, odd}} P^{a,c}_{b} (\vec{\kappa}_{i}) + \varphi_i^2 \nabla_s^2 \vec{\kappa}_{i} + 4 \varphi_i \partial_s \varphi_i \nabla_s \vec{\kappa}_i  + 2 (\partial_s \varphi_i)^{2}\vec{\kappa}_i   \, ,\\ \nonumber
\nabla_s^4 \nabla_t S_{1,i} & = \nabla_s^{10} \vec{\kappa}_i + \sum_{\substack{[[a,b]] \leq [[8,3]]\\c\leq 8,\ b \, odd}} P^{a,c}_{b} (\vec{\kappa}_{i}) + \varphi_i \sum_{\substack{[[a,b]] \leq [[7,1]]\\c\leq 7,\ b \, odd}} P^{a,c}_{b} (\vec{\kappa}_{i}) \\ \nonumber
& \quad +
\partial_s \varphi_i \sum_{\substack{[[a,b]] \leq [[6,1]]\\c\leq 6,\ b \, odd}} P^{a,c}_{b} (\vec{\kappa}_{i})
+ \varphi_i^2 \nabla_s^4 \vec{\kappa}_{i}+  8 \varphi_i \partial_s \varphi_i \nabla_s^3 \vec{\kappa}_{i} 
+ 12 (\partial_s \varphi_i)^2 \nabla_s^2 \vec{\kappa}_i \, ,\\ \nonumber
\nabla_t \nabla_t S_{1,i} & 
= - \nabla_s^{10} \vec{\kappa}_i + \sum_{\substack{[[a,b]] \leq [[8,3]]\\c\leq 8,\ b \, odd}} P^{a,c}_{b} (\vec{\kappa}_{i})  + \varphi_i \sum_{\substack{[[a,b]] \leq [[7,1]]\\c\leq 7,\ b \, odd}} P^{a,c}_{b} (\vec{\kappa}_{i}) \\   
\nonumber
& \quad + \partial_t \varphi_i \sum_{\substack{[[a,b]] \leq [[3,1]]\\c\leq 3,\ b \, odd}} P^{a,c}_{b} (\vec{\kappa}_{i}) 
 + \varphi_i^2 \sum_{\substack{[[a,b]] \leq [[4,1]]\\c\leq 4,\ b \, odd}} P^{a,c}_{b} (\vec{\kappa}_{i}) + \varphi_i^3 \nabla_s \vec{\kappa}_{i} + 2 \varphi_i \partial_t \varphi_i \vec{\kappa}_i \, .  
\end{align}}
Due to our choice of $\varphi_i$ and since the length is uniformly bounded from below, we have (up to a constant) the same estimate from above for $|\varphi_i|$ and $|\partial_s\varphi_i|$, namely
\begin{align}\label{phiabsch}
 |\varphi_i(t,x)|\leq |\varphi(t,0)|:=|\varphi_i(0)| \mbox{ and }|\partial_s\varphi_i(t,x)|\leq C|\varphi_i(0)| \mbox{ for all  }x \in [0,1].
 \end{align} 
 Also from \eqref{dert1varphi}, the uniform bounds for the curvature derived in the first step, and interpolation inequalities we infer
 \begin{align*}
 |\partial_{t}\varphi_{i}(t,x)| \leq |\partial_{t}\varphi_{i}(t,0)|+C (|\varphi_{i}(t,0)|^{2}+ |\varphi_{i}(t,0)|)= |\partial_{t}\varphi_{i}(0)|+C(|\varphi_{i}(0)|^{2}+ |\varphi_{i}(0)|).
 \end{align*}
As a consequence we find on $[0,1]$
\allowdisplaybreaks{\begin{align*}
| (\nabla_t +\nabla_s^4)\nabla_t S_{1,i} | & \leq \sum_{\substack{[[a,b]] \leq [[8,3]]\\c\leq 8,\ b \, odd}} |P^{a,c}_{b} (\vec{\kappa}_{i}) | + |\varphi_i(0)| \sum_{\substack{[[a,b]] \leq [[7,1]]\\c\leq 7,\ b \, odd}} |P^{a,c}_{b} (\vec{\kappa}_{i})|\\
& \quad + |\varphi_i(0)|^2 \sum_{\substack{[[a,b]] \leq [[4,1]]\\c\leq 4,\ b \, odd}} |P^{a,c}_{b} (\vec{\kappa}_{i}) |+ |\varphi_i(0)|^3 \sum_{\substack{[[a,b]] \leq [[1,1]]\\c\leq 1,\ b \, odd}} |P^{a,c}_{b} (\vec{\kappa}_{i})| \\
& \quad
+ |\partial_t \varphi_i(0)| \sum_{\substack{[[a,b]] \leq [[3,1]]\\c\leq 3,\ b \, odd}} |P^{a,c}_{b} (\vec{\kappa}_{i})| + 2 |\varphi_i(0) \partial_t \varphi_i(0)| |\vec{\kappa}_i |.
\end{align*}}
Next, using the simple inequalities
\begin{align*}
|a+b|^{2}=|a|^{2} + |b|^{2} + 2a \cdot b \geq (1-\epsilon)|a|^{2}-C_{\epsilon}|b^{2}|,  \qquad (\sum a_i)^2\leq c(\sum_i a_i^2)
\end{align*}
and \eqref{phiabsch} we infer that
\begin{align*}
|\nabla_s^2\nabla_t S_{1,i} |^2 & \geq \frac{1}{2}|\nabla_s^8 \vec{\kappa}_{i}|^2-\sum_{\substack{[[a,b]] \leq [[12,6]]\\c\leq 6,\ b \, even}} |P^{a,c}_{b} (\vec{\kappa}_{i}) | \\
& \quad - |\varphi_i(0)|^{2} \sum_{\substack{[[a,b]] \leq [[10,2]]\\c\leq 5,\ b \, even}} |P^{a,c}_{b} (\vec{\kappa}_{i}) | - |\varphi_i(0)|^{4} \sum_{\substack{[[a,b]] \leq [[4,2]]\\c\leq 2,\ b \, even}} |P^{a,c}_{b} (\vec{\kappa}_{i}) |
\end{align*}
as well as
\begin{align}\label{relN6NS}
|\nabla_t S_{1,i} |^2 & \geq \frac{1}{2}|\nabla_s^6 \vec{\kappa}_{i}|^2-\sum_{\substack{[[a,b]] \leq [[8,6]]\\c\leq 4,\ b \, even}} |P^{a,c}_{b} (\vec{\kappa}_{i}) | \\
& \quad - |\varphi_i(0)|^{2} \sum_{\substack{[[a,b]] \leq [[6,2]]\\c\leq 3,\ b \, even}} |P^{a,c}_{b} (\vec{\kappa}_{i}) | - |\varphi_i(0)|^{4} \sum_{\substack{[[a,b]] \leq [[0,2]]\\c\leq 0,\ b \, even}} |P^{a,c}_{b} (\vec{\kappa}_{i}) |. \nonumber
\end{align}

From \eqref{39+2} we  get using the previous relations that 
\allowdisplaybreaks{\begin{align*}
& \frac{d}{dt} \sum_{i=1}^3  \frac{1}{2}\int_{I} |\nabla_t S_{1,i}|^{2} ds + \sum_{i=1}^3  \frac{1}{2}\int_{I} |\nabla_t S_{1,i}|^{2} ds+ \sum_{i=1}^3 \frac{1}{2} \int_{I}
|\nabla_s^8 \vec{\kappa}_{i}|^{2 } ds \\ 
& \leq - \sum_{i=1}^3 [ \langle \nabla_t S_{1,i}, \nabla_s^3 \nabla_t S_{1,i} \rangle ]_0^1+\sum_{i=1}^3  [ \langle \nabla_s \nabla_t S_{1,i}, \nabla_s^2 \nabla_t S_{1,i} \rangle ]_0^1
\\  
& \qquad + \sum_{i=1}^3  \int_{I} \sum_{\substack{[[a,b]] \leq [[14,4]]\\c\leq 8,\ b \, even}} |P^{a,c}_{b} (\vec{\kappa}_{i}) |\,  ds+ \sum_{i=1}^3 \sum_{j=1}^5|\varphi_i(0)|^j   \int_{I} \sum_{\substack{[[a,b]] \leq [[16-3j,2]]\\c\leq \min\{8,a\},\ b \, even}} |P^{a,c}_{b} (\vec{\kappa}_{i})|\,ds\\
& \qquad +  \sum_{i=1}^3 \sum_{j=0}^3 |\partial_t \varphi_i(0)| |\varphi_i(0)|^j  \int_{I}\sum_{\substack{[[a,b]] \leq [[9-3j,2]]\\c\leq \min\{6,a\},\ b \, even}} |P^{a,c}_{b} (\vec{\kappa}_{i})| \,ds .
\end{align*}}

We first estimate the integral terms. By Lemma \ref{lemineqshsum} (with $A=14$, $B=4$, $c \leq 8$, $\ell=6$) one sees that
\begin{align*}
\sum_{i=1}^3  \int_{I} \sum_{\substack{[[a,b]] \leq [[14,4]]\\c\leq 8,\ b \, even}} |P^{a,c}_{b} (\vec{\kappa}_{i}) |
 \leq  \varepsilon \sum_{i=1}^3  \int_0^1 |\nabla_s^8 \vec{\kappa}_i|^2 \, ds + C_{\epsilon}(\lambda_{i}, f_{i,0}, \mathcal{L}[f_i],  \mathcal{E}_{\lambda_{i}}(f_{i,0}))\,,
\end{align*}
with $\varepsilon \in (0,1)$. Here we have used the fact that the length of the curves remains bounded away from zero. By the same interpolation result and \eqref{gio1zero}, again with $\ell=6$, we get
\begin{align*}
\sum_{j=1}^5|\varphi_i(0)|^j \sum_{i=1}^3  \int_{I} \sum_{\substack{[[a,b]] \leq [[16-3j,2]]\\c\leq \min\{8,a\},\ b \, even}} |P^{a,c}_{b} (\vec{\kappa}_{i})| & \leq \frac{C}{\delta} \sum_{j=1}^5 \sum_{i=1}^3 \sum_{k=1}^3 \| \vec{\kappa}_k\|_{8,2}^{\frac{5j}{16}} \max\{1,\| \vec{\kappa}_i\|_{8,2}\}^{\frac{16-3j}{8}}  \, .
\end{align*}
Then using Corollary \ref{corinter}, the uniform bound on the lengths and Lemma~\ref{Younggeneral} 
we find
\begin{align*}
\sum_{j=1}^5|\varphi_i(0)|^j \sum_{i=1}^3  \int_{I} \sum_{\substack{[[a,b]] \leq [[16-3j,2]]\\c\leq \min\{8,a\},\ b \, even}} |P^{a,c}_{b} (\vec{\kappa}_{i})| 
& \leq \varepsilon \sum_{i=1}^3  \int_0^1 |\nabla_s^8 \vec{\kappa}_i|^2 \, ds + C_{\epsilon}
\,, 
\end{align*}
where $C_{\epsilon}=C_{\epsilon}(\delta,\lambda_{i}, f_{i,0}, \mathcal{L}[f_i],  \mathcal{E}_{\lambda_{i}}(f_{i,0}))$. With the same arguments and now using also \eqref{gio1Tzero} we estimate
\begin{align*}
&\sum_{i=1}^3  \sum_{j=0}^3 |\partial_t \varphi_i(0)| |\varphi_i(0)|^j \int_{I}\sum_{\substack{[[a,b]] \leq [[9-3j,2]]\\c\leq \min\{6,a\},\ b \, even}} |P^{a,c}_{b} (\vec{\kappa}_{i})| ds\\
& \leq \frac{C}{\delta^{2}} \sum_{i=1}^3  \sum_{j=0}^3 \sum_{k=1}^3 \max\{1,\| \vec{\kappa}_k\|_{8,2}\}^{\frac{13}{16}}   \sum_{l=1}^3  \| \vec{\kappa}_l\|_{8,2}^{\frac{5j}{16}}  \max\{1,\| \vec{\kappa}_i\|_{8,2}\}^{\frac{9-3j}{8}}\\
& \leq \varepsilon \sum_{i=1}^3  \int_0^1 |\nabla_s^8 \vec{\kappa}_i|^2 \, ds + C_{\epsilon}(\delta,\lambda_{i}, f_{i,0}, \mathcal{L}[f_i],  \mathcal{E}_{\lambda_{i}}(f_{i,0}))\,, 
\end{align*}
using again the general Young inequality of Lemma~\ref{Younggeneral}.
It remains to treat the boundary terms. We may write 
\begin{align*}
-\sum_{i=1}^3 [ \langle \nabla_t S_{1,i}, \nabla_s^3 \nabla_t S_{1,i} \rangle ]_0^1 = - \sum_{i=1}^3 [ \langle \partial_t S_{1,i}, \nabla_s^3 \nabla_t S_{1,i} \rangle ]_0^1 \, .
\end{align*}
Since $S_{1,i}=\partial_t f_i$ at the boundary points, the velocities and their time derivatives coincide at the  boundary points (and vanish at $x=1$ where the points are fixed in time) we find
 \begin{align*}
- \sum_{i=1}^3 [ \langle \nabla_t S_{1,i}, \nabla_s^3 \nabla_t S_{1,i} \rangle ]_0^1 =  \langle \partial_t S_{1,1}, \sum_{i=1}^3 \nabla_s^3 \nabla_t S_{1,i} \rangle \Big|_{x=0} \, .
\end{align*}
Since 
\begin{align*}
\nabla_s^3 \nabla_t S_{1,i} & = \nabla_s^9 \vec{\kappa}_i + \sum_{\substack{[[a,b]] \leq [[7,3]]\\c\leq 7,\ b \, odd}} P^{a,c}_{b} (\vec{\kappa}_{i}) + \varphi_i \sum_{\substack{[[a,b]] \leq [[6,1]]\\c\leq 6,\ b \, odd}} P^{a,c}_{b} (\vec{\kappa}_{i}) \\
& \quad + \partial_s \varphi_i \sum_{\substack{[[a,b]] \leq [[5,1]]\\c\leq 5,\ b \, odd}} P^{a,c}_{b} (\vec{\kappa}_{i}) + \varphi_i^2 \nabla_s^3 \vec{\kappa}_{i} + 6 \varphi_i \partial_s \varphi_i \nabla_s^2 \vec{\kappa}_{i}  + 6 (\partial_s \varphi_i)^2 \nabla_s \vec{\kappa}_i \, ,
\end{align*}
by \eqref{haha}, we get using \eqref{bm9} the following order reduction
\begin{align*}
\sum_{i=1}^3 \nabla_s^3 \nabla_t S_{1,i} & = \sum_{i=1}^3 \Big( \sum_{\substack{[[a,b]] \leq [[7,3]]\\c\leq 7,\ b \, odd}} P^{a,c}_{b} (\vec{\kappa}_{i}) + \varphi_i \sum_{\substack{[[a,b]] \leq [[6,1]]\\c\leq 6,\ b \, odd}} P^{a,c}_{b} (\vec{\kappa}_{i}) \\
& \quad + \partial_s \varphi_i \sum_{\substack{[[a,b]] \leq [[5,1]]\\c\leq 5,\ b \, odd}} P^{a,c}_{b} (\vec{\kappa}_{i}) 
+ 6 \varphi_i \partial_s \varphi_i \nabla_s^2 \vec{\kappa}_{i} + 6 (\partial_s \varphi_i)^2 \nabla_s \vec{\kappa}_i  \\
& \quad - \partial_t \varphi_{i} \nabla_s^2\vec{\kappa}_{i} 
+ \partial_s f_i \big( \sum_{\substack{[[a,b]] \leq [[8,2]]\\c\leq 7,\ b \, even}} P^{a,c}_{b} (\vec{\kappa}_{i})  
+ \varphi_i \sum_{\substack{[[a,b]] \leq [[5,2]]\\c\leq 4,\ b \, even}} P^{a,c}_{b} (\vec{\kappa}_{i}) 
 \big) \Big) \Big|_{x=0}\, .   
\end{align*}
Note that we can write the first term as
$$\sum_{\substack{[[a,b]] \leq [[7,3]]\\c\leq 7,\ b \, odd}} P^{a,c}_{b} (\vec{\kappa}_{i}) =  \sum_{\substack{[[a,b]] \leq [[7,3]]\\c\leq 7,\ b \, odd, b \ne 1}} P^{a,c}_{b} (\vec{\kappa}_{i}) + \sum_{\substack{[[a,b]] \leq [[7,3]]\\c\leq 7,\ b=1}} P^{a,c}_{b} (\vec{\kappa}_{i}) 
$$
and  by \eqref{embed} we can write
$$ \sum_{\substack{[[a,b]] \leq [[7,3]]\\c\leq 7,\ b \, odd, b \ne 1}}| P^{a,c}_{b} (\vec{\kappa}_{i}) | \leq  \int_{I} \sum_{\substack{[[a,b]] \leq [[8,3]]\\c\leq 8,\ b \geq 2}}| P^{a,c}_{b} (\vec{\kappa}_{i})| ds.$$
Then using Lemma~\ref{lem:trickbdry} for the other terms and  \eqref{phiabsch} we obtain
\begin{align*}
|\sum_{i=1}^3 \nabla_s^3 \nabla_t S_{1,i} |_{x=0}& \leq \sum_{i=1}^3 
\int_{I} \sum_{\substack{[[a,b]] \leq [[8,3]]\\c\leq 8,\ b \geq 2}}| P^{a,c}_{b} (\vec{\kappa}_{i})| ds\\
 & \quad + \sum_{i=1}^3 \Big( \int_I \sum_{\substack{[[a,b]] \leq [[15,2]]\\c\leq 8,\ b \, even}} \hspace{-.3cm}| P^{a,c}_{b} (\vec{\kappa}_{i})| \, ds + |\varphi_i (0)|^2 \int_I\sum_{\substack{[[a,b]] \leq [[13,2]]\\c\leq 7,\ b \, even}} \hspace{-.3cm}|P^{a,c}_{b} (\vec{\kappa}_{i})| \\
& \quad  + |\varphi_i(0)|^4 \int_I \sum_{\substack{[[a,b]] \leq [[7,2]]\\c\leq 4,\ b \, even}} \hspace{-.3cm}| P^{a,c}_{b} (\vec{\kappa}_{i})| \, ds  \Big)^{\frac12}\\
& \quad +   \sum_{i=1}^3  \Big(
\int_I \sum_{\substack{[[a,b]] \leq [[9,2]]\\c\leq 8,\ b \, even}} \hspace{-.3cm} |P^{a,c}_{b} (\vec{\kappa}_{i}) |\, ds  +  |\varphi_i(0)| \int_I \sum_{\substack{[[a,b]] \leq [[6,2]]\\c\leq 5,\ b \, even}} \hspace{-.3cm}|P^{a,c}_{b} (\vec{\kappa}_{i}) |\, ds  \Big) \\
& \quad + \sum_{i=1}^3|\partial_t \varphi_{i}(0)| |\nabla_s^2 \vec{\kappa}_i| \, . 
\end{align*}

By Lemma~\ref{lemineqshsum} and \eqref{gio1zero}, again with $\ell=6$, with the same arguments as above we see that the integrals on the right hand side can be bounded by
\begin{align*}
|\sum_{i=1}^3 \nabla_s^3 \nabla_t S_{1,i} |_{x=0}& \leq C\sum_{i=1}^3 \max\{1, \|\vec{\kappa}_i \|_{8,2}\}^{\frac{18}{16}}+ \frac{C}{\delta} \sum_{i=1}^3 \sum_{k=1}^3  \|\vec{\kappa}_k \|_{8,2}^{\frac{5}{16}} \max\{1, \|\vec{\kappa}_i \|_{8,2}\}^{\frac{13}{16}}\\
&\quad   + \frac{C}{\delta^2} \sum_{i=1}^3 \sum_{k=1}^3  \|\vec{\kappa}_k \|_{8,2}^{\frac{5}{8}} \max\{1, \|\vec{\kappa}_i \|_{8,2}\}^{\frac{7}{16}} + \sum_{i=1}^3|\partial_t \varphi_{i}(0)| |\nabla_s^2 \vec{\kappa}_i| \big|_{x=0}\, .
\end{align*}

For the last term we observe that since $\vec{\kappa}_i(1)=0$ and $\partial_t f(1)=0$ it follows $\nabla_s^2 \vec{\kappa}_i(1)=0$
and hence with Lemma \ref{lemineqshsum} 
\begin{align*}
|\nabla_s^2 \vec{\kappa}_i (0)| & = | |\nabla_s^2 \vec{\kappa}_i (1)|^2 - |\nabla_s^2 \vec{\kappa}_i (0)|^2|^{\frac12}
 = \left( 2\int_I \langle \nabla_s^2 \vec{\kappa}_i, \nabla_s^3 \vec{\kappa}_i \rangle ds  \right)^{\frac12} \leq C \max\{1, \|\vec{\kappa}_i \|_{8,2}\}^{\frac{5}{16}} \, .
\end{align*}
By \eqref{gio1Tzero} we finally get
$$  |\partial_t \varphi_{i}(0)| |\nabla_s^2 \vec{\kappa}_i|_{x=0} \leq \frac{C}{\delta} \sum_{j=1}^3 \max\{1, \|\vec{\kappa}_j \|_{8,2}\}^{\frac{13}{16}} \max\{1, \|\vec{\kappa}_i \|_{8,2}\}^{\frac{5}{16}}  \, . $$

On the other hand with \eqref{d}, Lemma~\ref{evolcurvature}, \eqref{c}, and using that $\vec{\kappa}_i=0$ at the boundary we find
\begin{align*}
\partial_t S_{1,i} (0)& = \nabla_s^6 \vec{\kappa}_i + \sum_{\substack{[[a,b]] \leq [[4,3]]\\c\leq 4,\ b \, odd}} P^{a,c}_{b} (\vec{\kappa}_{i}) + \varphi_i \sum_{\substack{[[a,b]] \leq [[3,1]]\\c\leq 3,\ b \, odd}} P^{a,c}_{b} (\vec{\kappa}_{i})\\
& \quad + \partial_t \varphi_i \partial_s f_i +\partial_s f_i \sum_{\substack{[[a,b]] \leq [[5,2]]\\c\leq 3,\ b \, even}} P^{a,c}_{b} (\vec{\kappa}_{i}) \, , 
\end{align*}
so that with Lemma~\ref{lem:trickbdry}, Lemma \ref{lemineqshsum}, \eqref{gio1zero} and \eqref{gio1Tzero} we obtain
\allowdisplaybreaks{\begin{align*}
|\partial_t S_{1,i}| (0)& \leq  \Big( \int_I \sum_{\substack{[[a,b]] \leq [[13,2]]\\c\leq 7,\ b \, even}}| P^{a,c}_{b} (\vec{\kappa}_{i})| \, ds + |\varphi_i (0)|^2 \int_I\sum_{\substack{[[a,b]] \leq [[7,2]]\\c\leq 4,\ b \, odd}} |P^{a,c}_{b} (\vec{\kappa}_{i})|  ds \Big)^{\frac12}\\
& \quad + \sum_{\substack{[[a,b]] \leq [[6,2]]\\c\leq 4,\ b \, even}} \int_{I}| P^{a,c}_{b} (\vec{\kappa}_{i})| ds + |  \partial_t \varphi_i(0)| \\
& \hspace{-.5cm}\leq C  \max\{1, \|\vec{\kappa}_i \|_{8,2}\}^{\frac{13}{16}}+ \frac{C}{\delta}\big( \sum_{k=1}^3  \|\vec{\kappa}_k \|_{8,2}^{\frac{5}{16}} \max\{1, \|\vec{\kappa}_i \|_{8,2}\}^{\frac{7}{16}} + 
\max\{1, \|\vec{\kappa}_i \|_{8,2}\}^{\frac{13}{16}}  \big) . 
\end{align*}}
Since the sum of the exponents of the several terms is smaller than $2$  we can use Lemma~\ref{Younggeneral} and we obtain
\begin{align*}
|\sum_{i=1}^3 [ \langle \nabla_t S_{1,i}, \nabla_s^3 \nabla_t S_{1,i} \rangle ]_0^1| \leq  \varepsilon \sum_{i=1}^3  \int_0^1 |\nabla_s^8 \vec{\kappa}_i|^2 \, ds + C_{\epsilon}(\delta,\lambda_{i}, f_{i,0}, \mathcal{L}[f_i], \mathcal{E}_{\lambda_{i}}(f_{i,0}))\,,
\end{align*}
with $\varepsilon \in (0,1)$. 

It remains to evaluate the last boundary term $\sum_{i=1}^3  [ \langle \nabla_s \nabla_t S_{1,i}, \nabla_s^2 \nabla_t S_{1,i} \rangle ]_0^1
$ in \eqref{39+2}. First of all note that by \eqref{haha}, $\vec{\kappa}_{i}=0$, \eqref{bm8} respectively Lemma~\ref{kgood}, we infer that at the boundary points we have (since there are some cancellations and $\vec{\kappa}_i=0$)
\begin{align*}
\nabla_s^2 \nabla_t S_{1,i}&=\sum_{\substack{[[a,b]] \leq [[6,3]]\\c\leq 6,\ b \, odd}} P^{a,c}_{b} (\vec{\kappa}_{i}) + \varphi_i \sum_{\substack{[[a,b]] \leq [[5,1]]\\c\leq 5,\ b \, odd}} P^{a,c}_{b} (\vec{\kappa}_{i}) +
\partial_s \varphi_i \sum_{\substack{[[a,b]] \leq [[4,1]]\\c\leq 4,\ b \, odd}} P^{a,c}_{b} (\vec{\kappa}_{i}) \\
& \quad  + 4 \varphi_i \partial_s \varphi_i \nabla_s \vec{\kappa}_i  -\partial_{t}\varphi_{i} \nabla_s \vec{\kappa}_i \, ,
\end{align*}
and
\begin{align*}
\nabla_{s}\nabla_t S_{1,i} & = \nabla_s^7 \vec{\kappa}_i + \sum_{\substack{[[a,b]] \leq [[5,3]]\\c\leq 5,\ b \, odd}} P^{a,c}_{b} (\vec{\kappa}_{i}) + \partial_{ s}\varphi_i \sum_{\substack{[[a,b]] \leq [[3,1]]\\c\leq 3,\ b \, odd}} P^{a,c}_{b} (\vec{\kappa}_{i}) \\& \quad 
+ \varphi_i \sum_{\substack{[[a,b]] \leq [[4,1]]\\c\leq 4,\ b \, odd}} P^{a,c}_{b} (\vec{\kappa}_{i})+\varphi_i^2 \nabla_{s}\vec{\kappa}_{i} .
\end{align*}
We estimate each term separately at the boundary $x=0$ and $x=1$. (At $x=1$ many terms do actually vanish since here $\partial_{t} \varphi_{i}=0=\varphi_{i}$, however we do not diferentiate the treatment of the terms.) We use Lemma~\ref{lem:trickbdry}, \eqref{Eigvarphi}, 
Lemma~\ref{lemineqshsum}, \eqref{gio1zero}, \eqref{gio1Tzero} to obtain at $x=0$ or $x=1$
\begin{align*}
|\nabla_s^2 \nabla_t S_{1,i} |& \leq  \Big( \int_I \sum_{\substack{[[a,b]] \leq [[13,6]]\\c\leq 7,\ b \, even}}| P^{a,c}_{b} (\vec{\kappa}_{i})| \, ds + |\varphi_i (0)|^2 \int_I\sum_{\substack{[[a,b]] \leq [[11,2]]\\c\leq 6,\ b \, even}} |P^{a,c}_{b} (\vec{\kappa}_{i})| ds \\
& \quad  + |\varphi_i(0)|^4 \int_I \sum_{\substack{[[a,b]] \leq [[3,2]]\\c\leq 2,\ b \, even}} \hspace{-.25cm}| P^{a,c}_{b} (\vec{\kappa}_{i})| \, ds   \Big)^{\frac12}+ |\partial_t \varphi_{i}(0)| (\int_I \sum_{\substack{[[a,b]] \leq [[3,2]]\\c\leq 2,\ b \, even}}\hspace{-.25cm}| P^{a,c}_{b} (\vec{\kappa}_{i})| ds)^{\frac12} \\ 
& \leq C \max\{1, \|\vec{\kappa}_i \|_{8,2}\}^{\frac{15}{16}}+ \frac{C}{\delta}  \sum_{k=1}^3  \|\vec{\kappa}_k \|_{8,2}^{\frac{5}{16}} \max\{1, \|\vec{\kappa}_i \|_{8,2}\}^{\frac{11}{16}}\\
&\quad   + \frac{C}{\delta^{2}} \sum_{k=1}^3  \|\vec{\kappa}_k \|_{8,2}^{\frac{10}{16}} \max\{1, \|\vec{\kappa}_i \|_{8,2}\}^{\frac{3}{16}} +  \frac{C}{\delta} \sum_{k=1}^3  \|\vec{\kappa}_k \|_{8,2}^{\frac{13}{16}} \max\{1, \|\vec{\kappa}_i \|_{8,2}\}^{\frac{3}{16}}\, ,
\end{align*}
and similarly
\begin{align*}
|\nabla_s \nabla_t S_{1,i} |& \leq  \Big( \int_I \sum_{\substack{[[a,b]] \leq [[15,2]]\\c\leq 8,\ b \, even}}| P^{a,c}_{b} (\vec{\kappa}_{i})| \, ds + |\varphi_i (0)|^2 \int_I\sum_{\substack{[[a,b]] \leq [[9,2]]\\c\leq 5,\ b \, even}} |P^{a,c}_{b} (\vec{\kappa}_{i})| ds \\
& \quad  + |\varphi_i(0)|^4 \int_I \sum_{\substack{[[a,b]] \leq [[3,2]]\\c\leq 2,\ b \, even}}| P^{a,c}_{b} (\vec{\kappa}_{i})| \, ds   \Big)^{\frac12}\\ 
& \leq C \max\{1, \|\vec{\kappa}_i \|_{8,2}\}^{\frac{15}{16}}+ \frac{C}{\delta}  \sum_{k=1}^3  \|\vec{\kappa}_k \|_{8,2}^{\frac{5}{16}} \max\{1, \|\vec{\kappa}_i \|_{8,2}\}^{\frac{9}{16}}\\
&\quad   + \frac{C}{\delta^{2}} \sum_{k=1}^3  \|\vec{\kappa}_k \|_{8,2}^{\frac{10}{16}} \max\{1, \|\vec{\kappa}_i \|_{8,2}\}^{\frac{3}{16}} \,. 
\end{align*}

Finally using 
 Lemma~\ref{Younggeneral}, Corollary~\ref{corinter}, and the uniform bounds on the length we obtain
\begin{align*}
|\sum_{i=1}^3  [ \langle \nabla_s \nabla_t S_{1,i}, \nabla_s^2 \nabla_t S_{1,i} \rangle ]_0^1| & \leq \varepsilon \sum_{i=1}^3  \int_0^1 |\nabla_s^8 \vec{\kappa}_i|^2 \, ds + C_{\epsilon}(\delta,\lambda_{i}, f_{i,0}, \mathcal{L}[f_i], \mathcal{E}_{\lambda_{i}}(f_{i,0}))\,.
\end{align*}
Putting all estimates together we finally obtain
\begin{align*}
& \frac{d}{dt} \sum_{i=1}^3  \frac{1}{2}\int_{I} |\nabla_t S_{1,i}|^{2} ds + \sum_{i=1}^3  \frac{1}{2}\int_{I} |\nabla_t S_{1,i}|^{2} ds+ \sum_{i=1}^3 \frac{1}{2} \int_{I}
|\nabla_s^8 \vec{\kappa}_{i}|^{2 } ds \\ 
& \leq \varepsilon \sum_{i=1}^3  \int_0^1 |\nabla_s^8 \vec{\kappa}_i|^2 \, ds + C_{\epsilon}(\delta,\lambda_{i}, f_{i,0}, \mathcal{L}[f_i],  \mathcal{E}_{\lambda_{i}}(f_{i,0}))\,.
\end{align*}
Choosing $\varepsilon$ appropriately and a Gronwall Lemma give $\sup_{(0,T)}\sum_{i=1}^3  \frac{1}{2}\int_{I} |\nabla_t S_{1,i}|^{2} ds  \leq C$, with $
C=C(\delta, \lambda_{j},f_{j,0},\mathcal{L}[f_j],\mathcal{E}_{\lambda_{j}}(f_{j,0}))$, $j=1,2,3$.
Together with \eqref{relN6NS}, Lemma~\ref{lemineqshsum} and \eqref{gio1zero} with $\ell=4$, Corollary~\ref{corinter}, the uniform bound on the lengths and Young inequality, standard arguments yield
$$\sup_{(0,T)}\sum_{i=1}^3  \int_{I} |\nabla_s ^{6}\vec{\kappa}_{i}|^{2} ds  \leq C $$ where $C=C(\delta, \lambda_{j},f_{j,0},\mathcal{L}[f_j],\mathcal{E}_{\lambda_{j}}(f_{j,0}))$, $j=1,2,3$. 
For the tangential component we derive
$$ \sup_{(0,T)}\sum_{i=1}^3 | \varphi_{i}(t,0)| +  | \partial_{s} \varphi_{i}(t,0)| \leq C \quad \mbox{ and }\quad \sup_{(0,T)}\sum_{i=1}^3 (\| \varphi_{i}\|_{L^{\infty}} +  \| \partial_{s} \varphi_{i}\|_{L^{\infty}}) \leq C $$
by \eqref{gio1zero} with $\ell=4$,  \eqref{phiabsch} and \eqref{Eigvarphi}. Again here $C$ depends on $\delta$,$ \lambda_{j}$, $f_{j,0}$, $\mathcal{L}[f_j]$ and $ \mathcal{E}_{\lambda_{j}}(f_{j,0})$
for $j=1,2,3$.

\bigskip
\noindent\textbf{Induction Step:}
In the following we denote by
\begin{align*}
S_{m,i}:= \partial_{t}^{m-1}S_{1,i} = \partial_{t}^{m-1} (- \nabla_{s}^{2} \vec{\kappa}_{i} + \varphi_{i} \partial_{s} f_{i}), \quad \qquad m \in \N. \end{align*} 
The induction hypothesis reads (for some $m \in \N$):
\begin{align}\label{IndHyp2}
\sup_{(0,T)}\sum_{i=1}^3 \int_{I} |\nabla_t S_{m,i}|^{2} ds  \leq 
C, \qquad & \sup_{(0,T)}\sum_{i=1}^3  \int_{I} |\nabla_s ^{2+4m}\vec{\kappa}_{i}|^{2} ds  \leq C 
,
\\ \label{IndHyp3}
\mbox{together with } &\sup_{(0,T)}\sum_{i=1}^3 (\sum_{k=0}^{m-1} |\partial_{t}^{k} \varphi_{i} (t,0)| ) \leq C 
,
\end{align}
where $C=C(\delta, \lambda_{j},f_{j,0}, \mathcal{L}[f_j],\mathcal{E}_{\lambda_{j}}(f_{j,0}))$ for $j=1,2,3$. Thus let us consider $S_{m+1,i}$ and derive the corresponding estimates.
By Lemma \ref{lempartint} with $\vec{\phi}= \nabla_t S_{m+1,i}$ and summing over $i$ we find
\allowdisplaybreaks{\begin{align}\label{39+3}
& \frac{d}{dt} \sum_{i=1}^3  \frac{1}{2}\int_{I} |\nabla_t S_{m+1,i}|^{2} ds +\sum_{i=1}^3  \int_{I}
|\nabla_{s}^{2} \nabla_t S_{m+1,i}|^{2 } ds \\ \nonumber
& = - \sum_{i=1}^3 [ \langle \nabla_t S_{m+1,i}, \nabla_s^3 \nabla_t S_{m+1,i} \rangle ]_0^1+\sum_{i=1}^3  [ \langle \nabla_s \nabla_t S_{m+1,i}, \nabla_s^2 \nabla_t S_{m+1,i} \rangle ]_0^1
\\  \nonumber
& \qquad + \sum_{i=1}^3 \int_{I} \langle (\nabla_t +\nabla_s^4)\nabla_t S_{m+1,i} + \frac{1}{2} \nabla_t S_{m+1,i} \,  \partial_s \varphi_i, \nabla_t S_{m+1,i} \rangle ds \\
& \qquad - \frac{1}{2} \sum_{i=1}^3  \int_{I} |\nabla_t S_{m+1,i}|^{2} \langle \vec{\kappa}_i, \vec{V}_i \rangle ds . \nonumber
\end{align}}
We need to compute these terms. 
For $m\geq 1$ we have that by  calculations similar the ones performed in Remark~\ref{remjunctphi2} (that is inductively using \eqref{c}, \eqref{d}, Lemma~\ref{evolcurvature}, Lemma~\ref{lemLin}) 
\begin{align}\label{SM+1}
S_{m+1,i}&=\partial_{t}^{m} (- \nabla_{s}^{2} \vec{\kappa}_{i} + \varphi_{i} \partial_{s} f_{i})\\
&=(-1)^{m+1} \nabla_{s}^{4m+2} \vec{\kappa}_{i}+
\sum_{\substack{[[a,b]]\leq [[4m,3]]\\c\leq 4m\\b \, \, odd}}P^{a,c}_{b} (\vec{\kappa}_{i}) 
\notag \\
 & \qquad +  \sum_{\beta \in S_{4m+3}^{m-1}} \prod_{l=0}^{m-1} ( \varphi_{i}^{(l)})^{\beta_{l}}\sum_{\substack{[[a,b]]\leq [[4m+2-|\beta|,	1]]\\c\leq 4m+2-|\beta|\\b \, \, odd}}P^{a,c}_{b} (\vec{\kappa}_{i}) \notag \\
& \quad + \partial_{s} f_{i} \Big(
\partial_{t}^{m}\varphi_{i}(t,x)
 +\sum_{\substack{[[a,b]]\leq [[4m+1,2]]\\c\leq 4m-1\\b \, \, even}}P^{a,c}_{b} (\vec{\kappa}_{i})  \notag\\
 & \qquad 
+ \sum_{\beta \in S_{4m}^{m-1}} \prod_{l=0}^{m-1} ( \varphi_{i}^{(l)})^{\beta_{l}}\sum_{\substack{[[a,b]]\leq [[4m+1-|\beta|,	2]]\\c\leq 4m+1-|\beta|\\b \, \, even}}P^{a,c}_{b} (\vec{\kappa}_{i}) \Big).  \notag
\end{align}
 
Then by Lemma~\ref{evolcurvature}, \eqref{E2sum}, \eqref{cs3}, \eqref{cs2}, and \eqref{c} we infer
\begin{align}\label{haha3}
&\nabla_{t} S_{m+1,i}
 =(-1)^{m+2}\nabla_{s}^{4(m+1)+2} \vec{\kappa}_{i}\\& \quad +
\sum_{\substack{[[a,b]]\leq [[4(m+1),3]]\\c\leq 4(m+1)\\b \, \, odd}} \hspace{-.3cm} P^{a,c}_{b} (\vec{\kappa}_{i}) 
 +  \sum_{\beta \in S_{4(m+1)+3}^{m}} \prod_{l=0}^{m} ( \varphi_{i}^{(l)})^{\beta_{l}}\sum_{\substack{[[a,b]]\leq [[4(m+1)+2-|\beta|,	1]]\\c\leq 4(m+1)+2-|\beta|\\b \, \, odd}} \hspace{-.3cm}P^{a,c}_{b} (\vec{\kappa}_{i}) , \notag\\ 
 &\nabla_{t}^{2} S_{m+1,i}
=(-1)^{m+1}\nabla_{s}^{4(m+2)+2} \vec{\kappa}_{i} \notag \\
& \quad+
\sum_{\substack{[[a,b]]\leq [[4(m+2),3]]\\c\leq 4(m+2)\\b \, \, odd}} \hspace{-.3cm}P^{a,c}_{b} (\vec{\kappa}_{i}) 
 +  \sum_{\beta \in S_{4(m+2)+3}^{m+1}} \prod_{l=0}^{m+1} ( \varphi_{i}^{(l)})^{\beta_{l}}\sum_{\substack{[[a,b]]\leq [[4(m+2)+2-|\beta|,	1]]\\c\leq 4(m+2)+2-|\beta|\\b \, \, odd}} \hspace{-.3cm}P^{a,c}_{b} (\vec{\kappa}_{i}).  \notag
\end{align}
Furthermore using \eqref{cs1} we deduce
\begin{align} \label{haha4}
&\nabla_{s}^{2}\nabla_{t} S_{m+1,i}
 =(-1)^{m}\nabla_{s}^{4m+8} \vec{\kappa}_{i}\\& \quad +
\sum_{\substack{[[a,b]]\leq [[4m+6,3]]\\c\leq 4m+6\\b \, \, odd}} \hspace{-.3cm} P^{a,c}_{b} (\vec{\kappa}_{i}) 
 +  \sum_{\beta \in S_{4(m+1)+3}^{m}} \prod_{l=0}^{m} ( \varphi_{i}^{(l)})^{\beta_{l}}\sum_{\substack{[[a,b]]\leq [[4m+8-|\beta|,	1]]\\c\leq 4m+8-|\beta|\\b \, \, odd}} \hspace{-.3cm}P^{a,c}_{b} (\vec{\kappa}_{i}) , \notag \\
 & \qquad +\sum_{\beta \in S_{4(m+1)+3}^{m}} \partial_{s} \left(\prod_{l=0}^{m} ( \varphi_{i}^{(l)})^{\beta_{l}} \right)\sum_{\substack{[[a,b]]\leq [[4(m+1)+3-|\beta|,	1]]\\c\leq 4(m+1)+3-|\beta|\\b \, \, odd}}P^{a,c}_{b} (\vec{\kappa}_{i}) \notag \\
 & \qquad \sum_{\beta \in S_{4(m+1)+3}^{m}}\partial_{s}^{2} \left( \prod_{l=0}^{m} ( \varphi_{i}^{(l)})^{\beta_{l}} \right)\sum_{\substack{[[a,b]]\leq [[4(m+1)+2-|\beta|,	1]]\\c\leq 4(m+1)+2-|\beta|\\b \, \, odd}}P^{a,c}_{b} (\vec{\kappa}_{i}). \notag
\end{align}
Using Lemma~\ref{lem:behvarphi}, where the behaviour of  the 
derivatives with respect to $s$ of time derivatives of $\varphi_{i}$
is investigated, (and the inequality $|a+b|^{2} \geq \frac{1}{2}|a|^{2} -C |b|^{2}$) we obtain
\begin{align}
|\nabla_{s}^{2}\nabla_{t} S_{m+1,i}|^{2}& \geq \frac{1}{2}|\nabla_{s}^{4m+8} \vec{\kappa}_{i}|^{2} - \sum_{\substack{[[a,b]]\leq [[8m+12,6]]\\c\leq 4m+6, \, b \, \, even}}\hspace{-.3cm} |  P^{a,c}_{b} (\vec{\kappa}_{i}) | \notag\\
& \quad-  \sum_{\beta \in S_{4(m+1)+3}^{m}} \prod_{l=0}^{m}(|\partial_{t}^{l}\varphi_{i}(t,0)| +1)^{2\beta_{l}}\sum_{\substack{[[a,b]]\leq [[8m+16-2|\beta|,	2]]\\c\leq 4m+8-|\beta|\\b \, \, even}} \hspace{-.3cm}|P^{a,c}_{b} (\vec{\kappa}_{i}) | \notag\\
& \quad-  \sum_{\beta \in S_{4(m+1)+3}^{m}} \prod_{l=0}^{m}(|\partial_{t}^{l}\varphi_{i}(t,0)| +1)^{2\beta_{l}}
\sum_{\substack{[[a,b]]\leq [[8(m+1)+6-2|\beta|,	2]]\\c\leq 4(m+1)+3-|\beta|\\b \, \, even}} \hspace{-.3cm}|P^{a,c}_{b} (\vec{\kappa}_{i})|  \notag\\
& \quad-  \sum_{\beta \in S_{4(m+1)+3}^{m}} \prod_{l=0}^{m}(|\partial_{t}^{l}\varphi_{i}(t,0)| +1)^{2\beta_{l}}
\sum_{\substack{[[a,b]]\leq [[8(m+1)+4-2|\beta|,	2]]\\c\leq 4(m+1)+2-|\beta|\\b \, \, even}} \hspace{-.3cm}|P^{a,c}_{b} (\vec{\kappa}_{i})| \notag \\
& \geq \frac{1}{2}|\nabla_{s}^{4m+8} \vec{\kappa}_{i}|^{2} - \sum_{\substack{[[a,b]]\leq [[8m+12,6]]\\c\leq 4m+6, \, b \, \, even}} |  P^{a,c}_{b} (\vec{\kappa}_{i}) | \notag\\
& \quad-  \sum_{\beta \in S_{4(m+1)+3}^{m}} \prod_{l=0}^{m}(|\partial_{t}^{l}\varphi_{i}(t,0)| +1)^{2\beta_{l}}\sum_{\substack{[[a,b]]\leq [[8m+16-2|\beta|,	2]]\\c\leq 4m+8-|\beta|\\b \, \, even}}|P^{a,c}_{b} (\vec{\kappa}_{i}) |\notag \\
& \geq \frac{1}{2}|\nabla_{s}^{4m+8} \vec{\kappa}_{i}|^{2} - \sum_{\substack{[[a,b]]\leq [[8m+12,6]]\\c\leq 4m+6, \, b \, \, even}} |  P^{a,c}_{b} (\vec{\kappa}_{i}) | \notag \\ \label{haha5}
& \quad-  |\partial_{t}^{m}\varphi_{i}(t,0)|^{2}\sum_{\substack{[[a,b]]\leq [[10,	2]]\\c\leq 5\\b \, \, even}}|P^{a,c}_{b} (\vec{\kappa}_{i}) |,
\end{align}
where we have used the induction hypothethis \eqref{IndHyp3} in the last step.
Similarly
\begin{align}\label{relSmk}
&|\nabla_{t} S_{m+1,i}|^{2} \geq \frac{1}{2}
|\nabla_{s}^{4(m+1)+2} \vec{\kappa}_{i}|^{2}  \\
& \; -
\sum_{\substack{[[a,b]]\leq [[8(m+1),6]]\\c\leq 4(m+1)\\b \, \, even}}\hspace{-.2cm}|P^{a,c}_{b} (\vec{\kappa}_{i}) |
 -  \sum_{\beta \in S_{4m+7}^{m}} \prod_{l=0}^{m} | \partial_{t}^{l}\varphi_{i}(t,0)|^{2\beta_{l}}\hspace{-.2cm} \sum_{\substack{[[a,b]]\leq [[8(m+1)+4-2|\beta|,	2]]\\c\leq 4(m+1)+2-|\beta|\\b \, \, even}} \hspace{-.2cm}| P^{a,c}_{b} (\vec{\kappa}_{i})| , \notag\\
& \quad  \geq \frac{1}{2}
|\nabla_{s}^{4(m+1)+2} \vec{\kappa}_{i}|^{2}  \notag  \\
& \quad -
\sum_{\substack{[[a,b]]\leq [[8(m+1),6]]\\c\leq 4(m+1)\\b \, \, even}}|P^{a,c}_{b} (\vec{\kappa}_{i}) |
 -  |\partial_{t}^{m}\varphi_{i}(t,0)|^{2}\sum_{\substack{[[a,b]]\leq [[6,	2]]\\c\leq 3\\b \, \, even}}|P^{a,c}_{b} (\vec{\kappa}_{i}) | , \notag
\end{align}
and again with Lemma \ref{lem:behvarphi}
\begin{align*}
|(\nabla_{t}+\nabla_{s}^{4}) \nabla_{t} S_{m+1,i}| &\leq \sum_{\substack{[[a,b]]\leq [[4m+8,3]]\\c\leq 4m+8\\b \, \, odd}}|P^{a,c}_{b} (\vec{\kappa}_{i})| \\
& \quad
 +  \sum_{\beta \in S_{4m+7}^{m}} \prod_{l=0}^{m} (1+ |\partial_{t}^{l}\varphi_{i}(t,0)|)^{\beta_{l}} \hspace{-.2cm}\sum_{\substack{[[a,b]]\leq [[4(m+1)+6-|\beta|,	1]]\\c\leq 4(m+1)+6-|\beta|\\b \, \, odd}} \hspace{-.2cm}P^{a,c}_{b} (\vec{\kappa}_{i}) \\
 &\quad +  \sum_{\beta \in S_{4(m+2)+3}^{m+1}} \prod_{l=0}^{m+1} | \varphi_{i}^{(l)}|^{\beta_{l}}\sum_{\substack{[[a,b]]\leq [[4(m+2)+2-|\beta|,	1]]\\c\leq 4(m+2)+2-|\beta|\\b \, \, odd}}|P^{a,c}_{b} (\vec{\kappa}_{i})|\\
 &\leq \sum_{\substack{[[a,b]]\leq [[4m+8,3]]\\c\leq 4m+8\\b \, \, odd}}|P^{a,c}_{b} (\vec{\kappa}_{i})| 
 + |\partial_{t}^{m}\varphi_{i}(t,0)|\sum_{\substack{[[a,b]]\leq [[7,	1]]\\c\leq 7\\b \, \, odd}}P^{a,c}_{b} (\vec{\kappa}_{i}) \\
 & + |\partial^{m+1}_{t}\varphi_{i} (t,0)| \hspace{-.2cm}\sum_{\substack{[[a,b]]\leq [[3,	1]]\\c\leq 3\\b \, \, odd}}\hspace{-.4cm}|P^{a,c}_{b} (\vec{\kappa}_{i})| + |\partial_{t}^{m}\varphi_{i}(t,0)|^{2}\hspace{-.2cm}\sum_{\substack{[[a,b]]\leq [[0,	1]]\\c\leq 0\\b \, \, odd}}\hspace{-.4cm}|P^{a,c}_{b} (\vec{\kappa}_{i})|
\end{align*}
where for the second inequality we have used the induction hypothethis \eqref{IndHyp3} and where the term mutiplying $|\partial_{t}^{m}\varphi_{i}(t,0)|^{2}$ actually appears only when $m=1$.

From \eqref{39+3} we obtain adding the term $\frac12 \int_{I} |\nabla_t S_{m+1,i}|^{2} ds$ and using the expressions derived above (and recalling \eqref{phiabsch} and the fact that  $m \geq 1$)
\allowdisplaybreaks{\begin{align}\label{39+4}
& \frac{d}{dt} \sum_{i=1}^3  \frac{1}{2}\int_{I} |\nabla_t S_{m+1,i}|^{2} ds + \sum_{i=1}^3  \frac{1}{2}\int_{I} |\nabla_t S_{m+1,i}|^{2} ds+\sum_{i=1}^3 \frac{1}{2} \int_{I}
|\nabla_{s}^{4m +8} \vec{\kappa}_{i} |^{2 } ds \\ \nonumber
& \leq - \sum_{i=1}^3 [ \langle \nabla_t S_{m+1,i}, \nabla_s^3 \nabla_t S_{m+1,i} \rangle ]_0^1+\sum_{i=1}^3  [ \langle \nabla_s \nabla_t S_{m+1,i}, \nabla_s^2 \nabla_t S_{m+1,i} \rangle ]_0^1
\\  \nonumber
& \; + \sum_{i=1}^3 \int_{I} \sum_{\substack{[[a,b]]\leq [[8m +14,	4]]\\c\leq 4m +8\\b \, \, even}}\hspace{-.2cm}|P^{a,c}_{b} (\vec{\kappa}_{i})|ds  
+\sum_{i=1}^3 |\partial_{t}^{m}\varphi_{i}(t,0)|\int_{I} \sum_{\substack{[[a,b]]\leq [[4m +13,	2]]\\c\leq  4m +8\\b \, \, even}}\hspace{-.2cm}|P^{a,c}_{b} (\vec{\kappa}_{i})|ds \notag \\
& \; +\sum_{i=1}^3 |\partial_{t}^{m}\varphi_{i}(t,0)|^{2} \int_{I}\hspace{-.1cm}\sum_{\substack{[[a,b]]\leq [[10,	2]]\\c\leq 7\\b \, \, even}}\hspace{-.45cm}|P^{a,c}_{b} (\vec{\kappa}_{i})|ds  +\sum_{i=1}^3 |\partial_{t}^{m+1}\varphi_{i}(t,0)| \int_{I} \hspace{-.1cm}\sum_{\substack{[[a,b]]\leq [[4m+9,	2]]\\c\leq 4m+6\\b \, \, even}}\hspace{-.7cm}|P^{a,c}_{b} (\vec{\kappa}_{i})|
ds  \notag\\
& \; +\sum_{i=1}^3  |\partial_{t}^{m+1}\varphi_{i}(t,0)||\partial_{t}^{m}\varphi_{i}(t,0)| \int_{I}\sum_{\substack{[[a,b]]\leq [[6,	2]]\\c\leq 3\\b \, \, even}}|P^{a,c}_{b} (\vec{\kappa}_{i})|
ds  \notag\\
& \; +\sum_{i=1}^3  |\partial_{t}^{m}\varphi_{i}(t,0)|^{2}\int_{I} \sum_{\substack{[[a,b]]\leq [[4m+6,	2]]\\c\leq 4m+6\\b \, \, even}}\hspace{-.2cm}|P^{a,c}_{b} (\vec{\kappa}_{i})| + |\partial_{t}^{m}\varphi_{i}(t,0)|^{3} \int_{I} \sum_{\substack{[[a,b]]\leq [[3,	2]]\\c\leq 3\\b \, \, even}}\hspace{-.2cm}|P^{a,c}_{b} (\vec{\kappa}_{i})| 
ds , \notag
\end{align}}
where the last two integrals appear only if $m=1$.
We first estimate the integral terms.
By Lemma~\ref{lemineqshsum} with $\ell=4m+6$ we obtain
\begin{align*}
\sum_{i=1}^3  \int_{I} \sum_{\substack{[[a,b]] \leq [[8m+14,4]]\\c\leq 4m+8,\ b \, even}} |P^{a,c}_{b} (\vec{\kappa}_{i}) |
 \leq  \varepsilon \sum_{i=1}^3  \int_0^1 |\nabla_s^{4m+8} \vec{\kappa}_i|^2 \, ds + C_{\epsilon}(\lambda_{i}, f_{i,0}, \mathcal{L}[f_i], \mathcal{E}_{\lambda_{i}}(f_{i,0}))\,,
\end{align*}
with $\varepsilon \in (0,1)$. 
By the same interpolation result and \eqref{giomzero}, again with $\ell=4m+6$, we get
\begin{align*}
&\sum_{i=1}^3 |\partial_{t}^{m}\varphi_{i}(t,0)|\int_{I} \sum_{\substack{[[a,b]]\leq [[4m +13,	2]]\\c\leq  4m +8\\b \, \, even}}|P^{a,c}_{b} (\vec{\kappa}_{i})|ds \\
& \quad +\sum_{i=1}^3 |\partial_{t}^{m}\varphi_{i}(t,0)|^{2} \int_{I}\sum_{\substack{[[a,b]]\leq [[10,	2]]\\c\leq 7\\b \, \, even}}|P^{a,c}_{b} (\vec{\kappa}_{i})|ds \\
& \quad +\sum_{i=1}^3 |\partial_{t}^{m+1}\varphi_{i}(t,0)| \int_{I} \sum_{\substack{[[a,b]]\leq [[4m+9,	2]]\\c\leq 4m+6\\b \, \, even}}|P^{a,c}_{b} (\vec{\kappa}_{i})|
ds  \\
& \quad +\sum_{i=1}^3  |\partial_{t}^{m+1}\varphi_{i}(t,0)||\partial_{t}^{m}\varphi_{i}(t,0)| \int_{I}\sum_{\substack{[[a,b]]\leq [[6,	2]]\\c\leq 3\\b \, \, even}}|P^{a,c}_{b} (\vec{\kappa}_{i})|
ds \\
& \leq C(m,\delta) \sum_{i=1}^3 \sum_{k=1}^3  \Big(  \max\{1,\| \vec{\kappa}_k\|_{4m+8,2} \}^{\frac{(5+8m)}{2(4m+8)}} \max\{1,\| \vec{\kappa}_i\|_{4m+8,2}\}^{\frac{4m+13}{4m+8}}   \\
 & \; +    \max\{1,\| \vec{\kappa}_k\|_{4m+8,2} \}^{\frac{2(5+8m)}{2(4m+8)}} \max\{1,\| \vec{\kappa}_i\|_{4m+8,2}\}^{\frac{10}{4m+8}}  \\
 & \; +   \max\{1,\| \vec{\kappa}_k\|_{4m+8,2} \}^{\frac{(13+8m)}{2(4m+8)}} \max\{1,\| \vec{\kappa}_i\|_{4m+8,2}\}^{\frac{4m+9}{4m+8}}  \\
 & \; +   \sum_{l=1}^{3}  \max\{1,\| \vec{\kappa}_k\|_{4m+8,2} \}^{\frac{(5+8m)}{2(4m+8)}} \max\{1,\| \vec{\kappa}_l\|_{4m+8,2} \}^{\frac{(13+8m)}{2(4m+8)}} \max\{1,\| \vec{\kappa}_i\|_{4m+8,2}\}^{\frac{6}{4m+8}} \Big) \\
 &\leq \varepsilon \sum_{i=1}^3  \int_0^1 |\nabla_s^{4m+8} \vec{\kappa}_i|^2 \, ds + C_{\epsilon}(\delta,\lambda_{i}, f_{i,0}, \mathcal{L}[f_i], \mathcal{E}_{\lambda_{i}}(f_{i,0}))\,,
\end{align*}
where we have used  Corollary \ref{corinter}, the uniform bound on the lengths and Lemma~\ref{Younggeneral} in the last inequality.
The last two integral terms in \eqref{39+4} have to be evaluated  only when $m=1$. In this case we calculate with the same arguments as above and \eqref{IndHyp2}
\begin{align*}
&\sum_{i=1}^3  |\partial_{t}^{m}\varphi_{i}(t,0)|^{2}\int_{I} \sum_{\substack{[[a,b]]\leq [[4m+6,	2]]\\c\leq 4m+6\\b \, \, even}}|P^{a,c}_{b} (\vec{\kappa}_{i})| + |\partial_{t}^{m}\varphi_{i}(t,0)|^{3} \int_{I} \sum_{\substack{[[a,b]]\leq [[3,	2]]\\c\leq 3\\b \, \, even}}|P^{a,c}_{b} (\vec{\kappa}_{i})| 
ds  \\
& \leq C(m,\delta)\sum_{i=1}^3 \sum_{k=1}^3  \max\{1,\| \vec{\kappa}_k\|_{4m+8,2} \}^{\frac{2(5+8m)}{2(4m+8)}} \max\{1,\| \vec{\kappa}_i\|_{4m+8,2}\}^{\frac{4m+6}{4m+8}}\\
& \quad + C(m,\delta)\sum_{i=1}^3 \sum_{k=1}^3  \max\{1,\| \vec{\kappa}_k\|_{4m+8,2} \}^{\frac{3(5+8m)}{2(4m+8)}} \max\{1,\| \vec{\kappa}_i\|_{4m+2,2}\}^{\frac{3}{4m+2}}\\
& \leq C(m,\delta)\sum_{i=1}^3 \sum_{k=1}^3  \max\{1,\| \vec{\kappa}_k\|_{4m+8,2} \}^{\frac{2(5+8m)}{2(4m+8)}} \max\{1,\| \vec{\kappa}_i\|_{4m+8,2}\}^{\frac{4m+6}{4m+8}}\\
& \quad + C(m,\delta) \sum_{i=1}^3 \sum_{k=1}^3  \max\{1,\| \vec{\kappa}_k\|_{4m+8,2} \}^{\frac{3(5+8m)}{2(4m+8)}} \\
&\leq \varepsilon \sum_{i=1}^3  \int_0^1 |\nabla_s^{4m+8} \vec{\kappa}_i|^2 \, ds + C_{\epsilon}(\delta,\lambda_{i}, f_{i,0}, \mathcal{L}[f_i], \mathcal{E}_{\lambda_{i}}(f_{i,0}))\,,
\end{align*}
where, we have appplied  Lemma~\ref{Younggeneral} and  Corollary \ref{corinter} in the last step since by  $m=1$ we have that $\frac{3(5+8m)}{2(4m+8)} <2$ and similarly $2(5+8m) +2(4m+6)< 4(4m+8)$.

It remains to treat the boundary terms.
We can write
\begin{align*}
- \sum_{i=1}^3 [ \langle \nabla_t S_{m+1,i}, \nabla_s^3 \nabla_t S_{m+1,i} \rangle ]_0^1 = - \sum_{i=1}^3 [ \langle \partial_t S_{m+1,i}, \nabla_s^3 \nabla_t S_{m+1,i} \rangle ]_0^1.
\end{align*}
Since  $  \partial_t S_{m+1,i} = \partial_{t}^{m+1} f_{i}$ at the boundary points, and  here $\partial_{t}^{m+1} f_{i}=\partial_{t}^{m+1} f_{j}$, for $i,j=1,2,3$, we find
\begin{align*}
- \sum_{i=1}^3 [ \langle \nabla_t S_{m+1,i}, \nabla_s^3 \nabla_t S_{m+1,i} \rangle ]_0^1=  \langle  \partial_t S_{m+1,1}, 
\sum_{i=1}^{3}\nabla_s^3 \nabla_t S_{m+1,i} \rangle \Big|_{x=0}.
\end{align*}
Since from \eqref{haha5}
\begin{align*}
&\nabla_{s}^{3}(\nabla_{t} S_{m+1,i})
=(-1)^{m}\nabla_{s}^{4(m+1)+5} \vec{\kappa}_{i}\\
& \quad+
\sum_{\substack{[[a,b]]\leq [[4(m+1)+3,3]]\\c\leq 4(m+1)+3\\b \, \, odd}}P^{a,c}_{b} (\vec{\kappa}_{i}) 
 +  \nabla_{s}^{3} \Big(\sum_{\beta \in S_{4(m+1)+3}^{m}} \prod_{l=0}^{m} ( \varphi_{i}^{(l)})^{\beta_{l}}\sum_{\substack{[[a,b]]\leq [[4(m+1)+2-|\beta|,	1]]\\c\leq 4(m+1)+2-|\beta|\\b \, \, odd}}\hspace{-1cm}P^{a,c}_{b} (\vec{\kappa}_{i}) \Big),
\end{align*}
we get using \eqref{b5+4m}
\begin{align*}
& \sum_{i=1}^3 \nabla_{s}^{3}(\nabla_{t} S_{m+1,i})
=
\sum_{i=1}^3 \Big(
 \sum_{\substack{[[a,b]]\leq [[5+4(m+1) -2,3]]\\c\leq 5+4(m+1)-2\\b \, \, odd}}P^{a,c}_{b} (\vec{\kappa}_{i}) 
 \\
 & \; +\sum_{\beta \in S_{4+4(m+1)}^{m+1}} \prod_{l=0}^{m+1} ( \varphi_{i}^{(l)})^{\beta_{l}}\sum_{\substack{[[a,b]]\leq [[5+4(m+1)-|\beta|,1]]\\c\leq 5+4(m+1)-|\beta|\\b \, \, odd}}P^{a,c}_{b} (\vec{\kappa}_{i}) 
 \\
 & \; +  \partial_{s}f_{i} \Big[    \sum_{\substack{[[a,b]]\leq [[4+4(m+1),2]]\\c\leq 3+4(m+1)\\b \, \, even}} 
 P^{a,c}_{b} (\vec{\kappa}_{i}) 
 +  \sum_{\beta \in S_{ 4(m+1)}^{m}} \prod_{l=0}^{m} ( \varphi_{i}^{(l)})^{\beta_{l}}\sum_{\substack{[[a,b]]\leq [[4+4(m+1)-|\beta|,2]]\\c\leq 3+4(m+1)-|\beta|\\b \, \, even}}\hspace{-1cm}P^{a,c}_{b} (\vec{\kappa}_{i}) 
  \Big ]    \Big) 
\\
& \;+\sum_{i=1}^3 \Big(
\sum_{\substack{[[a,b]]\leq [[4(m+1)+3,3]]\\c\leq 4(m+1)+3\\b \, \, odd}}\hspace{-1cm}P^{a,c}_{b} (\vec{\kappa}_{i}) 
 +  \nabla_{s}^{3} \Big(\sum_{\beta \in S_{4(m+1)+3}^{m}} \prod_{l=0}^{m} ( \varphi_{i}^{(l)})^{\beta_{l}}\sum_{\substack{[[a,b]]\leq [[4(m+1)+2-|\beta|,	1]]\\c\leq 4(m+1)+2-|\beta|\\b \, \, odd}}\hspace{-1cm}P^{a,c}_{b} (\vec{\kappa}_{i}) \Big)\Big). 
\end{align*}
Therefore we infer using \eqref{IndHyp3} and Lemma~\ref{lem:behvarphi}
\begin{align*}
& \Big|\sum_{i=1}^3 \nabla_{s}^{3}(\nabla_{t} S_{m+1,i}) \Big|_{x=0}
\leq \sum_{i=1}^3  \Big( \Big|\sum_{\substack{[[a,b]]\leq [[4(m+1) +3,3]]\\c\leq 4(m+1)+3\\b \, \, odd}}P^{a,c}_{b} (\vec{\kappa}_{i}) \Big| 
\\
& \quad+ |\partial_{t}^{m+1} \varphi_{i}(t,0)|\sum_{\substack{[[a,b]]\leq [[2,1]]\\c\leq 2\\b \, \, odd}}|P^{a,c}_{b} (\vec{\kappa}_{i})| + |\partial_{t}^{m} \varphi_{i}(t,0)|\sum_{\substack{[[a,b]]\leq [[6,1]]\\c\leq 6\\b \, \, odd}}|P^{a,c}_{b} (\vec{\kappa}_{i})|\\
& \quad + \sum_{\substack{[[a,b]]\leq [[4+4(m+1),2]]\\c\leq 3+4(m+1)\\b \, \, even}} 
 |P^{a,c}_{b} (\vec{\kappa}_{i})| + |\partial_{t}^{m} \varphi_{i}(t,0)|\sum_{\substack{[[a,b]]\leq [[5,2]]\\c\leq 5\\b \, \, even}} 
 |P^{a,c}_{b} (\vec{\kappa}_{i})|
\Big).
\end{align*}
We write
\begin{align*}   \sum_{\substack{[[a,b]]\leq [[4(m+1) +3,3]]\\c\leq 4(m+1)+3\\b \, \, odd}}\hspace{-.25cm}P^{a,c}_{b} (\vec{\kappa}_{i})& 
\sum_{\substack{[[a,b]]\leq [[4(m+1) +3,3]]\\c\leq 4(m+1)+3\\b \, \, odd, b \ne 1}}\hspace{-.25cm}P^{a,c}_{b} (\vec{\kappa}_{i}) +\sum_{\substack{[[a,b]]\leq [[4(m+1) +3,3]]\\c\leq 4(m+1)+3\\b=1}}\hspace{-.25cm}P^{a,c}_{b} (\vec{\kappa}_{i})
  \end{align*}
and estimate by \eqref{embed}
\begin{align*} 
\sum_{\substack{[[a,b]]\leq [[4(m+1) +3,3]]\\c\leq 4(m+1)+3\\b \, \, odd, b \ne 1}}P^{a,c}_{b} (\vec{\kappa}_{i})  & \leq \int_{I} \sum_{\substack{[[a,b]]\leq [[4(m+1) +4,3]]\\c\leq 4(m+1)+4\\b \geq 2 }}|P^{a,c}_{b} (\vec{\kappa}_{i}) | ds. 
\end{align*}
Next, using Lemma~\ref{lem:trickbdry}, Lemma~\ref{lemineqshsum} with $\ell=4m+6$, and \eqref{giomzero} we derive
\begin{align*}
& \Big|\sum_{i=1}^3 \nabla_{s}^{3}(\nabla_{t} S_{m+1,i}) \Big|_{x=0}
\leq  \sum_{i=1}^3 \Big(
\int_{I} \sum_{\substack{[[a,b]]\leq [[4(m+1) +4,3]]\\c\leq 4(m+1)+4\\b \geq 2 }}|P^{a,c}_{b} (\vec{\kappa}_{i}) | ds \\
& \quad + \big(\int_{I}  \sum_{\substack{[[a,b]]\leq [[8(m+1) +7,2]]\\c\leq 4(m+1)+4\\b \, even }}\hspace{-.7cm}|P^{a,c}_{b} (\vec{\kappa}_{i}) |   ds \big)^{\frac12} + |\partial_{t}^{m+1} \varphi_{i}(t,0)| \big(\int_{I}  \sum_{\substack{[[a,b]]\leq [[5,2]]\\c\leq 3\\b \, even }}\hspace{-.25cm}|P^{a,c}_{b} (\vec{\kappa}_{i}) |   ds \big)^{\frac12}\\
& \quad +  |\partial_{t}^{m} \varphi_{i}(t,0)| \big(\int_{I}  \sum_{\substack{[[a,b]]\leq [[13,2]]\\c\leq 7\\b \, even }}\hspace{-.25cm}|P^{a,c}_{b} (\vec{\kappa}_{i}) |   ds \big)^{\frac12}  + \int_{I} \sum_{\substack{[[a,b]]\leq [[4(m+1) +5,2]]\\c\leq 4(m+1)+4\\b \, even }}\hspace{-.25cm}|P^{a,c}_{b} (\vec{\kappa}_{i}) | ds\\
& \quad + |\partial_{t}^{m} \varphi_{i}(t,0)| \int_{I} \sum_{\substack{[[a,b]]\leq [[6,2]]\\c\leq 6 \\b \, even }}|P^{a,c}_{b} (\vec{\kappa}_{i}) | ds \Big)\\
& \leq
 \sum_{i=1}^3  \Big( \max\{1,\| \vec{\kappa}_i\|_{4m+8,2}\}^{\frac{8m+18}{2(4m+8)}} \\
 &\quad+ C(m ,\delta)\sum_{k=1}^{3}   \max\{1,\| \vec{\kappa}_k\|_{4m+8,2}\}^{\frac{8m+13}{2(4m+8)}}    \max\{1,\| \vec{\kappa}_i\|_{4m+8,2}\}^{\frac{5}{2(4m+8)}}\\
 & \quad + C(m ,\delta)\sum_{k=1}^{3}   \max\{1,\| \vec{\kappa}_k\|_{4m+8,2}\}^{\frac{8m+5}{2(4m+8)}}    \max\{1,\| \vec{\kappa}_i\|_{4m+8,2}\}^{\frac{13}{2(4m+8)}}
\Big).
\end{align*}
Next we observe that by \eqref{SM+1} we have
\begin{align*}
\partial_{t}S_{m+1,i}&= S_{m+2,i}=\partial_{t}^{m+1} (- \nabla_{s}^{2} \vec{\kappa}_{i} + \varphi_{i} \partial_{s} f_{i})\\
&=(-1)^{m} \nabla_{s}^{4(m+1)+2} \vec{\kappa}_{i}+
\sum_{\substack{[[a,b]]\leq [[4(m+1),3]]\\c\leq 4(m+1)\\b \, \, odd}}P^{a,c}_{b} (\vec{\kappa}_{i}) \\& \quad 
 +  \sum_{\beta \in S_{4(m+1)+3}^{m}} \prod_{l=0}^{m} ( \varphi_{i}^{(l)})^{\beta_{l}}\sum_{\substack{[[a,b]]\leq [[4(m+1)+2-|\beta|,	1]]\\c\leq 4(m+1)+2-|\beta|\\b \, \, odd}}P^{a,c}_{b} (\vec{\kappa}_{i})  \\
& \quad + \partial_{s} f_{i} \Big(
\partial_{t}^{m+1}\varphi_{i}(t,0)
 +\sum_{\substack{[[a,b]]\leq [[4(m+1)+1,2]]\\c\leq 4(m+1)-1\\b \, \, even}}P^{a,c}_{b} (\vec{\kappa}_{i})  \notag\\
 & \quad 
+ \sum_{\beta \in S_{4(m+1)}^{m}} \prod_{l=0}^{m} ( \varphi_{i}^{(l)})^{\beta_{l}}\sum_{\substack{[[a,b]]\leq [[4(m+1)+1-|\beta|,	2]]\\c\leq 4(m+1)+1-|\beta|\\b \, \, even}}P^{a,c}_{b} (\vec{\kappa}_{i}) \Big).  
\end{align*}
so that by Lemma~\ref{lem:behvarphi}, Lemma \ref{lem:trickbdry}, Lemma~\ref{lemineqshsum} with $\ell=4m+6$, and \eqref{giomzero}, we infer
\begin{align*}
|\partial_{t}S_{m+1,i}(0)| &\leq \big(\int_{I}  \sum_{\substack{[[a,b]]\leq [[8m+13,2]]\\c\leq 4m+7\\b \, even }}\hspace{-.55cm}|P^{a,c}_{b} (\vec{\kappa}_{i}) |   ds \big)^{\frac12} + |\partial_{t}^{m} \varphi_{i}(t,0)| \big(\int_{I}  \sum_{\substack{[[a,b]]\leq [[7,2]]\\c\leq 7\\b \, even }}\hspace{-.35cm}|P^{a,c}_{b} (\vec{\kappa}_{i}) |   ds \big)^{\frac12} \\
& \; + |\partial_{t}^{m+1} \varphi_{i}(t,0)| + \int_{I}  \sum_{\substack{[[a,b]]\leq [[4m+6,2]]\\c\leq 4m+4\\b \, even }}|P^{a,c}_{b} (\vec{\kappa}_{i}) |   ds \\
& \; + |\partial_{t}^{m} \varphi_{i}(t,0)| \int_{I}  \sum_{\substack{[[a,b]]\leq [[3,2]]\\c\leq 3\\b \, even }}|P^{a,c}_{b} (\vec{\kappa}_{i}) |   ds \\
& \leq  \max\{1,\| \vec{\kappa}_i\|_{4m+8,2}\}^{\frac{8m+13}{2(4m+8)}} \\
 &\,+ C(m ,\delta)\sum_{k=1}^{3}   \max\{1,\| \vec{\kappa}_k\|_{4m+8,2}\}^{\frac{8m+13}{2(4m+8)}}  +  \max\{1,\| \vec{\kappa}_i\|_{4m+8,2}\}^{\frac{8m+12}{2(4m+8)}}\\
 & \; + C(m ,\delta)\sum_{k=1}^{3}   \max\{1,\| \vec{\kappa}_k\|_{4m+8,2}\}^{\frac{8m+5}{2(4m+8)}}    \max\{1,\| \vec{\kappa}_i\|_{4m+8,2}\}^{\frac{7}{2(4m+8)}}.
\end{align*}
Putting all estimates together, using Lemma~\ref{Younggeneral} and Corollary \ref{corinter} we obtain
\begin{align*}
|- \sum_{i=1}^3 [ \langle \nabla_t S_{m+1,i}, \nabla_s^3 \nabla_t S_{m+1,i} \rangle ]_0^1 | \leq \varepsilon \sum_{i=1}^3  \int_0^1 |\nabla_s^{4m+8} \vec{\kappa}_i|^2 \, ds + C_{\epsilon},
\end{align*}
with $C_{\epsilon}=C_{\epsilon}(\delta, \lambda_{i}, f_{i,0}, \mathcal{L}[f_i],\mathcal{E}_{\lambda_{i}}(f_{i,0}))$. 

It remains to evaluate the last boundary term 
$\sum_{i=1}^3  [ \langle \nabla_s \nabla_t S_{m+1,i}, \nabla_s^2 \nabla_t S_{m+1,i} \rangle ]_0^1. $
First of all note that by \eqref{haha3}, Lemma~\ref{lem:behvarphi}, Lemma~\ref{lem:trickbdry}, Lemma~\ref{lemineqshsum} with $\ell=4m+6$, and \eqref{giomzero},
we have for $x \in I$
\begin{align*}
&|\nabla_{s}\nabla_{t} S_{m+1,i}|
 = |(-1)^{m+2}\nabla_{s}^{4(m+1)+3} \vec{\kappa}_{i}\\
 & \quad +
\sum_{\substack{[[a,b]]\leq [[4(m+1) +1,3]]\\c\leq 4(m+1)+1\\b \, \, odd}}\hspace{-.5cm}P^{a,c}_{b} (\vec{\kappa}_{i}) 
 +  \sum_{\beta \in S_{4(m+1)+3}^{m}}\partial_{s} ( \prod_{l=0}^{m} ( \varphi_{i}^{(l)})^{\beta_{l}} )\sum_{\substack{[[a,b]]\leq [[4(m+1)+2-|\beta|,	1]]\\c\leq 4(m+1)+2-|\beta|\\b \, \, odd}}\hspace{-.5cm} P^{a,c}_{b} (\vec{\kappa}_{i}) \\
& \quad +\sum_{\beta \in S_{4(m+1)+3}^{m}} \prod_{l=0}^{m} ( \varphi_{i}^{(l)})^{\beta_{l}} \sum_{\substack{[[a,b]]\leq [[4(m+1)+3-|\beta|,	1]]\\c\leq 4(m+1)+3-|\beta|\\b \, \, odd}}P^{a,c}_{b} (\vec{\kappa}_{i}) \, \, | \\
& \leq \sum_{\substack{[[a,b]]\leq [[4(m+1) +3,1]]\\c\leq 4(m+1)+3\\b \, \, odd}} \hspace{-.6cm}|P^{a,c}_{b} (\vec{\kappa}_{i})| 
+\sum_{\beta \in S_{4(m+1)+3}^{m}} \prod_{l=0}^{m} |\partial_{t}^{l} \varphi_{i}(t,0)|^{\beta_{l}} \sum_{\substack{[[a,b]]\leq [[4(m+1)+3-|\beta|,	1]]\\c\leq 4(m+1)+3-|\beta|\\b \, \, odd}}\hspace{-.6cm}|P^{a,c}_{b} (\vec{\kappa}_{i})|\\
& \leq \Big(\int_{I}\sum_{\substack{[[a,b]]\leq [[8(m+1) +7,2]]\\c\leq 4(m+1)+4\\b \, \, even}} |P^{a,c}_{b} (\vec{\kappa}_{i})| ds \Big) ^{\frac12}
+  |\partial_{t}^{m} \varphi_{i}(t,0)|  \Big( \int_{I} \sum_{\substack{[[a,b]]\leq [[9,	2]]\\c\leq 5\\b \, \, even}}|P^{a,c}_{b} (\vec{\kappa}_{i})| ds \Big )^{\frac12}\\
&  \leq  \max\{1,\| \vec{\kappa}_i\|_{4m+8,2}\}^{\frac{8m+15}{2(4m+8)}} \\
 &\quad+ C(m ,\delta)\sum_{k=1}^{3}   \max\{1,\| \vec{\kappa}_k\|_{4m+8,2}\}^{\frac{8m +5}{2(4m+8)}}   \max\{1,\| \vec{\kappa}_i\|_{4m+8,2}\}^{\frac{9}{2(4m+8)}}.
\end{align*}
Next, using \eqref{haha4}, \eqref{b4m} (to lower the order of the term $\nabla_{s}^{4(m+2)} \vec{\kappa}_{i}$), 
Lemma~\ref{lem:behvarphi}, Lemma~\ref{lem:trickbdry}, Lemma~\ref{lemineqshsum} with $\ell=4m+6$, and \eqref{giomzero}, we infer that at the boundary (thus for $x \in \{ 0, 1\}$) we have
\begin{align*} 
& |\nabla_{s}^{2}\nabla_{t} S_{m+1,i} |
 =
|\sum_{\substack{[[a,b]]\leq [[4m+6,3]]\\c\leq 4m+6\\b \, \, odd}}\hspace{-.5cm}P^{a,c}_{b} (\vec{\kappa}_{i}) 
 +\sum_{\beta \in S_{4(m+2)}^{m+1}} \prod_{l=0}^{m+1} ( \varphi_{i}^{(l)})^{\beta_{l}}\sum_{\substack{[[a,b]]\leq [[4m+8-|\beta|,	1]]\\c\leq 4m+8-|\beta|\\b \, \, odd}}\hspace{-.5cm}P^{a,c}_{b} (\vec{\kappa}_{i}) \\
& \qquad +  \sum_{\beta \in S_{4(m+1)+3}^{m}} \prod_{l=0}^{m} ( \varphi_{i}^{(l)})^{\beta_{l}}\sum_{\substack{[[a,b]]\leq [[4m+8-|\beta|,	1]]\\c\leq 4m+8-|\beta|\\b \, \, odd}}P^{a,c}_{b} (\vec{\kappa}_{i}) ,  \\
 & \qquad +\sum_{\beta \in S_{4(m+1)+3}^{m}} \partial_{s} \left(\prod_{l=0}^{m} ( \varphi_{i}^{(l)})^{\beta_{l}} \right)\sum_{\substack{[[a,b]]\leq [[4(m+1)+3-|\beta|,	1]]\\c\leq 4(m+1)+3-|\beta|\\b \, \, odd}}P^{a,c}_{b} (\vec{\kappa}_{i})  \\
 & \qquad +\sum_{\beta \in S_{4(m+1)+3}^{m}}\partial_{s}^{2} \left( \prod_{l=0}^{m} ( \varphi_{i}^{(l)})^{\beta_{l}} \right)\sum_{\substack{[[a,b]]\leq [[4(m+1)+2-|\beta|,	1]]\\c\leq 4(m+1)+2-|\beta|\\b \, \, odd}}P^{a,c}_{b} (\vec{\kappa}_{i})|\\
 & \leq \sum_{\substack{[[a,b]]\leq [[4m+6,3]]\\c\leq 4m+6\\b \, \, odd}}|P^{a,c}_{b} (\vec{\kappa}_{i}) |
 +\sum_{\beta \in S_{4(m+2)}^{m+1}} \prod_{l=0}^{m+1} |\partial_{t}^{l} \varphi_{i}(t,0)|^{\beta_{l}}\sum_{\substack{[[a,b]]\leq [[4m+8-|\beta|,	1]]\\c\leq 4m+8-|\beta|\\b \, \, odd}}|P^{a,c}_{b} (\vec{\kappa}_{i})| \\
 & \leq \sum_{\substack{[[a,b]]\leq [[4m+6,3]]\\c\leq 4m+6\\b \, \, odd}}|P^{a,c}_{b} (\vec{\kappa}_{i}) |
 +   |\partial_{t}^{m} \varphi_{i}(t,0)|  \sum_{\substack{[[a,b]]\leq [[5,	1]]\\c\leq 5\\b \, \, odd}} |P^{a,c}_{b} (\vec{\kappa}_{i})|  \\
& \quad + |\partial_{t}^{m+1} \varphi_{i}(t,0)|  \sum_{\substack{[[a,b]]\leq [[1,	1]]\\c\leq 1\\b \, \, odd}} |P^{a,c}_{b} (\vec{\kappa}_{i})|\\
 & \leq \Big( \int_{I} \sum_{\substack{[[a,b]]\leq [[8m+13,6]]\\c\leq 4m+7\\b \, \, even}}|P^{a,c}_{b} (\vec{\kappa}_{i}) | ds \Big )^{\frac12} + |\partial_{t}^{m} \varphi_{i}(t,0)| \Big( \int_{I} \sum_{\substack{[[a,b]]\leq [[11,2]]\\c\leq 6\\b \, \, even}}|P^{a,c}_{b} (\vec{\kappa}_{i}) | ds \Big )^{\frac12} \\
 &\quad + |\partial_{t}^{m+1} \varphi_{i}(t,0)| \Big( \int_{I} \sum_{\substack{[[a,b]]\leq [[3,2]]\\c\leq 2\\b \, \, even}}|P^{a,c}_{b} (\vec{\kappa}_{i}) | ds \Big )^{1/2}\\
 & \leq \max\{1,\| \vec{\kappa}_i\|_{4m+8,2}\}^{\frac{8m+15}{2(4m+8)}} \\
 &\quad+ C(m ,\delta)\sum_{k=1}^{3}   \max\{1,\| \vec{\kappa}_k\|_{4m+8,2}\}^{\frac{8m +5}{2(4m+8)}}   \max\{1,\| \vec{\kappa}_i\|_{4m+8,2}\}^{\frac{11}{2(4m+8)}}\\
 &\quad+ C(m ,\delta)\sum_{k=1}^{3}   \max\{1,\| \vec{\kappa}_k\|_{4m+8,2}\}^{\frac{8m +13}{2(4m+8)}}   \max\{1,\| \vec{\kappa}_i\|_{4m+8,2}\}^{\frac{3}{2(4m+8)}}.
\end{align*}
Putting the estimates together 
using Lemma~\ref{Younggeneral} and Corollary \ref{corinter} we obtain
\begin{align*}
|\sum_{i=1}^3 [ \langle \nabla_{s}\nabla_t S_{m+1,i}, \nabla_s^2 \nabla_t S_{m+1,i} \rangle ]_0^1 | \leq \varepsilon \sum_{i=1}^3  \int_0^1 |\nabla_s^{4m+8} \vec{\kappa}_i|^2 \, ds + C_{\epsilon},
\end{align*}
with $C_{\epsilon}=C_{\epsilon}(\delta , \lambda_{i}, f_{i,0}, \mathcal{L}[f_i],\mathcal{E}_{\lambda_{i}}(f_{i,0})$. Putting  the estimates together, obtained for the boundary terms and the integral terms in \eqref{39+4}, we can finally state
\begin{align*}
& \frac{d}{dt} \sum_{i=1}^3  \frac{1}{2}\int_{I} |\nabla_t S_{m+1,i}|^{2} ds + \sum_{i=1}^3  \frac{1}{2}\int_{I} |\nabla_t S_{m+1,i}|^{2} ds+ \sum_{i=1}^3 \frac{1}{2} \int_{I}
|\nabla_s^{4m+8} \vec{\kappa}_{i}|^{2 } ds \\ 
& \leq \varepsilon \sum_{i=1}^3  \int_0^1 |\nabla_s^{4m+8} \vec{\kappa}_i|^2 \, ds + C_{\epsilon}(\delta,\lambda_{i}, f_{i,0}, \mathcal{L}[f_i], \mathcal{E}_{\lambda_{i}}(f_{i,0}))\,.
\end{align*}
Choosing $\varepsilon$ appropriately and applying Gronwall Lemma give that $$\sup_{(0,T)}\sum_{i=1}^3  \int_{I} |\nabla_t S_{m+1,i}|^{2} ds  \leq 
C=C(\delta, \lambda_{j},f_{j,0}, \mathcal{L}[f_j],\mathcal{E}_{\lambda_{j}}(f_{j,0}))
 \text{ for } j=1,2,3.$$
Together with \eqref{relSmk}, Lemma~\ref{lemineqshsum} and \eqref{giomzero} with $\ell=4m+4$, Corollary~\ref{corinter}, the uniform bound on the lengths and Young inequality, standard arguments yield
$$\sup_{(0,T)}\sum_{i=1}^3  \int_{I} |\nabla_s ^{4m+6}\vec{\kappa}_{i}|^{2} ds  \leq C(\delta, \lambda_{j},f_{j,0}, \mathcal{L}[f_j],\mathcal{E}_{\lambda_{j}}(f_{j,0}))
 \text{ for } j=1,2,3 ,$$
and also
$$ \sup_{(0,T)}\sum_{i=1}^3 \| \partial_{t}^{m}\varphi_{i}(t,\cdot)\|_{L^{\infty}(I)}  \leq C(\delta, \lambda_{j},f_{j,0}, \mathcal{L}[f_j],\mathcal{E}_{\lambda_{j}}(f_{j,0}))
 \text{ for } j=1,2,3 , $$
by \eqref{giomzero} with $\ell=4m+4$,  and Lemma \ref{lem:behvarphi}.

\bigskip

\noindent \textbf{Fourth Step}: \underline{Long-time existence}
\smallskip

First of all we show that for all $t \in (0,T)$
 \begin{align*} \sum_{i=1}^{3}\|\partial_{x}^{m} \vec{\kappa}_{i}\|_{C^{0}(\bar{I})} \leq C&=C(m,\delta, n, T,\mathcal{L}[f_j], \mathcal{E}_{\lambda}(f_{j,0}), f_{j,0}, \lambda_{j}) \text{ with } j=1,2,3. \end{align*}
From the previous step, Lemma~\ref{Qlemma}, embedding inequalities (see for instance the first inequality in \eqref{embed}) and the fact that the length of the curves remains uniformly bounded along the flow we can state that
\begin{equation}\label{sept1}
\sum_{i=1}^{3}\| \partial_{s}^{m} \vec{\kappa}_{i} \|_{C^{0}(\bar{I})} 
\leq C(m,\delta, n,\mathcal{L}[f_j], \mathcal{E}_{\lambda}(f_{j,0}), f_{j,0}, \lambda_{j}) \text{ with } j=1,2,3,
\end{equation}
for any $m\in \mathbb{N}$ and $ t \in (0,T)$. From now on the proof follows most of the arguments depicted in \cite[\S~5 (Step seventh onwards)]{DP} with just minimal changes due to the presence of the tangential component that we point out. For the sake of completeness we sketch here again the main ideas.
In the following let $\gamma_{i}:= |\partial_x f_{i}|$. By induction it can be proven that for any function $h:\bar{I} \to \R$ or vector field $h:\bar{I} \to \R^{n}$, and for any $m \in \mathbb{N}$
\begin{equation}\label{sept4}
\partial_x^m h= \gamma^m \partial_s^m h + \sum_{j=1}^{m-1} P_{m-1}(\gamma, .. , \partial_x^{m-j} \gamma) \partial_s^j h \, ,
\end{equation}
with $P_{m-1}$ a polynomial of degree at most $m-1$. A bound on $\| \partial_{x}^{\ell} \vec{\kappa}_{i}\|_{C^{0}(\bar{I})}$ follows from \eqref{sept4} taking 
$h= \vec{\kappa}_{i}$ and from bounds on  $\| \partial_{s}^{\ell} \vec{\kappa}_{i}\|_{C^{0} (\bar{I})}$ (see \eqref{sept1}) and on $\|\partial_x^\ell \gamma_{i}\|_{C^{0} (\bar{I})}$. 
Thus it remains  to estimate $\|\partial_x^\ell \gamma_{i}\|_{C^{0} (\bar{I})}$ for $\ell \in \mathbb{N}_0$. 
We start by showing that $\gamma_{i} = |\partial_x f_{i}|$ is uniformly bounded from above and below. Upon recalling \eqref{a} we see that each function $\gamma_{i}$, $i=1,2,3$, satisfies the following parabolic equation
\begin{equation}\label{sept2}
\partial_t \gamma_{i}  = (\partial_{s} \varphi_{i}   - \langle \vec{\kappa}_{i}, \vec{V}_{i} \rangle)  \gamma_{i} \, .
\end{equation}
By regularity of the initial datum we have that $1/c_{0} \leq \gamma_{i} (0) \leq c_{0} $ for some positive $c_{0}$.
From the estimates given in \eqref{sept1}, \eqref{Eigvarphi}, and the uniform estimates for the tangential components and for the lengths of the curves, it follows that  the coefficients $\|\partial_{s} \varphi_{i} \|_{C^{0} (\bar{I})} +\| \langle \vec{\kappa}_{i}, \vec{V}_{i} \rangle\|_{C^{0} (\bar{I})}$  in \eqref{sept2} are uniformly bounded and hence we  infer that $1/C \leq \gamma_{i} \leq C$, with $C$ having the same dependencies as the constant in~\eqref{sept1} as well as~$T$.
In order to prove bounds on $\partial_x^m \gamma_{i}$ we proceed by induction. Let us assume that we have shown
\begin{equation}\label{sept3}
\| \partial_x^j \gamma_{i} \|_{C^{0} (\bar{I})} \leq C(m,\delta, n, T,\mathcal{L}[f_k],\mathcal{E}_{\lambda}(f_{k,0}), f_{k,0}, \lambda_{k}) \text{ with } k=1,2,3,
 \, \mbox{ for  } 0 \leq j \leq m \, 
\end{equation}
$i=1,2,3,$ and  $m \in \mathbb{N}_{0}$.  
Choosing $h=\langle \vec{\kappa}_{i}, \vec{V}_{i} \rangle$ in \eqref{sept4}, the induction assumption and \eqref{sept1} yield that 
\begin{equation}\label{sept5}
\| \partial_x^{i} \langle \vec{\kappa}_{q}, \vec{V}_{q} \rangle \|_{C^{0} (\bar{I})} \leq C(m,\delta, n, T,\mathcal{L}[f_j], \mathcal{E}_{\lambda}(f_{j,0}), f_{j,0}, \lambda_{j}) \text{ with } j=1,2,3,
\end{equation}
for all $ 0 \leq i \leq m+1 $, $q=1,2,3$. 
Differentiating \eqref{sept2} $(m+1)$-times with respect to $x$, and recalling that $\partial_{x}(\partial_{s} \varphi_{i})=0$ by \eqref{Eigvarphi}, we find
\begin{equation*}
\partial_t \partial_x^{m+1} \gamma_{q} = - \langle \vec{\kappa}_{q}, \vec{V}_{q} \rangle \partial_x^{m+1} \gamma_{q} - \sum_{\substack{i+j=m+1\\j \leq m}} c(i,j,m,q) \partial_x^{i}(\langle \vec{\kappa}_{q}, \vec{V}_{q} \rangle) \partial_x^j \gamma_{q}\, , 
\end{equation*}
for some coefficients $c(i,j,m,q)$ and $q=1,2,3$. Together
 with \eqref{sept3}, \eqref{sept5} we derive
\begin{equation*}
\big| \sum_{\substack{i+j=m+1\\j \leq m}} c(i,j,m,q) \partial_x^{i}(\langle \vec{\kappa}_{q}, \vec{V}_{q} \rangle) \partial_x^j \gamma_{q} \big | \leq  C \, , 
\end{equation*}
with $C=C(m,\delta, n, T,\mathcal{L}[f_j], \mathcal{E}_{\lambda}(f_{j,0}), f_{j,0}, \lambda_{j}$,  with $j=1,2,3$, which implies
\begin{equation*}
\| \partial_x^{m+1} \gamma_{i} \|_{C^{0} (\bar{I})} \leq C(m,\delta, n, T,\mathcal{L}[f_j], \mathcal{E}_{\lambda}(f_{j,0}), f_{j,0}, \lambda_{j}) \text{ with } j=1,2,3,  
\end{equation*}
for $i=1,2,3$.
Next note that \eqref{sept1} implies 
$$\sum_{i=1}^{3} \|  \partial_{s}^{m} \vec{V}_{i} \|_{C^{0}(\bar{I})} \leq C(m,\delta, n, \mathcal{L}[f_j], \mathcal{E}_{\lambda}(f_{j,0}), f_{j,0}, \lambda_{j}) \text{ with } j=1,2,3,
$$ which in turns gives uniform estimates for $\|\partial_{x}^{m}\vec{V}_{i}\|_{C^{0}(\bar{I})}$, $i=1,2,3$, in view of \eqref{sept4} and the bounds for the length elements and its derivatives.

Finally, the uniform $C^{0}$-bounds on the curvature $\vec{\kappa}_{i}$, the velocity $\vec{V}_{i}$, $\gamma_{i}$, $\varphi_{i}$ (recall Lemma~\ref{lem:behvarphi}, \eqref{Eigvarphi}) and all their derivatives, allow for a smooth  extension of $f_{i}$ up  to $t=T$ and then by the short-time existence result even beyond. In view of this contradiction, the flow must exist globally.

\bigskip

\noindent \textbf{Fifth Step}: \underline{Sub-convergence} The statement follows from an adaptation of the arguments depicted in \cite[\S~5 (Step nine)]{DP} to the present case.
\smallskip

\end{proof}

\begin{rem}\label{rem:lambda0}
The long-time existence result can be extended to the case $\lambda_i\geq 0$ with just few modifications. Indeed, in order to derive the bounds on the curvatures and on the tangential components one needs only bounds on the lengths from below. The fact that $\lambda_i>0$ has been used in the proof above to derive a bound from above on the length. In the case $\lambda_i\geq 0$ since by \eqref{a} (see also Remark \ref{rem3.3}) the lengths grow at most linearly one has also a bound from above on the length in finite time. This is sufficient to conclude the argument by contradiction. See \cite[(5.14)]{DP} for a similar argument. The presence of the tangential component does not create any difficulty.    
\end{rem}

\renewcommand{\thesection}{}

\appendix\renewcommand{\thesection}{\Alph{section}}
\setcounter{equation}{0}
\renewcommand{\theequation}{\Alph{section}\arabic{equation}}

\section{Supporting lemmas}
\label{techproofs}

\begin{lemma}\label{Younggeneral}
Let $a_{1}, \ldots, a_{n}$ be positive numbers and assume $0<i_{1}+ \ldots +i_{n} <2$, where $i_{j}>0 $ for all $j=1, \ldots, n$. Then for any $\epsilon>0$ we have
$$ a_{1}^{i_{1}} \cdot a_{2}^{i_{2}} \cdot \ldots \cdot a_{n}^{i_{n}} \leq \epsilon (a_{1}^{2} + \ldots a_{n}^{2}) + C_{\epsilon}. $$
\end{lemma}
\begin{proof}The proof goes by induction. The case $n=1$ is simply the standard Young inequality. Now suppose the claim holds for $n-1$. Then since by hypothesis we have $2(i_{2}+ \ldots +i_{n})/(2-i_{1}) <2$ we infer applying Young and the induction hypothesis that
\begin{align*}
a_{1}^{i_{1}} \cdot a_{2}^{i_{2}} \cdot \ldots \cdot a_{n}^{i_{n}} \leq \epsilon a_{1}^{2} + C_{\epsilon} (a_{2}^{i_{2}} \cdot \ldots \cdot a_{n}^{i_{n}})^{2/(2-i_{1})} \leq \epsilon a_{1}^{2} + C_{\epsilon} \delta ( a_{2}^{2} + \ldots + a_{n}^{n}) + C_{\epsilon}C_{\delta}
\end{align*}
for any $\delta >0$. Choosing $\delta < \epsilon/ C_{\epsilon}$ the claim follows.
\end{proof}

\begin{lemma}\label{PQgeneral}
Let $a_{1}, \ldots, a_{n}$ be positive numbers such that $a_{i} \geq 1$ for any $i=1, \ldots, n$ and assume $0<i_{1}+ \ldots +i_{n} <\gamma$, where $i_{j}>0 $ for all $j=1, \ldots, n$ and some $\gamma>0$. Then  we have
$$ a_{1}^{i_{1}} \cdot a_{2}^{i_{2}} \cdot \ldots \cdot a_{n}^{i_{n}} \leq  \sum_{i=1}^{n} a_{i}^{\gamma}. $$
\end{lemma}
\begin{proof}The proof follows by an induction argument over $n$ and uses the $(p,q)$-Young inequality.
Induction start: for $n=1$ the claim follows since $a_{1} \geq 1$. Suppose the claim holds for some $n-1$. Then Young inequality gives (note that $\frac{\gamma}{i_{1}} >1$)
\begin{align*}
a_{1}^{i_{1}} \cdot a_{2}^{i_{2}} \cdot \ldots \cdot a_{n}^{i_{n}} \leq \frac{i_{1}}{\gamma} a_{1}^{\gamma} + \frac{\gamma-i_{1}}{\gamma}(a_{2}^{i_{2}} \cdot \ldots \cdot a_{n}^{i_{n}})^{\frac{\gamma}{\gamma-i_{1}}} \leq a_{1}^{\gamma}+ (a_{2}^{i_{2}} \cdot \ldots \cdot a_{n}^{i_{n}})^{\frac{\gamma}{\gamma-i_{1}}}  .
\end{align*}
On the second term the induction hypothesis can be applied since $\sum_{j=2}^{n}\frac{i_{j} \gamma}{\gamma-i_{1}}< \gamma$.
\end{proof}

\begin{lemma}\label{lem:behvarphi}
Let $m \in \N$ and suppose that \eqref{IndHyp2} and \eqref{IndHyp3} hold. Let the tangential components be defined as in \eqref{defEigvarphi}. Then we have that for any $r \in \N$, $r \leq m$
\begin{align*}
|\partial_{t}^{r}\varphi_{i} (t,x)| & \leq |\partial^{r}_{t}\varphi_{i} (t,0)| +C,
\qquad \qquad |\frac{d^{r}}{dt^{r}} \mathcal{L}(f_{i})| \leq C,\\
|\partial_{t}^{m+1}\varphi_{i} (t,x)| & \leq |\partial^{m+1}_{t}\varphi_{i} (t,0)| +C (|\partial^{m}_{t}\varphi_{i} (t,0)| +1),\\
|\partial_{s} \partial_{t}^{r} \varphi_{i}(t,x)| & \leq C (|\partial_{t}^{r} \varphi_{i}(t,0)|+1)\mbox{ and}\\
|\partial_{s}^{j} \partial_{t}^{r} \varphi_{i}(t,x)|& \leq C \mbox{ for }j=2,3,4. 
\end{align*}
for any $(t,x) \in (0,T) \times [0,1]$.
Here the constant $C$ has the same dependencies as the constants appearing in \eqref{IndHyp2} and \eqref{IndHyp3}.
\end{lemma}
\begin{proof} 
Recalling the definition \eqref{defEigvarphi} of $\varphi_{i}$, \eqref{Eigvarphi}, \eqref{reltangpart0}, and \eqref{dert1varphi} we can write 
\begin{align}
\label{W1}
\partial_{s} \partial_{t}\varphi_{i} (t,x)
&  = -\frac{\partial_{t}\varphi_{i}(t,0)}{\mathcal{L}(f_{i})}
+\frac{\varphi_{i}(t,0)}{\mathcal{L}(f_{i})^{2}} (-\varphi_{i}(t,0)-\int_{I}\langle\vec{\kappa}_{i}, \vec{V}_{i} \rangle ds)
\\
& \quad -\frac{\varphi_{i}(t,0)}{\mathcal{L}(f_{i})}  (\partial_{s} \varphi_{i} - \langle\vec{\kappa}_{i}, \vec{V}_{i} \rangle),\notag \\
\label{W2}
\partial^{m}_{s}\partial_{t}\varphi_{i} (t,x)&= \frac{\varphi_{i}(t,0)}{\mathcal{L}(f_{i})}\partial_{s}^{m-1}(\langle\vec{\kappa}_{i}, \vec{V}_{i} \rangle) \\
&= \frac{\varphi_{i}(t,0)}{\mathcal{L}(f_{i})} \sum_{\substack{[[a,b]]\leq [[m+1,2]]\\c\leq m+1\\b \, \, even}}P^{a,c}_{b} (\vec{\kappa}_{i}) \text{ for any } m \geq 2. \notag
\end{align}
To get some induction argument going let us write
\begin{align*}
\varphi_{i}(x,t)= \varphi_{i}(t,0) B(t,x), 
\end{align*}
 with 
\begin{align*}
B(t,x):=1 - \frac{1}{\mathcal{L}(f_{i})} \int_{0}^{x} |(f_{i})_{x}| dx =\frac{1}{\mathcal{L}(f_{i})} \int_{x}^{1} |(f_{i})_{x}| dx 
\end{align*}
Obviously $0 \leq B(t,x) \leq 1$.
For $m, k \in \N$ we have
\begin{align*}
\partial_{t}^{m}\varphi_{i}(x,t)&= \sum_{r=0}^{m} \binom{m}{r}  \partial_{t}^{m-r}\varphi_{i}(t,0)  \partial_{t}^{r}B(t,x) \\
&= \partial_{t}^{m}\varphi_{i}(t,0) B(t,x) +  \sum_{r=1}^{m} \binom{m}{r}  \partial_{t}^{m-r}\varphi_{i}(t,0)  \partial_{t}^{r}B(t,x),\\
\partial_{s}^{k}\partial_{t}^{m}\varphi_{i}(x,t)&= \sum_{r=0}^{m} \binom{m}{r}  \partial_{t}^{m-r}\varphi_{i}(t,0) \partial_{s}^{k} \partial_{t}^{r}B(t,x) \\&=\partial_{t}^{m}\varphi_{i}(t,0) \partial_{s}^{k} B(t,x) +  \sum_{r=1}^{m} \binom{m}{r}  \partial_{t}^{m-r}\varphi_{i}(t,0)  \partial_{s}^{k}\partial_{t}^{r}B(t,x).
\end{align*}
The term $B$ satisfies
\begin{align}
\mathcal{L}(f_{i}) B(t,x) &= \int_{x}^{1} |(f_{i})_{x}| dx, \notag \\
\mathcal{L}(f_{i}) \partial_{s}B(t,x) &= -1, \qquad \partial_{s}^{k}B(t,x) =0 \text{ for any } k \geq 2,  \notag \\
\mathcal{L}(f_{i}) \partial_{t} B(t,x) + B(t,x) \frac{d}{dt} \mathcal{L}(f_{i}) & = \int_{x}^{1} (\partial_{s} \varphi_{i} - \langle\vec{\kappa}_{i}, \vec{V}_{i} \rangle) ds = -\varphi_{i}(t,x) -\int_{x}^{1} \langle\vec{\kappa}_{i}, \vec{V}_{i} \rangle ds,  \notag \\
& =-\varphi_{i}(t,x) +\int_{x}^{1} \sum_{\substack{[[a,b]]\leq [[2,2]]\\c\leq 2\\b \, \, even}}P^{a,c}_{b} (\vec{\kappa}_{i})  ds,  \notag \\
\mathcal{L}(f_{i}) \partial_{s}\partial_{t} B(t,x) + \partial_{s}B(t,x) \frac{d}{dt} \mathcal{L}(f_{i}) &  = - \partial_{s} \varphi_{i}(t,x) + \langle\vec{\kappa}_{i}, \vec{V}_{i} \rangle , \notag 
\end{align}
so that 
\begin{align}
\label{W2bis}
\mathcal{L}(f_{i}) \partial_{s}^{m}\partial_{t} B(t,x)  &  =  \partial_{s}^{m-1}  \langle\vec{\kappa}_{i}, \vec{V}_{i} \rangle  \text{ for } m \geq 2,   \\
\mathcal{L}(f_{i}) \partial_{t}^{m} B(t,x)&= -  \sum_{r=0}^{m-1} \binom{m}{r} \left( \frac{d}{d t}^{m-r}\mathcal{L}(f_{i})  \right) \partial_{t}^{r} B(t,x)  -  \partial_{t}^{m-1}\varphi_{i}(t,x)  \notag  \\
& \quad + \partial_{t}^{m-1} \Big(\int_{x}^{1} \sum_{\substack{[[a,b]]\leq [[2,2]]\\c\leq 2\\b \, \, even}}P^{a,c}_{b} (\vec{\kappa}_{i})  ds \Big),  \notag \\
\mathcal{L}(f_{i})\partial_{s}^{k} \partial_{t}^{m} B(t,x)&= -  \sum_{r=0}^{m-1} \binom{m}{r} \left( \frac{d}{d t}^{m-r}\mathcal{L}(f_{i})  \right) \partial_{s}^{k} \partial_{t}^{r} B(t,x)  - \partial_{s}^{k} \partial_{t}^{m-1}\varphi_{i}(t,x)   \notag \\
& \quad + \partial_{s}^{k}\partial_{t}^{m-1} (\int_{x}^{1} \langle\vec{\kappa}_{i}, \vec{V}_{i} \rangle ds).  \notag 
\end{align}
Now from \eqref{reltangpart0} we know that
\begin{align*}
\frac{d}{dt} \mathcal{L}(f_{i}) = - \varphi_{i}(0,t) + \int_I \sum_{\substack{[[a,b]]\leq [[2,2]]\\c\leq 2\\b \, \, even}}P^{a,c}_{b} (\vec{\kappa}_{i})  ds 
\end{align*}
so that by \eqref{E3even}, \eqref{a}, and an induction argument that uses repeatedly (for $\psi$  a scalar quantity)
\begin{align*}
\frac{d}{dt}\left( \int_{I} \psi ds \right) = \int_{I}\left( \psi_{t}  -\psi \langle \vec{\kappa}, \vec{V} \rangle  \right)ds +\varphi_{s} \int_{I} \psi ds
\end{align*}
 we infer for $m \in \N$, 
 \begin{align*}
\frac{d^{m+1}}{dt^{m+1}} \mathcal{L}(f_{i}) &= - \partial_{t}^{m}\varphi_{i}(0,t)  \\
& \quad  + \tilde{p}\Big(\varphi_{i}(0), .., \partial_{t}^{m-1}\varphi_{i}(0),\frac{1}{\mathcal{L}(f_{i})}, \frac{d}{dt}\mathcal{L}(f_{i}),.., \frac{d^{m-1}}{dt^{m-1}}\mathcal{L}(f_{i}),r_{m}(t) \Big),
\end{align*}
with
\begin{align*}
r_{m}(t)&:=\int_I \sum_{\substack{[[a,b]]\leq [[2+4m,2]]\\c\leq 2+4m\\b \, \, even}}\hspace{-.5cm}P^{a,c}_{b} (\vec{\kappa}_{i})  ds +
\int_I  \sum_{\beta \in S^{m-1}_{4m}} \prod_{l=0}^{m-1}
(\varphi_{i}^{(l)})^{\beta_{l}} \sum_{\substack{[[a,b]]\leq [[2+4m-|\beta|,2]]\\c\leq 2+4m -|\beta|\\b \, \, even}}\hspace{-.5cm}P^{a,c}_{b} (\vec{\kappa}_{i})  ds
\end{align*}
and where $\tilde{p}$ is a polynomial in the listed variables.
Similarly
\begin{align*}
\partial_{t}^{m} \int_{x}^{1} \hspace{-.25cm}\sum_{\substack{[[a,b]]\leq [[2,2]]\\c\leq 2\\b \, \, even}}\hspace{-.5cm}P^{a,c}_{b} (\vec{\kappa}_{i})  ds 
 = 
\tilde{p}\Big(\varphi_{i}(0), .., \partial_{t}^{m-1}\varphi_{i}(0),\frac{1}{\mathcal{L}(f_{i})}, \frac{d}{dt}\mathcal{L}(f_{i}),.., \frac{d^{m-1}}{dt^{m-1}}\mathcal{L}(f_{i}),r_m^{x}(t) \Big)
\end{align*}
where
\begin{align*}
r_m^{x}(t) :=\int_x^{1} \sum_{\substack{[[a,b]]\leq [[2+4m,2]]\\c\leq 2+4m\\b \, \, even}}\hspace{-.5cm}P^{a,c}_{b} (\vec{\kappa}_{i})  ds  
+  \int_x^{1}  \sum_{\beta \in S^{m-1}_{4m}} \prod_{l=0}^{m-1}
(\varphi_{i}^{(l)})^{\beta_{l}} \sum_{\substack{[[a,b]]\leq [[2+4m-|\beta|,2]]\\c\leq 2+4m -|\beta|\\b \, \, even}}\hspace{-.5cm}P^{a,c}_{b} (\vec{\kappa}_{i})  ds \,.
\end{align*}
Finally observe that for $m \in \N \cup \{ 0 \}$
\begin{align*}
& \partial_{s}\partial_{t}^{m} \Big(\int_{x}^{1} \sum_{\substack{[[a,b]]\leq [[2,2]]\\c\leq 2\\b \, \, even}}P^{a,c}_{b} (\vec{\kappa}_{i})  ds \Big) \\
& \qquad \qquad 
 = \tilde{p}\Big(\varphi_{i}(0), .., \partial_{t}^{m-1}\varphi_{i}(0),\frac{1}{\mathcal{L}(f_{i})}, \frac{d}{dt}\mathcal{L}(f_{i}),.., \frac{d^{m-1}}{dt^{m-1}}\mathcal{L}(f_{i}),R_m(t,x) \Big)
\end{align*}
where
\begin{align*}
R_m(t,x) := \sum_{\substack{[[a,b]]\leq [[2+4m,2]]\\c\leq 2+4m\\b \, \, even}}\hspace{-.5cm}P^{a,c}_{b} (\vec{\kappa}_{i})  
+   \sum_{\beta \in S^{m-1}_{4m}} \prod_{l=0}^{m-1}
(\varphi_{i}^{(l)})^{\beta_{l}} \sum_{\substack{[[a,b]]\leq [[2+4m-|\beta|,2]]\\c\leq 2+4m -|\beta|\\b \, \, even}}\hspace{-.5cm}P^{a,c}_{b} (\vec{\kappa}_{i})  ,
\end{align*}
and more generaly for $k \in \N$
\begin{align*}
& \partial_{s}^{k}\partial_{t}^{m} \Big(\int_{x}^{1} \sum_{\substack{[[a,b]]\leq [[2,2]]\\c\leq 2\\b \, \, even}}P^{a,c}_{b} (\vec{\kappa}_{i})  ds \Big) \\
& \qquad \qquad = \tilde{p}\Big( \varphi_{i}(0), .., \partial_{t}^{m-1}\varphi_{i}(0),\frac{1}{\mathcal{L}(f_{i})}, \frac{d}{dt}\mathcal{L}(f_{i}),.., \frac{d^{m-1}}{dt^{m-1}}\mathcal{L}(f_{i}),\partial_{s}^{k-1} R_m(t,x) \Big)
\end{align*}
with
\begin{align*}
\partial_{s}^{k-1} R_m(t,x)& := \tilde{p} (\sum_{\substack{[[a,b]]\leq [[2+4m +k-1,2]]\\c\leq 2+4m+k-1\\b \, \, even}}P^{a,c}_{b} (\vec{\kappa}_{i}) ,\\
& \quad \varphi_{i}(t,x), \partial_{s} \varphi_{i}(t,x), .., \partial_{s}^{k-1} \varphi_{i}(t,x), ..,   \varphi_{i}^{(m-1)}(t,x), .., \partial_{s}^{k-1}\varphi_{i}^{(m-1)}(t,x) ) .
\end{align*}
The first three claims follows now by an induction argument. More precisely: for $m=1$ starting from
\begin{align*}
|B(t,x)| \leq C, \quad  |\varphi_{i}(t,x)| \leq |\varphi_{i}(t,0)| \leq C
\end{align*} 
 and using the expression for the derivatives of $\mathcal{L}(f_{i})$, $B$, and $\varphi_{i}(x,t)$ we first derive 
\begin{align*}
|\frac{d}{dt} \mathcal{L}(f_{i})| \leq C, \quad  |\partial_{t}B(t,x)| \leq C, \quad |\partial_{t}\varphi_{i}(x,t)|  \leq |\partial_{t}\varphi_{i}(t,0)| + C
\end{align*}
by \eqref{IndHyp3}, \eqref{IndHyp2}, Lemma~\ref{lemineqshsum} (with $\ell=4m$), and the uniform control of the length from above and below.
Then assuming that 
\begin{align}\label{genbounds}
|\frac{d^{r}}{dt^{r}} \mathcal{L}(f_{i})| \leq C, \quad |\partial_{t}^{r}B(t,x)| \leq C, \quad |\partial_{t}^{r}\varphi_{i}(x,t)|  \leq |\partial_{t}^{r}\varphi_{i}(t,0)| + C
\end{align}
holds for some $1 \leq r <m$, we derive again by \eqref{IndHyp3}, \eqref{IndHyp2}
and using Lemma~\ref{lemineqshsum} (with $\ell=4m$) that
\begin{align*}
|\frac{d^{r+1}}{dt^{r+1}} \mathcal{L}(f_{i})| \leq C, \quad  |\partial_{t}^{r+1}B(t,x)| \leq C, \quad |\partial_{t}^{r+1}\varphi_{i}(t,x)| \leq |\partial_{t}^{r+1}\varphi_{i}(t,0)| + C,
\end{align*}
respectively
\begin{align*}
|\frac{d^{m+1}}{dt^{m+1}} \mathcal{L}(f_{i})|& \leq |\partial_{t}^{m}\varphi_{i}(t,0)|+ C, \quad  |\partial_{t}^{m+1}B(t,x)| \leq C (|\partial_{t}^{m}\varphi_{i}(t,0)|+ 1), \\
|\partial_{t}^{m+1}\varphi_{i}(t,x)| &\leq |\partial_{t}^{m+1}\varphi_{i}(t,0)| + C(|\partial_{t}^{m}\varphi_{i}(t,0)|+ 1).
\end{align*}
Next, observe that from \eqref{W1}, we infer (using also the bounds derived above for the derivative of the length functional)
\begin{align*}
|\partial_{s} \partial_{t} \varphi_{i}(t,x)| &\leq C (|\partial_{t} \varphi_{i}(t,0)| + | \sum_{\substack{[[a,b]]\leq [[2,2]]\\c\leq 2\\b \, \, even}}P^{a,c}_{b} (\vec{\kappa}_{i})| ) \leq C(|\partial_{t} \varphi_{i}(t,0)| +C),\\
|\partial_{s} \partial_{t} B(t,x)| &\leq C,
\end{align*}
where we have used Lemma~\ref{lem:trickbdry}, Lemma~\ref{lemineqshsum}, and \eqref{IndHyp2}.
Next we infer using \eqref{genbounds} and \eqref{IndHyp3}
 that
\begin{align*}
&|\partial_{s} \partial_{t}^{2} B(t,x) |\leq C (|\partial_{t}\varphi_{i}(t,0)| +1) \leq C,\\
 & |\partial_{s} \partial_{t}^{2} \varphi_{i}(t,x) |\leq C (|\partial_{t}^{2}\varphi_{i}(t,0)|+|\partial_{t}\varphi_{i}(t,0)| +1) \leq C (|\partial_{t}^{2}\varphi_{i}(t,0)|+1).
\end{align*}
Repeating the same arguments inductively we obtain the fourth claim in the lemma.
Next, we observe from \eqref{W2}, \eqref{W2bis}, Lemma~\ref{lem:trickbdry}, Lemma~\ref{lemineqshsum}, and \eqref{IndHyp2} that
\begin{align*}
|\partial_{s}^{k} \partial_{t} \varphi_{i}(t,x)|=|\frac{\varphi_{i}(t,0)}{\mathcal{L}(f_{i})} \sum_{\substack{[[a,b]]\leq [[k+1,2]]\\c\leq k+1\\b \, \, even}}\hspace{-.3cm}P^{a,c}_{b} (\vec{\kappa}_{i}) | \leq C,  \qquad |\partial_{s}^{k} \partial_{t} B(t,x)| \leq C  \text{ for } k =2,3,4.
\end{align*}
The final three claims of the lemma are again proved using an induction arguments and employing all estimates achieved so far: it is important that one proves the claim first for $k=2$, then $k=3$ and finally $k=4$.
\end{proof}

\begin{lemma}\label{lem:beta}
Let $\varphi_{i}$ be the tangential component in \eqref{flownetwork1}, and $j, p \in\mathbb{N}_0$. Let $S_{p}^{j}$ be defined as in \eqref{DefBeta}. Then 
\begin{align*}
(i)& \quad \varphi_i \sum_{\substack{\beta \in S_{p}^{j}}} \prod_{l=0}^{j} ( \varphi_{i}^{(l)})^{\beta_{l}}  \sum_{\substack{[[a,b]]\leq [[k-|\beta|,B]]\\c\leq k-|\beta|\\b \, \, odd}}\hspace{-.6cm}P^{a,c}_{b} (\vec{\kappa}_{i}) 
= \sum_{\substack{\beta \in S_{p+3}^{j}}} \prod_{l=0}^{j} ( \varphi_{i}^{(l)})^{\beta_{l}}  \sum_{\substack{[[a,b]]\leq [[3+k-|\beta|,B]]\\c\leq 3+k-|\beta|\\b \, \, odd}}\hspace{-.6cm}P^{a,c}_{b} (\vec{\kappa}_{i}), \\
(ii)&\quad  \partial_t \sum_{\substack{\beta \in S_{p}^{j}}} \prod_{l=0}^{j} ( \varphi_{i}^{(l)})^{\beta_{l}} 
= \sum_{\substack{\beta \in S_{p+4}^{j+1}}} \prod_{l=0}^{j+1} c_{\beta_l}( \varphi_{i}^{(l)})^{\beta_{l}} \text{ for some constants }c_{\beta_l},  \\ 
(iii)& \quad \partial_t \Big(\sum_{\substack{\beta \in S_{p}^{j}}} \prod_{l=0}^{j} ( \varphi_{i}^{(l)})^{\beta_{l}} 
\hspace{-.2cm}\sum_{\substack{[[a,b]]\leq [[k-|\beta|,B]]\\c\leq k-|\beta|\\b \, \, even}}\hspace{-.6cm}P^{a,c}_{b} (\vec{\kappa}_{i}) \Big) 
= \sum_{\substack{\beta \in S_{p+4}^{j+1}}} \prod_{l=0}^{j+1} ( \varphi_{i}^{(l)})^{\beta_{l}} 
\hspace{-.2cm}\sum_{\substack{[[a,b]]\leq [[4+k-|\beta|,B]]\\c\leq 4+k-|\beta|\\b \, \, even}}\hspace{-.6cm}P^{a,c}_{b} (\vec{\kappa}_{i}) ,\\
(iv)& \quad \partial_t \Big(\sum_{\substack{\beta \in S_{p}^{j}}} \prod_{l=0}^{j} ( \varphi_{i}^{(l)})^{\beta_{l}} 
\hspace{-.2cm}\sum_{\substack{[[a,b]]\leq [[k-|\beta|,B]]\\c\leq k-|\beta|\\b \, \, odd}}\hspace{-.5cm}P^{a,c}_{b} (\vec{\kappa}_{i}) \Big) 
= \sum_{\substack{\beta \in S_{p+4}^{j+1}}} \prod_{l=0}^{j+1} ( \varphi_{i}^{(l)})^{\beta_{l}} 
\hspace{-.3cm}\sum_{\substack{[[a,b]]\leq [[4+k-|\beta|,B]]\\c\leq 4+k-|\beta|\\b \, \, odd}}\hspace{-.7cm}P^{a,c}_{b} (\vec{\kappa}_{i}) \\
&\qquad \qquad \qquad + \partial_{s} f \sum_{\substack{\beta \in S_{p}^{j}}} \prod_{l=0}^{j} ( \varphi_{i}^{(l)})^{\beta_{l}} 
\hspace{-.2cm}\sum_{\substack{[[a,b]]\leq [[3+k-|\beta|,B+1]]\\c\leq 3+k-|\beta|\\b \, \, even}}\hspace{-.5cm}P^{a,c}_{b} (\vec{\kappa}_{i}), 
\end{align*}
at the points where $\vec{\kappa}_i=0$, in particular at the boundary.
\end{lemma}
\begin{proof} 
Observe that 
$$\varphi_i \sum_{\substack{\beta \in S_{p}^{j}}} \prod_{l=0}^{j} ( \varphi_{i}^{(l)})^{\beta_{l}}  \sum_{\substack{[[a,b]]\leq [[k-|\beta|,B]]\\c\leq k-|\beta|\\b \, \, odd}}\hspace{-.5cm}P^{a,c}_{b} (\vec{\kappa}_{i}) 
= \sum_{\substack{\beta^{\prime} \in S_{p+3}^{j}\\\beta\in S_{p}^{j}}} \prod_{l=0}^{j} ( \varphi_{i}^{(l)})^{\beta_{l}^{\prime}}  \sum_{\substack{[[a,b]]\leq [[k-|\beta|,B]]\\c\leq k-|\beta|\\b \, \, odd}}\hspace{-.5cm}P^{a,c}_{b} (\vec{\kappa}_{i}).$$ 
Since $|\beta^{\prime}|=|\beta|+3$, 
the proof of (i) is obtained from replacing $\beta$ by $\beta^{\prime}$ 
in the term $\sum_{\substack{[[a,b]]\leq [[k-|\beta|,B]]\\c\leq k-|\beta|\\b \, \, odd}}P^{a,c}_{b} (\vec{\kappa}_{i})$. 
The proofs of (ii) is straightforward by using the defition of the notation $S_{p}^{j}$ in \eqref{DefBeta}. 
The proofs of (iii) and (iv) are also straightforward, by applying (i), (ii) above, and \eqref{E4sum-even} and \eqref{E4sum-odd}. 

\end{proof}

\section{Proofs of technical lemmas}
\label{AppB}

\subsection{Proof of Lemma \ref{lemLin}}\label{AppBLem3.6}

\begin{proof}[Proof of \eqref{E2sumsimple} in Lemma \ref{lemLin}]
Since $P^{\mu,d}_{\nu}$ is linear combination of terms of the type 
$$ \langle \nabla_{s}^{i_{1}}\vec{\kappa}, \nabla_{s}^{i_{2}}\vec{\kappa} \rangle \ldots 
\langle \nabla_{s}^{i_{\nu-2}} \vec{\kappa}, \nabla_{s}^{i_{\nu-1}}\vec{\kappa} \rangle \nabla_{s}^{i_{\nu}}\vec{\kappa} $$
with $i_1+ \dots +i_{\nu}=\mu$ and $\max \{ i_{j} \} \leq d$, by Leibnitz's rule we need to understand terms of the kind
$$ \langle \nabla_{s}^{i_{1}}\vec{\kappa}, \nabla_{s}^{i_{2}}\vec{\kappa} \rangle \ldots \langle \nabla_t \nabla_{s}^{i_{j}} \vec{\kappa}, ..  \rangle \dots 
\langle \nabla_{s}^{i_{\nu-2}} \vec{\kappa}, \nabla_{s}^{i_{\nu-1}}\vec{\kappa} \rangle \nabla_{s}^{i_{\nu}}\vec{\kappa} $$
for $j \in \{1, \dots, \nu-1\}$ or 
\begin{equation}\label{bel1}
 \langle \nabla_{s}^{i_{1}}\vec{\kappa}, \nabla_{s}^{i_{2}}\vec{\kappa} \rangle \dots 
\langle \nabla_{s}^{i_{\nu-2}} \vec{\kappa}, \nabla_{s}^{i_{\nu-1}}\vec{\kappa} \rangle \nabla_{t} \nabla_{s}^{i_{\nu}}\vec{\kappa}
\end{equation}
with as before $i_1+ \dots +i_{\nu}=\mu$ and $\max \{ i_{k} \} \leq d$. 
By Lemma~\ref{evolcurvature} we have
\begin{align*}
\nabla_t \nabla_{s}^{i_{j}} \vec{\kappa}  
& =   \varphi \nabla_{s}^{i_{j}+1} \vec{\kappa}
+\sum_{\substack{[[a,b]] \leq [[4+i_{j},1]]\\c\leq 4+i_{j}\\b \in [1,5]\, odd}} P^{a,c}_{b} (\vec{\kappa})  + \lambda \sum_{\substack{[[a,b]] \leq [[2+i_{j},1]]\\c\leq 2+i_{j}\\b\, \in [1,3]\, odd}} P^{a,c}_{b} (\vec{\kappa})  .
\end{align*}
It follows that 
\begin{align*}
&  \langle \nabla_{s}^{i_{1}}\vec{\kappa}, \nabla_{s}^{i_{2}}\vec{\kappa} \rangle \ldots \langle \nabla_t \nabla_{s}^{i_{j}} \vec{\kappa}, .. \rangle \dots 
\langle \nabla_{s}^{i_{\nu-2}} \vec{\kappa}, \nabla_{s}^{i_{\nu-1}}\vec{\kappa} \rangle \nabla_{s}^{i_{\nu}}\vec{\kappa} \\
& =  \sum_{\substack{[[a,b]] \leq [[\mu+4,\nu]]\\c\leq 4+d\\b \in [\nu,\nu+4], odd}} P^{a,c}_{b} (\vec{\kappa})  + \lambda
\sum_{\substack{[[a,b]] \leq [[\mu+2,\nu]]\\ c \leq d+2\\b \in [\nu,2+\nu], odd}} P^{a,c}_{b} (\vec{\kappa}) 
+  \varphi\, P_{\nu}^{\mu+1, \max \{ d, i_{j}+1\}} (\vec{\kappa}) 
\end{align*}
for any $ j \in \{1, \dots, \nu -1\}$ and the same formula holds for the term in \eqref{bel1}. We get 
\begin{align} 
\nabla_{t} P^{\mu,d}_{\nu} (\vec{\kappa}) &  = \sum_{\substack{[[a,b]] \leq [[\mu+4,\nu]]\\c\leq 4+d\\b \in [\nu,\nu+4], odd }} P^{a,c}_{b} (\vec{\kappa}) +\lambda \sum_{\substack{[[a,b]] \leq [[\mu+2,\nu]]\\c\leq 2+d\\b \in [\nu,2+\nu], odd }} P^{a,c}_{b} (\vec{\kappa})   + \varphi\, P_{\nu}^{\mu+1,d+1,  } (\vec{\kappa})   \label{riemcen1},
\end{align}
and the claim follows.
\end{proof}


\begin{proof}[Proof of \eqref{E2sum} in Lemma \ref{lemLin}]
Next we observe that  formula \eqref{riemcen1} implies that
\begin{align*}
& \nabla_{t} \sum_{\substack{[[a,b]]\leq [[A,B]]\\c\leq C\\b\in [N,M], odd}}P^{a,c}_{b} (\vec{\kappa}) \\
& = \sum_{\substack{[[a,b]] \leq [[A,B]]\\c\leq C\\b\in[N,M],odd}} \Big( \sum_{\substack{[[\alpha,\beta]] \leq [[a+4,b]]\\\gamma\leq 4+c \\\beta \in [b,b+4], odd }} P^{\alpha,\gamma}_{\beta} (\vec{\kappa}) +\lambda \sum_{\substack{[[\alpha,\beta]] \leq [[a+2,b]]\\\gamma\leq 2+c\\\beta \in [b,2+b], odd }} P^{\alpha,\gamma}_{\beta} (\vec{\kappa})  
\\ 
& \qquad +\varphi\, P_{b}^{a+1, c+1} (\vec{\kappa}) \Big)\\
& =  \sum_{\substack{[[a,b]] \leq [[A+4,B]]\\c\leq C+4\\b\in[N,M+4],odd}}  P^{a,c}_{b} (\vec{\kappa}) +\lambda  \sum_{\substack{[[a,b]] \leq [[A+2,B]]\\c\leq C+2\\b\in[N,M+2],odd}} P^{a,c}_{b} (\vec{\kappa}) 
 + \varphi \sum_{\substack{[[a,b]] \leq [[A+1,B]]\\c\leq C+1\\b\in[N,M],odd}} P_{b}^{a,  c} (\vec{\kappa}) \, .
\end{align*}
\end{proof}

\begin{proof}[Proof of \eqref{E3odd} and \eqref{E3even} in Lemma \ref{lemLin}]
The proof of \eqref{E3odd} follows immediately from observing that $\partial_{t}  P^{\mu,d}_{\nu}(\vec{\kappa})= \nabla_{t}P^{\mu,d}_{\nu}(\vec{\kappa}) + (\partial_{s}f) \langle P^{\mu,d}_{\nu}(\vec{\kappa}), \partial_{t} \tau \rangle$ and using \eqref{c} and \eqref{E2sumsimple}.
Equation \eqref{E3even} is derived with similar arguments employed for the proof of \eqref{E2sumsimple}.
\end{proof} 

\begin{proof}[Proof of \eqref{E4sum-odd} and \eqref{E4sum-even} in Lemma \ref{lemLin}]
The statements follow from \eqref{E3odd} and \eqref{E3even}. 
\end{proof}

\subsection{Proof of parts of Lemma \ref{lemtriple}}\label{AppBLemtriple}
\begin{proof}[Proof of \eqref{b4m} in Lemma \ref{lemtriple}] 
Based on \eqref{bm8}, the proof follows from an induction argument. 
Suppose that \eqref{b4m} holds for some $m\in\mathbb{N}$ bigger than or equal to $2$. 
Then we take the covariant derivatives, $\nabla_t$, of  \eqref{b4m}. 
The left-hand side is simply obtained from using Lemma \ref{evolcurvature}, 
\begin{align}
  \label{b4m-lhs}
 \nabla_t \nabla_s^{4m}  \vec{\kappa}_i  
 =
 -\nabla_s^{4+4m} \vec{\kappa}_i + \varphi_i \nabla_s^{1+4m} \vec{\kappa}_i +
 \sum_{\substack{[[a,b]]\leq [[2+4m,3]]\\c\leq 2+4m \\b \, \, odd}}P^{a,c}_{b} (\vec{\kappa}_{i})
 \text{.} 
\end{align} 
By applying \eqref{E2sum} and Lemma \ref{lem:beta}, the right-hand side is 
 \begin{align}
 \nonumber
& \nabla_t \sum_{\substack{[[a,b]]\leq [[4m-2,3]]\\c\leq 4m-2 \\b \, \, odd}} 
P^{a,c}_{b} (\vec{\kappa}_{i}) 
 + \nabla_t  \sum_{\substack{\beta \in S_{4m}^{m-1}}} \prod_{l=0}^{m-1} ( \varphi_{i}^{(l)})^{\beta_{l}} 
 \sum_{\substack{[[a,b]]\leq [[4m-|\beta|,1]]\\c\leq 4m-|\beta|\\b \, \, odd}}P^{a,c}_{b} (\vec{\kappa}_{i})  \\ 
& 
=\sum_{\substack{[[a,b]]\leq [[2+4m,3]]\\c\leq 2+4m \\b \, \, odd}} 
P^{a,c}_{b} (\vec{\kappa}_{i}) 
 + \varphi_i \sum_{\substack{[[a,b]]\leq [[4m-1,3]]\\c\leq 4m-1 \\b \, \, odd}}P^{a,c}_{b} (\vec{\kappa}_{i})  \nonumber \\ 
 &
\, \, \, \, \, \, \, \, + \sum_{\substack{\beta \in S_{4+4m}^{m}}} \prod_{l=0}^{m} ( \varphi_{i}^{(l)})^{\beta_{l}} \sum_{\substack{[[a,b]]\leq [[4m+4-|\beta|,1]]\\c\leq 4m+4-|\beta|\\b \, \, odd}}P^{a,c}_{b} (\vec{\kappa}_{i})  \nonumber \\ 
 &
 = \sum_{\substack{[[a,b]]\leq [[2+4m,3]]\\c\leq 2+4m \\b \, \, odd}} 
P^{a,c}_{b} (\vec{\kappa}_{i}) 
+ \sum_{\substack{\beta \in S_{4+4m}^{m}}} \prod_{l=0}^{m} ( \varphi_{i}^{(l)})^{\beta_{l}} \sum_{\substack{[[a,b]]\leq [[4+4m-|\beta|,1]]\\c\leq 4+4m-|\beta|\\b \, \, odd}}P^{a,c}_{b} (\vec{\kappa}_{i})   \label{b4m-rhs}  
  \text{.} 
\end{align} 
 From \eqref{b4m-lhs} and \eqref{b4m-rhs}, we finish the induction argument since \eqref{b4m} holds as $m$ therein is replaced by $m+1$.
 \end{proof}

\begin{proof}[Proof of \eqref{b5+4m} in Lemma \ref{lemtriple}]
Besides  \eqref{bm5} and \eqref{bm9}, the following also holds by similar arguments,
\begin{align}
  \label{bm13}
 \sum_{i=1}^3 \nabla_s^{13}  \vec{\kappa}_i  &  = \sum_{i=1}^3 \Big(
 \sum_{\substack{[[a,b]]\leq [[11,3]]\\c\leq 11\\b \, \, odd}}P^{a,c}_{b} (\vec{\kappa}_{i}) 
 +\varphi_{i} \sum_{\substack{[[a,b]]\leq [[10,1]]\\c\leq 10\\b \, \, odd}}P^{a,c}_{b} (\vec{\kappa}_{i})  
 +\varphi_{i}^{2} \sum_{\substack{[[a,b]]\leq [[7,1]]\\c\leq 7\\b \, \, odd}} \hspace{-.5cm} P^{a,c}_{b} (\vec{\kappa}_{i}) \nonumber \\  
&\qquad  \qquad +\varphi_{i}^{3} \sum_{\substack{[[a,b]]\leq [[4,1]]\\c\leq 4\\b \, \, odd}}P^{a,c}_{b} (\vec{\kappa}_{i}) 
 +\partial_t\varphi_{i} \sum_{\substack{[[a,b]]\leq [[6,1]]\\c\leq 6\\b \, \, odd}}P^{a,c}_{b} (\vec{\kappa}_{i}) \nonumber \\ 
&\qquad \qquad + \varphi_{i} \partial_t\varphi_{i} \sum_{\substack{[[a,b]]\leq [[3,1]]\\c\leq 3\\b \, \, odd}}P^{a,c}_{b} (\vec{\kappa}_{i}) 
+ \partial_t^{2}\varphi_{i} \sum_{\substack{[[a,b]]\leq [[2,1]]\\c\leq 2\\b \, \, odd}}P^{a,c}_{b} (\vec{\kappa}_{i})  \\
 & \qquad \qquad
 +  \partial_{s}f_{i} 
 \Big[    
 \sum_{\substack{[[a,b]]\leq [[12,2]]\\ c\leq 11 \\b \, \, even}} P^{a,c}_{b} (\vec{\kappa}_{i}) 
 +  \varphi_{i}\sum_{\substack{[[a,b]]\leq [[9,2]]\\c\leq 8 \\b \, \, even}}     P^{a, c}_{b} (\vec{\kappa}_{i}) \nonumber \\ 
&\qquad \qquad \qquad \qquad +  \varphi_{i}^2 \sum_{\substack{[[a,b]]\leq [[6,2]]\\c\leq 5 \\b \, \, even}}     P^{a, c}_{b} (\vec{\kappa}_{i})
+  \partial_t \varphi_{i}\sum_{\substack{[[a,b]]\leq [[5,2]]\\c\leq 4 \\b \, \, even}}     P^{a, c}_{b} (\vec{\kappa}_{i})
  \Big ]   
  \Big) 
  ,\nonumber
\end{align} 
 that is \eqref{b5+4m} for $m=2$. In these computations we use that the term multiplying $\varphi_i$ in the tangential component of \eqref{E4sum-odd} vanishes at the boundary since there $\vec{\kappa}_i=0$. 

The proof of \eqref{b5+4m} follows then from an induction argument. 
 Since \eqref{bm5}, \eqref{bm9}, and \eqref{bm13} 
 are the cases of $m=0,1,2$ in \eqref{b5+4m}, we suppose that \eqref{b5+4m} holds for some $m\in\mathbb{N}_0$. 
Then we take the partial differentiation, $\partial_t$, of \eqref{b5+4m}. By using Lemma \ref{evolcurvature}, \eqref{d}, and the fact that $\vec{\kappa}_{i}=0$
we obtain for the left-hand side, 
\begin{align}
  \label{bm-lhs}
 \sum_{i=1}^3 \partial_t \nabla_s^{5+4m}  \vec{\kappa}_i  &  = \sum_{i=1}^3 \Big( 
 -\nabla_s^{9+4m} \vec{\kappa}_i + \varphi_i \nabla_s^{6+4m} \vec{\kappa}_i +
 \sum_{\substack{[[a,b]]\leq [[7+4m,3]]\\c\leq 7+4m \\b \, \, odd}}P^{a,c}_{b} (\vec{\kappa}_{i}) \nonumber \\ 
 & \qquad \qquad 
 +  \partial_{s}f_{i} 
 \Big[    
 \sum_{\substack{[[a,b]]\leq [[8+4m,2]]\\ c\leq 5+4m \\b \, \, even}} P^{a,c}_{b} (\vec{\kappa}_{i}) 
  \Big ]   
  \Big). 
\end{align} 
By applying 
\eqref{E4sum-odd}, \eqref{E4sum-even}, \eqref{c}, Lemma \ref{lem:beta} and using the fact that $\vec{\kappa}_{i}=0$
 the right-hand-side is equal to 
 \begin{align}
  \label{bm-rhs}
 \sum_{i=1}^3 & \Big\{ \Big[
 \sum_{\substack{[[a,b]]\leq [[7+4m,3]]\\c\leq 7+4m \\b \, \, odd}}P^{a,c}_{b} (\vec{\kappa}_{i}) 
 + \varphi_i \sum_{\substack{[[a,b]]\leq [[4+4m,3]]\\c\leq 4+4m \\b \, \, odd}}P^{a,c}_{b} (\vec{\kappa}_{i}) \Big] \\ 
 & 
 +\partial_{s}f_{i} 
 \Big[ \, 
  \sum_{\substack{[[a,b]]\leq [[6+4m,4]]\\ c\leq 3+4m \\b \, \, even}} P^{a,c}_{b} (\vec{\kappa}_{i}) 
 \Big] 
  \nonumber \\
 &
 +  \sum_{\beta \in S_{8+4m}^{m+1} } \prod_{l=0}^{m+1} ( \varphi_{i}^{(l)})^{\beta_{l}} \sum_{\substack{[[a,b]]\leq [[5+4(m+1)-|\beta|,1]]\\c\leq 5+4(m+1)-|\beta|\\b \, \, odd}}P^{a,c}_{b} (\vec{\kappa}_{i})  
 \nonumber \\ 
& + \partial_s f_i \Big[ \sum_{\beta \in S_{4+4m}^{m} } \prod_{l=0}^{m} ( \varphi_{i}^{(l)})^{\beta_{l}} \sum_{\substack{[[a,b]]\leq [[4+4(m+1)-|\beta|,2]]\\c\leq 4+4(m+1)-|\beta|\\b \, \, even}}P^{a,c}_{b} (\vec{\kappa}_{i}) \Big]\nonumber \\
 & +\partial_{s}f_{i} \Big[ 
 \sum_{\substack{[[a,b]]\leq [[8+4m,2]]\\c\leq7+4m\\b \, \, even}} P^{a,c}_{b} (\vec{\kappa}_{i}) 
 +  \varphi_i  \sum_{\substack{[[a,b]]\leq [[5+4m,2]]\\c\leq 4+4m \\b \, \, even}}P^{a,c}_{b} (\vec{\kappa}_{i}) 
  \Big ]   \nonumber \\ 
& +\partial_{s}f_{i} \Big[ 
\sum_{\beta \in S_{4+4m}^{m} } \prod_{l=0}^{m} ( \varphi_{i}^{(l)})^{\beta_{l}}   \sum_{\substack{[[a,b]]\leq [[4+4(m+1)-|\beta|,2]]\\c\leq 3+4(m+1)-|\beta|\\b \, \, even}}P^{a,c}_{b} (\vec{\kappa}_{i}) \Big] \nonumber \\ 
 &+  \sum_{\substack{[[a,b]]\leq [[7+4m,3]]\\c\leq3+4m\\b \, \, even}} P^{a,c}_{b} (\vec{\kappa}_{i}) 
 +  \sum_{\substack{\beta \in S_{4m}^{m-1}}} \prod_{l=0}^{m-1} ( \varphi_{i}^{(l)})^{\beta_{l}} \sum_{\substack{[[a,b]]\leq [[7+4m-|\beta|,3]]\\c\leq 3+4m-|\beta|\\b \, \, even}}P^{a,c}_{b} (\vec{\kappa}_{i}) \Big\} \nonumber  \\
&=  
\sum_{i=1}^3 \Big(
 \sum_{\substack{[[a,b]]\leq [[7+4m,3]]\\c\leq 7+4m\\b \, \, odd}}P^{a,c}_{b} (\vec{\kappa}_{i}) 
 +\sum_{\beta \in S_{8+4m}^{m+1}} \prod_{l=0}^{m+1} ( \varphi_{i}^{(l)})^{\beta_{l}}\sum_{\substack{[[a,b]]\leq [[9+4m-|\beta|,1]]\\c\leq 9+4m-|\beta|\\b \, \, odd}}P^{a,c}_{b} (\vec{\kappa}_{i}) 
\nonumber  \\
 & \quad +  \partial_{s}f_{i} \Big[    \sum_{\substack{[[a,b]]\leq [[8+4m,2]]\\c\leq 7+4m\\b \, \, even}} 
 P^{a,c}_{b} (\vec{\kappa}_{i}) 
 +  \sum_{\beta \in S_{4+4m}^{m}} \prod_{l=0}^{m} ( \varphi_{i}^{(l)})^{\beta_{l}}\sum_{\substack{[[a,b]]\leq [[8+4m-|\beta|,2]]\\c\leq 7+4m-|\beta|\\b \, \, even}}\hspace{-.5cm}P^{a,c}_{b} (\vec{\kappa}_{i}) 
  \Big ]    \Big) 
 \nonumber 
  \text{.} 
\end{align} 
The proof follows from using \eqref{bm-lhs} and \eqref{bm-rhs}.
\end{proof}

\section{Compatibility conditions}
\label{sec:cc}
 Besides \eqref{initdatum1} and \eqref{initdatum2},
the initial network $\Gamma_{0}$  needs to satify a set of compatibility conditions. These are required to ensure that the solution of the parabolic problem is smooth up to the initial time $t=0$.  
They are given as follows. Let $L_{i}$ denote the quasi linear differential operator of fourth order  such that $\partial_{t} f_{i}=L_{i} f$ as in \eqref{flownetwork1}, $i=1,2, 3$. Similarly for $j \in \N$ let $L_{i}^{(j)}$ denote the quasilinear differential operator of oder $4j$ such that $$ \partial_{t}^{j}f_{i} = L_{i}^{(j)}.$$
Note that \eqref{d} and more generally Lemma~\ref{lemform} plays a role in writing down precisely the operators.
The first set of compatibility conditions requires that
\begin{align*}
&L_{i}^{(j)}f_{i,0}=0  \qquad \text{ at } x=1, \text{  and for all } j \in \N, \text{ and } i=1,2,3,\\
&L_{1}^{(j)}f_{1,0}=L_{2}^{(j)}f_{2,0}=L_{3}^{(j)}f_{3,0}   \qquad \text{ at } x=0, \text{  and for all } j \in \N.
\end{align*}
Next,  let $Q^{0}_{i}$ denote the following quasilinear second-order operators:
$$ Q_{i}^{(0)}f_{i} = \vec{\kappa}_{i}, \qquad \text{ and let } Q_{i}^{(j)}= \partial_{t}^{j} Q^{(0)}_{i}.$$
The next set of compatibility conditions reads then
$$ Q_{i}^{(j)}f_{i,0} =0 \qquad  \text{ at } x \in \{0,1\} \text{ for all } j \in \N_0, \text{ and } i=1,2,3.$$ 
Finally let $W^{(0)} $ denote the third order operator
$$ W^{(0)}(f_{1},f_{2},f_{3})=\sum_{i=1}^{3} (\nabla_s \vec{\kappa}_{i} - \lambda_i \partial_s f_{i}),$$
and $W^{(j)}= \partial_{t}^{j} W^{(0)}$. The last set of compatibility conditions reads:
$$ W^{(j)}(f_{1,0},f_{2,0},f_{3,0}) =0 \quad \text{ at } x=0.$$
It is important to note that, concerning the maps $\varphi_{i}$, we have that  only $\partial_{t}^{j}\varphi_{i} (t, 0)$ plays a role in the expression for the operators $W^{(j)}, Q^{(j)}_{i}$, and $L^{(j)}_{i}$ at the junction point. Moreover since the points $P_{i}$ are fixed, it must me $\varphi_{i}(t,1)=0$ for all times, therefore $\varphi_{i}$ does not contribute at all in the  expression for $L_{i}^{(j)}$ and $Q_{i}^{(j)}$ at $x=1$.




\bibliography{ref2}
\bibliographystyle{acm}

\end{document}